\documentclass[12pt,leqno]{amsart}

\usepackage{enumerate}
\usepackage{graphicx}
\usepackage{epsfig}
\usepackage{amssymb,amsmath,amsthm,amsfonts}

\usepackage {latexsym}

\usepackage{bbm}

\allowdisplaybreaks

\newcommand{\nc}{\newcommand}
\nc{\les}{\lesssim}
\nc{\nit}{\noindent}
\nc{\nn}{\nonumber}
\nc{\D}{\partial}
\nc{\diff}[2]{\frac{d #1}{d #2}}
\nc{\diffn}[3]{\frac{d^{#3} #1}{d {#2}^{#3}}}
\nc{\pdiff}[2]{\frac{\partial #1}{\partial #2}}
\nc{\pdiffn}[3]{\frac{\partial^{#3} #1}{\partial{#2}^{#3}}}
\nc{\abs}[1] {\lvert #1 \rvert}
\nc{\cAc}{{\cal A}_c}
\nc{\cE}{{\cal E}}
\nc{\cF}{{\cal F}}
\nc{\cP}{{\cal P}}
\nc{\cV}{{\cal V}}
\nc{\cQ}{{\cal Q}}
\nc{\cGin}{{\cal G}_{\rm in}}
\nc{\cGout}{{\cal G}_{\rm out}}
\nc{\cO}{{\cal O}}
\nc{\Lav}{{\cal L}_{\rm av}}
\nc{\cL}{{\cal L}}
\nc{\cB}{{\cal B}}
\nc{\cZ}{{\cal Z}}
\nc{\cR}{{\cal R}}
\nc{\cT}{{\cal T}}
\nc{\cY}{{\cal Y}}
\nc{\cX}{{\cal X}}
\nc{\cXT}{{{\cal X}(T)}}
\nc{\cBT}{{{\cal B}(T)}}
\nc{\vD}{{\vec \mathcal{D}}}
\nc{\efield}{\mathcal{E}}
\nc{\vE}{{\vec \efield}}
\nc{\vB}{{\vec \mathcal{B}}}
\nc{\vH}{{\vec \mathcal{H}}}
\nc{\F}{  \mathcal{F} }
\nc{\ty}{{\tilde y}}
\nc{\tu}{{\tilde u}}
\nc{\tV}{{\tilde V}}
\nc{\Pc}{{\bf P_c}}
\nc{\bx}{{\bf x}}
\nc{\bX}{{\bf X}}
\nc{\bXYZ}{{\bf XYZ}}
\nc{\bY}{{\bf Y}}
\nc{\bF}{{\bf F}}
\nc{\bS}{{\bf S}}
\nc{\dV}{{\delta V}}
\nc{\dE}{{\delta E}}
\nc{\TT}{{\Theta}}
\nc{\dPsi}{{\delta\Psi}}
\nc{\order}{{\cal O}}
\nc{\Rout}{R_{\rm out}}
\nc{\eplus}{e_+}
\nc{\eminus}{e_-}
\nc{\epm}{e_\pm}
\nc{\eps}{\varepsilon}
\nc{\vnabla}{{\vec\nabla}}
\nc{\G}{\Gamma}
\nc{\w}{\omega}
\nc{\mh}{h}
\nc{\mg}{g}
\nc{\vphi}{\varphi}
\nc{\tlambda}{\tilde\lambda}
\nc{\be}{\begin{equation}}
\nc{\ee}{\end{equation}}
\nc{\ba}{\begin{eqnarray}}
\nc{\ea}{\end{eqnarray}}

\nc{\g}{\gamma}
\nc{\ol}{\overline}

\newtheorem{theorem}{Theorem}[section]
\newtheorem{lemma}[theorem]{Lemma}
\newtheorem{prop}[theorem]{Proposition}
\newtheorem{corollary}[theorem]{Corollary}
\newtheorem{defin}[theorem]{Definition}
\newtheorem{rmk}[theorem]{Remark}

\nc{\pT}{\partial_T}
\nc{\pz}{\partial_z}
\nc{\pt}{\partial_t}
\nc{\la}{\langle}
\nc{\ra}{\rangle}
\nc{\infint}{\int_{-\infty}^{\infty}}
\nc{\halfwidth}{6.5cm}
\nc{\figwidth}{10cm}
\newcommand{\f}{\frac}

\nc{\nlayers}{L} \nc{\nsectors}{M}
\nc{\indicator}{\mathbf{1}}
\nc{\Rhole}{R_{\rm hole}}
\nc{\Rring}{R_{\rm ring}}
\nc{\neff}{n_{\rm eff}}
\nc{\Frem}{F_{\rm rem}}
\nc{\R}{\mathbb R}
\nc{\C}{\mathbb C}
\nc{\Z}{\mathbb Z}
\nc{\DD}{\Delta}
\nc{\cD}{\mathcal D}
\nc{\lnorm}{\left\|}
\nc{\rnorm}{\right\|}
\nc{\rnormp}{\right\|_{\ell^{p,\eps}}}
\nc{\rar}{\rightarrow}
\nc{\mR}{\mathcal R}
\nc{\oo}{\"o}   

\nc{\os}{\overset{o}}
\begin{document}

\begin{abstract}

We study the fourth order Schr\"odinger operator $H=(-\Delta)^2+V$ for a decaying potential $V$ in four dimensions. In particular, we show that 
the $t^{-1}$ decay rate holds in the $L^1\to L^\infty$ setting if zero energy is regular.  Furthermore, if the threshold energies are regular then a faster decay rate of $t^{-1}(\log t)^{-2}$ is attained for large $t$, at the cost of logarithmic spatial weights. Zero is not regular for the free equation, hence the free evolution does not satisfy this bound due to the presence of a resonance at the zero energy.  We provide a full classification of the different types of zero energy resonances and study the effect of each type on the time decay in the dispersive bounds. 

\end{abstract}

\title[Dispersive estimates for Fourth order Schr\"odinger]{On the Fourth order Schr\"odinger equation in four dimensions: dispersive estimates and zero energy resonances}

\author[Green, Toprak]{ William~R. Green and Ebru Toprak}
\thanks{The first author is supported by  Simons  Foundation  Grant  511825. \\
 This material is based upon work supported by the National Science Foundation under Grant No. DMS-1440140 while the second author was in residence at the Mathematical Sciences Research Institute in Berkeley, California, during the Fall 2018 Semester. }
\address{Department of Mathematics\\
Rose-Hulman Institute of Technology \\
Terre Haute, IN 47803, U.S.A.}
\email{green@rose-hulman.edu}
\address{Department of Mathematics \\
University of Illinois \\
Urbana, IL 61801, U.S.A.}
\email{toprak2@illinois.edu}

\maketitle

\section{Introduction}

We consider the linear fourth order Schr\"odinger equation 
\begin{align*} 
 i \psi_t = H \psi, \,\,\, \psi(0,x)= f(x), \,\,\, H:=(-\Delta)^2+ V.
\end{align*}
This equation was introduced by Karpman \cite{K} 
and Karpman and Shagalov \cite{KS} to account for   small fourth-order
dispersion in the propagation of  laser beams in a bulk medium with Kerr
nonlinearity.

In the free case, i.e. when $V=0$, the solution operator $e^{-it \Delta^2 }$ preserves the $L^2$ norm and satisfies the following $L^1 \rightarrow L^{\infty}$ dispersive estimate, see \cite{BKS},
\begin{align*}
\|e^{-it \Delta^2 } f \|_{L^{\infty}} \les |t|^{-\f{d} 4} \|f\|_{L^1}. 
\end{align*}
Let $P_{ac}(H)$ be the projection onto the absolutely continuous spectrum of $H$ and $V(x)$ be a real valued decaying potential. Our main purpose in this paper is to extend the above dispersive estimate in dimension four when there are obstructions at zero, i.e  distributional solutions to $H\psi =0$ with $\psi \in L^p(\mathbb R^4)$ for some $p\geq 2$.  In particular, we provide a full classification of the zero energy obstructions as eigenfunctions or resonances, in terms of distributional solutions to $H\psi=0$ with the type of obstruction depending on the decay of $\psi$ at infinity.  We then prove dispersive bounds of the form 
\begin{align} \label{dispersivefree}
\|e^{-it H }P_{ac}(H) f \|_{L^{\infty}} \les \gamma(t) \|f\|_{L^1},
\end{align}
or a variant with spatial weights, for each type of zero energy obstruction where $\gamma(t) \rightarrow 0$ as $t\rightarrow \infty$. Such estimates can be important tools in the study of asymptotic stability of solitons for non-linear  equations.

The dispersive estimates in the form of \eqref{dispersivefree} has been widely studied for Schr\"odinger equation. It was observed that the natural $|t|^{-\f{d}2}$ decay rate for the Schr\"odinger evolution is affected by zero energy obstructions.  
In particular, the time decay for large $|t|$ is slower if there are obstructions at zero, see for example \cite{JSS, Yaj, Sc2, Gol2, eg2, EGG, GG1,GG2} which consider the Schr\"odinger equation.  The notation of a resonance is defined for a rather general class of operators of the form $f(-\Delta)+V$ in \cite{BN}, and the large time decay as an operator is studied between weighted $L^2(\R^n)$ spaces under the assumption that zero energy is regular.    There are no existing works on the effect of zero energy obstructions on the time decay for the fourth order equation to the best of the authors' knowledge.

Similar dynamics are expected also for the fourth order Schr\"odinger equation, in the sense that zero energy obstructions should make the time decay slower. In fact, it was shown in \cite{fsy} that in this case if zero is regular  then the natural time decay $\gamma(t)=|t|^{-\f{d}4}$ is valid in dimensions $d>4$, and is $|t|^{-\f12}$ for large $t$ in $d=3$. In particular, we note that the case of $d=4$, and the case when zero energy is not regular in all dimensions were open until now.

 In the case of the Schr\"odinger equation, the full structure of the obstructions are known and are obtained by a careful expansion of the Schr\"odinger resolvent, $R _0(z):= (-\Delta -z)^{-1}$, around $z=0$, see \cite{JN, ES, eg2, EGG}.  In particular, in dimensions $d=3,4$ the zero energy obstructions are composed of a one dimensional space of resonances and a finite dimensional  eigenspace whereas the structure is more complicated in $d=2$, \cite{JN,eg2}. We note that a similar difficulty and structure appears for the fourth order Schr\"odinger equation in dimension four.  Specifically, we have the following representation, which follows from the second resolvent identity (see also \cite{fsy})
\begin{align}  \label{RH_0 rep}
R (H_0; z):=( (-\Delta)^2 - z)^{-1} = \frac{1}{2z^{\f12}} \Big( R _0(z^{\f12}) - R_0 (-z^{\f12}) \Big), \quad z\in \mathbb C\setminus[0,\infty).
\end{align}
Moreover, we obtain that the set of zero energy obstructions consists of a space of three distinct types of resonances in addition to the zero energy eigenspace.  This resonance space is at most  $15$ dimensional,  and there is a finite-dimensional space of eigenfunctions, see Section~\ref{sec:classification}.

Before we give the main results we define the following spaces, 
\begin{align*}
L^{p,\sigma}:= \{ f: \la \cdot \ra^{\sigma} f \in L^{p} \}, \quad
L^p_{ \pm \omega} :=  \{ f: (\log (2+ |\cdot|)^{\pm 2}  f \in L^{p} \}.
\end{align*}
Here $\la \cdot \ra = (1+ |\cdot|^2)^{\f12}$.  Throughout the paper we write $a-$ to mean $a-\epsilon$ for a small, but fixed $\epsilon>0$.  Similarly, $a+$ denotes $a+\epsilon$.

Our main theorem is the following, see  Definiton~\ref{def:restype} for the precise definition of the different types of resonances mentioned below.  Heuristically, a resonance of the first kind may be classified in terms of the existence of solutions to $H\psi=0$ with $\psi \in L^\infty \setminus L^p$ for any $p<\infty$.  There is a resonance of the second kind if $\psi \in L^p$ for all $p>4$ but $\psi\notin L^4$.  There is a resonance of the third kind if $\psi \in L^p$ for all $p>2$ but $\psi\notin L^2$, and there is a resonance of the fourth kind if $\psi \in L^2$.

\begin{theorem}\label{thm:main}
	
	Suppose that $|V(x)|\les \la x\ra^{-\beta}$,
	\begin{enumerate}[i)]
	\item If zero is regular,
	then if $\beta>4$, 
	$$
		\|e^{-itH} P_{ac}(H)\|_{L^1 \rightarrow L^{\infty}} \les | t | ^{-1}
	$$
	\item If there is a resonance of the first kind,
	then if $\beta>4$, 
	$$
		\|e^{-itH}P_{ac}(H)\|_{L^1 \rightarrow L^{\infty}} \les | t | ^{-1}
	$$
	\item If there is a resonance of the second kind at zero,
	then if $\beta>12 $, 
	$$
		\|e^{-itH} P_{ac}(H)\|_{L^1 \rightarrow L^{\infty}} \les \left\{ \begin{array}{ll}
		|t|^{-\f12} & |t|\geq 2\\
		|t|^{-1} & |t|<2
		\end{array} \right.
	$$
	Furthermore, there exists a finite rank operator $F_t$ satisfying $\|F_t\|_{L^1\to L^\infty} \les \la t \ra^{-\f12}$ so  that
	$$
		\| e^{-itH} P_{ac}(H) -F_t\|_{L^{1,2+}\to L^{\infty,-2-}} \les | t |^{-1}
	$$

	\item If there is a resonance of the third or fourth kind at zero,
	then if $\beta>12$, 
	$$ \|e^{-itH} P_{ac}(H)\|_{L^1 \rightarrow L^{\infty}} \les 
	\left\{ \begin{array}{ll}
		\frac{1}{\log |t|} & |t|\geq 2\\
		|t|^{-1} & |t|<2
	\end{array} \right.
	$$
\end{enumerate}
\end{theorem}

The case of a resonance of the second kind has no analogue in the Schr\"odinger equation, such a resonance is new in the case of the fourth order equation.   Furthermore, we give an explicit formulation of the operator $F_t$, see \eqref{eqn:Ft explicit}.

We also show that if zero is regular we can obtain integrable time decay  rate for the cost of spatial logarithmic weights.  This provides a result analogous to what is known for the two dimensional Schr\"odinger equation, see \cite{Mur,eg3}.   The free evolution satisfies the estimate $\eqref{dispersivefree}$ with $\gamma(t)=t^{-1}$ and cannot decay faster due to the resonance $\psi=1$ of $(-\Delta)^2$. Therefore, we expect a faster time decay for the perturbed evolution if zero energy is regular. Our second result is the following.

\begin{theorem} \label{thm:main1}Suppose that $|V(x)|\les \la x\ra^{-4-}$ and $t>2$. If zero is regular then 
$$
\| e^{-itH} P_{ac}(H)\|_{L_{\omega}^1 \rightarrow L_{-\omega}^{\infty}} \les \frac{1}{t \log ^2 t}.
$$
\end{theorem}

As in \cite{eg3} and \cite{ebru}, to obtain Theorem~\ref{thm:main1} we use the following interpolation 
\be\label{eqn:interp trick}
	\min\bigg(1, \frac{a}{b}\bigg) =  \frac{\log^2 a}{\log^2 b}, \,\,\ a,b>2.
\ee
between the result of Theorem~\ref{thm:main} when zero is regular and the pointwise bound 
$$
		\big|[e^{-itH} P_{ac}(H) ](x,y) \big| \les \frac{w(x) w(y)}{t \log^2 (t)} + \frac{ \la x \ra ^{0+} \la y \ra ^{0+}}{t^{1+}}, \,\,\ t>2,
	$$
which we prove in Section~\ref{sec:wtd} for small energy and Section~\ref{sec:large} for large energy.  We note that, in Theorem~\ref{thm:main1} we assume the same decay on the potential with the case when zero is regular.  To achieve the improved time decay rate, we employ a careful argument based on Lipschitz continuity of the resolvents, which was inspired by a similar analysis for the two-dimensional Schr\"odinger operator  in \cite{eg3}.

There are not many works considering the perturbed linear fourth order Schr\"odinger equation. There is more study of scattering, global existence, and the  stability or instability of the solitons of the nonlinear equations, see for example \cite{Lev,P,P1,MXZ1,MXZ2,Dinh}. There are also works which study the decay estimates for the fourth order wave equation, \cite{Lev1,LS}.

The free linear fourth order Schr\"odinger equation is studied by Ben-Artzi, Koch, and Saut \cite{BKS}. They present sharp estimates on the derivatives of the kernel of the free operator, (including $ (-\Delta)^2 \pm \Delta $), which may be used to obtain $L^{1,\sigma} \rar L^{\infty,-\sigma}$ and Strichartz type of estimates for the free operator. In \cite{DDY},  the generalized Schr\"odinger operator $(-\Delta)^{2m} + V$ is studied by means of maximal and minimal forms. The authors applied their main result in this paper to obtain sharp bound  on the kernel of corresponding semigroup for $d<2m$. 

 The perturbed equation is considered by Feng, Soffer and Yao in \cite{fsy}, they prove time decay estimates between weighted $L^{2}$ spaces.  This work has roots in Jensen and Kato's study of the Schr\"odinger operator \cite{JenKat}, see also \cite{JSS}. In addition they presented $L^{1} \rar L^{\infty}$ dispersive estimates for $d=3$ and $d>4$, when zero is regular. In the same paper, they established Strichartz type of estimates for the linear equation with time source,   for $d=3$ and $d>4$.  Our work was motivated by \cite{fsy}, in particular their Remark~2.11, where they state that the threshold behavior in four dimensions would be difficult and of interest to study.

   There are some other high energy results in terms of weighted Sobolev norms which are applicable to the fourth order Schr\"odinger operator. These type of estimates mainly arise as a consequence of the holomorphic extension of the corresponding resolvent operators to the real line away from point spectrum between weighted Sobolev spaces. For instance, in \cite{agmon}, Agmon studied constant coefficient differential operators $P(D)$ of order $m$ and of principal type. Later, Murata  established high energy decay estimate for the first order pseudo-differential operators \cite{Mur1} and higher order elliptic operators \cite{Mur2}. In \cite{Mur}, Murata also established low energy result on constant coefficient  differential operators of order $m$, however the assumption that all critical points of polynomial $P(\xi)$ are non-degenerate does not apply to $(-\Delta)^2$. 

We note that $(-\Delta)^2$ is essentially self-adjoint with $\sigma_{ac}((-\Delta)^2)= [0,\infty)$. Therefore, by Weyl's criterion, we have $\sigma_{ess}(H) = [0,\infty)$ for a sufficiently decaying potential.  Let $\lambda \in \R^{+}$, we define the resolvent operators as 
\begin{align}
&\label{RH_0 def}R^{\pm}(H_0; \lambda ) := R^{\pm}(H_0; \lambda \pm i0 )= \lim_{\epsilon \to 0^+}((-\Delta)^2 - ( \lambda  \pm i\epsilon))^{-1}, \\
& \label{Rv_0 def} R_V^{\pm}(\lambda ) := R_V^{\pm}(\lambda  \pm i0 )= \lim_{\epsilon \to 0^+}(H - ( \lambda \pm i\epsilon))^{-1}.
\end{align}

Note that using the representation \eqref{RH_0 rep} for $R (H_0;z)$ in the definition \eqref{RH_0 def}, for $\lambda$ in the first quandrant of the complex plane, writing $z=\lambda^4$ we obtain 
\be\label{eq:4th resolvdef}
R^{\pm}(H_0; \lambda^4)= \frac{1}{2 \lambda^2} \Big( R^{\pm}_0(\lambda^2) - R_0(-\lambda^2) \Big).
\ee 
These operators are well-defined from $L^{2,-\sigma}$ to $L^{2,\sigma}$, for $\sigma>\f12$ by Agmon's limiting absorption principle, \cite{agmon}.
From this identity, it is noted in \cite{fsy} that the behavior of the spectral variable $\lambda$ of the fourth order resolvent in dimension $d$ is the same as that of the Schr\"odinger resolvent in dimension $d-2$ as $\lambda\to 0$.  One can see this from the power like behavior in $\lambda$ as $\lambda \to 0$ of the resolvents and the fact that the operator $\frac{1}{2\pi \lambda}\frac{d}{d\lambda}$ takes the $d$-dimensional Schr\"odinger resolvent to an $d-2$ dimensional resolvent.  Due to this similarity, our approach has roots in the analysis of the Schr\"odinger operator, particularly \cite{eg2}.  However, the dependence on the spatial variables of the integral kernel of the resolvent operator is different.  One consequence of this difference, which
we show in Section~\ref{sec:classification}, is that the threshold behavior is more complicated for the fourth order equation.

We believe that the method used here can be modified to analyze the operators $(-\Delta)^2 \pm \Delta + V$. However, we do not expect the structure of the threshold resonances should not be expected to be similar to Section~\ref{sec:classification} below. We expect the operator $(-\Delta)^2-\Delta+V$ to have a threshold structure that mirrors that of the Schr\"odinger operators studied in \cite{EGG}, while the differential operator $(-\Delta)^2 + \Delta $ has two positive critical points and hence requires additional investigation.

As usual, we use functional calculus and the Stone's formula to write 
\begin{align}
\label{stone}
 \ e^{-itH}    P_{ac}(H) f(x) = \frac1{2\pi i}  \int_0^{\infty} e^{-it\lambda}   [R_V^+(\lambda)-R_V^{-}(\lambda)]  f(x) d\lambda. 
 \end{align}
Here the difference of the perturbed resolvents provides the spectral measure.   Unfortunately, unlike the Schr\"odinger operator, see for example \cite{kato}, we can not guarantee the absence of embedded eigenvalues in the continuous spectrum for $H=(-\Delta)^2+V$. Typically one uses a Carleman type estimate for $(-\Delta)^2$ and unique continuation theorems for $H$ to rule out positive eigenvalues. Unfortunately, none of these are available even for compactly supported or differentiable potentials. Therefore, as in \cite{fsy}, we assume absence of positive eigenvalues. Under this assumption, a limiting absorption principle for $H$ is established, see \cite[Theorem~2.23]{fsy}, which we use to control the large energy portion of the evolution.  The large energy is unaffected by the zero energy obstructions, and our main contribution is to control the small energy portion of the evolution in all possible cases.

The paper is organized as follows. We begin with developing expansions for the free resolvent in Section~\ref{sec:free}. We also analyze the dispersive bounds for the free equation in this section. In Section~\ref{sec:low energy}, we develop expansions for the resolvents and other operators we need to analyze the spectral measure for small $\lambda$ in a neighborhood of zero to understand the dispersive bounds.
In Section~\ref{sec:reg sect}, we consider dispersive bounds when zero is regular to establish the uniform and weighted dispersive bounds.  In Section~\ref{sec:nonreg} we consider the effect of the zero energy resonances.
We divide this section into  subsections in which we develop expansions for $R_V^{\pm}(\lambda )$ for each different type of resonance at zero.  We further estimate the contribution of the low energy portion of \eqref{stone}, when $\lambda$ is in a neighborhood of zero, for each type of resonance. These bounds establish the low energy portions of Theorems~\ref{thm:main} and \ref{thm:main1}. In Section~\ref{sec:large}, we estimate \eqref{stone} away from zero, thus completing the proofs of Theorems~\ref{thm:main} and \ref{thm:main1}. Finally in Section~\ref{sec:classification}, we provide a characterization
of the threshold resonances and eigenfunctions.

\section{The Free  Evolution}\label{sec:free}
 
In this section we obtain expansions for the free fourth order Schr\"odinger resolvent operators $R^{\pm}(H_0; \lambda^4)$, using the identity \eqref{RH_0 rep}  and the Bessel function representation of the Schr\"odinger free resolvents $R^{\pm}_0(\lambda^2)$. We use these expansions to establish  dispersive estimates for the free fourth order Schr\"odinger evolution, and throughout the remainder of the paper to study the spectral measure for the perturbed operator. 

Before we obtain an expansion for $R^{\pm}(H_0; \lambda^4)$, we give the definition of the following operators that arise naturally in our expansions. 
\begin{align} 
	G_0f(x)&:=-\frac{1}{4\pi^2}\int_{\R^4} 
	\frac{f(y)}{|x-y|^2}\,dy=(-\Delta)^{-1} f(x) , \label{G0 def}\\
	G_1f(x)&:=-\frac{1}{8\pi^2}\int_{\R^4} \log(|x-y|) f(y)\, dy \label{G1 def}, \\
	G_{2j} f(x) & := c_{2j} \int_{\R^4} |x-y|^{2j} f(y)\, dy ,\,\, j>0, \\
	G_{2j+1} f(x) & := c_{2j+1} \int_{\R^4} |x-y|^{2j}  \log|x-y| f(y)\, dy ,\,\, j>0, \label{Gj def}
\end{align}
where $c_j$ are certain real-valued constants. Moreover, we define \begin{align}
\label{gj def}
g^+_j(\lambda)=\overline{g_j^-(\lambda)}=\lambda^{2j}(a_j\log(\lambda)+z_j)
\end{align}
where $a_j\in \R\setminus\{0\}$ and $z_j\in \mathbb C
\setminus \R$. The exact values of the constants in these definitions are unimportant for our analysis. 

Throughout the paper, we use the notation
$
f(\lambda)=\widetilde O(g(\lambda))
$
to denote
$$
\frac{d^j}{d\lambda^j} f = O\big(\frac{d^j}{d\lambda^j} g\big),\,\,\,\,\,j=0,1,2,3,...
$$
Unless otherwise specified, the notation refers only to derivatives with respect to the spectral variable $\lambda$.
If the derivative bounds hold only for the first $k$ derivatives we  write $f=\widetilde O_k (g)$.  In the following sections, we use that notation for operators as well as
scalar functions; the meaning should be clear from
context.

\begin{lemma} \label{lem:R0low} If $\lambda|x-y|\ll 1$, then we have 
\begin{multline} \label{eq:R0low}
 R^{\pm}(H_0;\lambda^4)(x,y)	
	=\tilde g_1^{\pm}(\lambda)+G_1(x,y)
	+c^{\pm}\lambda^2 G_2(x,y)
	+\tilde g^{\pm}_3(\lambda) G_4(x,y) \\ + \lambda^{4} G_5 (x,y)
	+ O( (\lambda |x-y|)^{6-}  ). 
\end{multline} 
Here $ \tilde g_j^-(\lambda) + i \Im{z_j} = \tilde g_j^+(\lambda) : = \lambda^{2j-2} \big( a_j \log(\lambda)+b_j\big) $ with $b_j \in \mathbb{C}$. Moreover, $c^{\pm} \in \mathbb{C}$ and $a_j$, $z_j$ are the same coefficients defined in $g_j(z)$ in \eqref{gj def}.
\end{lemma}

A simple calculation shows that $(-\Delta) G_1(x,y)=G_0(x,y)$, and consequently $G_1(x,y)=[(-\Delta)^2]^{-1}(x,y)$.

\begin{proof} 
We use the expansion \eqref{eq:4th resolvdef} for the spectral measure.	
Therefore, to prove the statement we first recall the expression of the free Schr\"odinger resolvents in dimension four in terms of the Bessel functions,
\begin{align} \label{eq:freeschresolvent}
	R_0^{\pm}(\lambda^2)(x,y)=\pm \frac{i}{4}\frac{\lambda}{2\pi |x-y|}\bigg( 
	J_1(\lambda|x-y|)\pm i Y_1(\lambda|x-y|)
	\bigg).
\end{align}
When $\lambda|x-y|\ll 1$, we may write (see \cite{AS,JN,EGG, GT})
\begin{align}\label{resolv expansion4}
	R_0^{\pm}(\lambda^2)(x,y)  &=G_0(x,y)+g_1^{\pm}(\lambda)
	 +\lambda^2 G_1(x,y)+g_2^\pm(\lambda)G_2(x,y)+\lambda^4G_3(x,y)\\
&\qquad  + g_3^{\pm}(\lambda)G_4(x,y)
	+  \lambda^6 G_5(x,y)+   \widetilde O_2( \lambda^2 (\lambda |x-y|)^{6-}), \nn
\end{align} 
 where $g_j$'s are given as in \eqref{gj def}. 

Notice that exchanging $\lambda$ with $i \lambda$ in \eqref{resolv expansion4} yields 
\begin{multline}\label{resolvimaginaryexpansion4}
R_0^{+}(-\lambda^2)(x,y)  =G_0(x,y)+ g_1^{+}(i \lambda)
	 -\lambda^2 G_1(x,y) +g_2^+ (i\lambda)G_2(x,y) +\lambda^4G_3(x,y) \\
	  + g_3^{+}(i \lambda)G_4(x,y) - \lambda^6 G_5(x,y) +
	\widetilde O_2(\lambda^{8-}|x-y|^{6-}  ).
\end{multline}

Finally substituting \eqref{resolv expansion4} and  \eqref{resolvimaginaryexpansion4} into  \eqref{eq:4th resolvdef}, we obtain the small argument expansion for the kernel of the resolvent,
\begin{align}\label{RH}
	R^{\pm}(H_0;\lambda^4)(x,y)
	 :=\frac{1}{2\lambda^2}\bigg[ &
	[g_1^{\pm}(\lambda) -g_1^+(i\lambda)] +2\lambda^2 G_1(x,y)  +  
	  [g_2^{\pm}(\lambda)-g_2^+(i\lambda)]G_2(x,y)  \nn \\
	 & +[g_3^{\pm}(\lambda)-g_3^+(i\lambda)]G_4(x,y)  
	 +2\lambda^6 G_5 (x,y)
	\bigg]  
	 +\widetilde O_2( (\lambda |x-y|)^{6-} ).  
\end{align}
Recalling the definition of $g_j$'s we see for $\lambda \in \R^{+}$ 
$$
      g_j^+(i\lambda)= \begin{cases} - g_j^{+}(\lambda) - i \frac{a_j \pi}{2} \lambda^{2j} & j=1,3  \\
       g_j^{+}(\lambda) + i \frac{a_j \pi}{2} \lambda^{2j} & j=2.
       \end{cases}
$$
 In particular, we obtain	
\begin{multline*}
R^+(H_0;\lambda^4)(x,y)	
	=\bigg[ \frac{g_1^+(\lambda)}{\lambda^2} + i\frac{a_1\pi}{4} \bigg] +G_1(x,y)
	- i\frac{\pi a_2}{4}\lambda^2 G_2(x,y)\\
	+\bigg[ \frac{g_2^+(\lambda)}{\lambda^2} + i\frac{a_2 \pi}{4} \bigg] G_4(x,y) + \lambda^4 G_5(x,y)
	+ \widetilde O_2( (\lambda |x-y|)^{6-}  ).
\end{multline*}
Letting $  \tilde g_j^+(\lambda) := \frac{g_j^+(\lambda)}{\lambda^2} + i\frac{a_j\pi}{4} $ and $c^{+}=-i\frac{a_2 \pi}{4}$,  we establish the statement for $R^+(H_0;\lambda^4)$.   

For $R^-(H_0;\lambda^4)$, note that $\Im\{g_j^+(\lambda)\} = \Im{z_j} \lambda^{2j}$ and hence,
$$
 \frac{g_j^-(\lambda) - g_j^+(i\lambda)}{2 \lambda^2} =  \frac{\overline{g_j^{+}(\lambda)} + g_j^+(\lambda)}{2 \lambda^2} + i\frac{a_j\pi}{4}= \tilde g_j^+(\lambda)- i \Im{z_j}, \,\,\ j=1,3,
$$
$$
 \frac{g_j^-(\lambda) - g_j^+(i\lambda)}{2 \lambda^2} =  \frac{\overline{g_j^{+}(\lambda)} - g_j^+(\lambda)}{2 \lambda^2} - i\frac{a_j\pi}{4}= - i \Im{z_j}+c^{+} ,  \,\,\ j=2
$$
Using these equalities in \eqref{RH}, we obtain the statement.

\end{proof}

We define a smooth cut-off function to a neighborhood of zero, $\chi\in C^\infty(\mathbb R)$ with $\chi(\lambda)=1$ for $|\lambda |\leq \lambda_1\ll 1$ and $\chi(\lambda)=0$ for $|\lambda |\geq2 \lambda_1\ll 1$ for a sufficiently small constant $\lambda_1$.  We also use the complementary cut-off away from a neighborhood of zero, $\widetilde \chi(\lambda)=1-\chi(\lambda)$.

\begin{rmk} \label{rmk:R0low}The following observation will be useful in the next sections. 
\begin{align*}
	& \tilde g_1^{+}(\lambda)+G_1(x,y)
	=a_1\log(\lambda|x-y|)+z_1+ \frac{a_1\pi}{2}  \\
	& \tilde g_3^+(\lambda)G_4(x,y)+\lambda^4 G_5(x,y)=(\lambda |x-y|)^4\bigg(a_3\log(\lambda|x-y|)+z_3+i\frac{a_3\pi}{2}\bigg).
\end{align*}
In particular, for any $\epsilon > 0$
\begin{align*}
	& \chi(\lambda |x-y|) [\tilde g_1^+(\lambda)+G_1(x,y)]= \widetilde O \big( (\lambda |x-y|)^{- \epsilon} \big). \\
	& \chi(\lambda |x-y|) [ \tilde g_3^+(\lambda)G_4(x,y)+\lambda^4 G_5(x,y)]= \widetilde O_2 \big( (\lambda |x-y|)^{4- \epsilon} \big).
	\end{align*}

\end{rmk}

Next we obtain an expansion for $R^{\pm}(H_0;\lambda^4)(x,y)$ when $\lambda|x-y| \gtrsim 1$. To do that we use the large energy expansion for the free Schr\"odinger resolvent, (see, {\em e.g.}, \cite{AS, EGG, GT})
$$ R_0^{\pm}(\lambda^2)(x,y)= c \frac{\lambda  } {r} H_1^{\pm} ( \lambda) \,\,\  R_0^{+}(-\lambda^2)(x,y)= c \frac{\lambda} {r} H_1^+ (i \lambda), $$
where 
\begin{align*}
	&H_1^\pm(z)= e^{\pm iz} \omega_\pm(z),\,\,\,\, |\omega_{\pm}^{(\ell)}(z)|\lesssim (1+|z|)^{-\frac{1}{2}-\ell},\,\,\,\ell=0,1,2,\ldots.
\end{align*}
Therefore, for $\lambda r \gtrsim1$, $r:=|x-y|$, we have 
\begin{align} \label{eq:R0high}
R^{\pm} (H_0; \lambda^4) = e^{\pm  i \lambda r} \tilde{ \omega}_{\pm}  ( \lambda r) + e^{-\lambda r} \tilde{ \omega}_{+}  ( \lambda r)
\end{align} 
where $\tilde{\omega}_{\pm}  ( \lambda r)=  (\lambda r)^{-1} \omega_{\pm}(\lambda r)$.

The following representation is useful for further analysis. 

\begin{lemma} \label{lem:R0 exp}
	We have
	$$ R^{\pm} (H_0; \lambda^4)= \tilde{g}_1^{\pm}(\lambda) + G_1(x,y) + E^\pm_0(\lambda)(x,y)$$
 where  the error term satisfies
$ E^\pm_0(\lambda)(x,y) =  \widetilde O_1\big( (\lambda |x-y|)^{\ell} \big) $ for $ 0< \ell \leq 2$, and $ E^\pm_0(\lambda)(x,y) =  \widetilde O _2\big( (\lambda |x-y|)^{\f12} \big) $. 

\end{lemma}

\begin{proof}
	For convenience, let $r:=|x-y|$, then the statement is clear for $\lambda r \ll 1$ by \eqref{eq:R0low} and Remark~\ref{rmk:R0low}. If $\lambda r \gtrsim 1$, let $\ell >0$, one has 
\begin{align*}
	& | e^{\pm  i \lambda r} \tilde{ \omega}  ( \lambda r) + e^{-\lambda r} \tilde{ \omega}  ( \lambda r)- \tilde{g}_1^{\pm}(\lambda) -G_1(x,y)| \les (\lambda r)^{\ell}, \\
	& | \partial_\lambda \{ e^{\pm  i \lambda r} \tilde{ \omega}  ( \lambda r) + e^{-\lambda r} \tilde{ \omega}  ( \lambda r)- \tilde{g}_1^{\pm}(\lambda) \}| \les \frac{r}{(\lambda r)(1+\lambda r)^{\f12}} + \frac{1}{\lambda} \les \lambda^{\ell -1} r^{\ell}, \\
	& | \partial^2_\lambda \{ e^{\pm  i \lambda r} \tilde{ \omega}  ( \lambda r) + e^{-\lambda r} \tilde{ \omega}  ( \lambda r)- \tilde{g}_1^{\pm}(\lambda) \} | \les \frac{r^2}{(\lambda r)(1+\lambda r)^{\f12}} + \frac{1}{\lambda^2} \les \lambda^{-\f32} r^{\f12}.
\end{align*}
This establishes the proof.
\end{proof}

\begin{corollary}\label{cor:Error Lips}
	
	For any $0<\alpha<1$ and $0<a<b\ll1$, we have
	$$
		|\partial_\lambda E^\pm_0(b)-\partial_\lambda E^\pm_0(a)|
		\les |b-a|^{\alpha} a^{-\f32}|x-y|^{\f12}(a^{\f12+\ell} |x-y|^{\ell-\f12} )^{1-\alpha}.
	$$
	
\end{corollary}

\begin{proof}
	
	By the Mean Value Theorem, we have
	$$
		|\partial_\lambda E^\pm_0(b)-\partial_\lambda E^\pm_0(a)|\les |b-a| a^{-\f32}|x-y|^{\f12}.
	$$	
	Since $a<b$, we also have the trivial bound
	$$
	|\partial_\lambda E^\pm_0(b)-\partial_\lambda E^\pm_0(a)|\les a^{\ell-1}|x-y|^\ell=a^{-\f32}|x-y|^{\f12}(a^{\f12+\ell} |x-y|^{\ell-\f12} ).
	$$	
	Interpolating between the two bounds yields the claim.
	
\end{proof}

We will use this Lipschitz bound on the error term in the proof of the weighted, time-integrable bound.  We note that, in particular, if we choose $\alpha=\ell=0+$, then we obtain
\begin{align*}
	|\partial_\lambda E^\pm_0(b)-\partial_\lambda E^\pm_0(a)|
	\les |b-a|^{\ell} a^{-1+\f{\ell}{2}-\ell^2}|x-y|^{\f12+\f{3\ell}{2}-\ell^2}
\end{align*}
with $ 2 \ell \geq \f{3\ell}{2}-\ell^2 >0 $.

Finally we are ready to analyze the free evolution. The first estimate in the next lemma was obtained by  Ben-Artzi, Koch and Saut in \cite{BKS}. Our approach relies on the well-known Stone's formula \eqref{stone}.

\begin{lemma}\label{lem:free est} We have the bound
\begin{align} \label{to bound by1}
	\sup_{x,y} \left|\int_0^{\infty} e^{-it \lambda^4} \lambda^3 [R^{+}(H_0; \lambda^4)-R^{-}(H_0; \lambda^4)] d \lambda \right| \les  | t | ^{-1}.
\end{align}
Furthermore, 
\begin{align*}
	\int_0^{\infty} e^{-it \lambda^4} \lambda^3 \chi(\lambda) [R^{+}(H_0; \lambda^4)-R^{-}(H_0; \lambda^4)] d \lambda=- \frac{\Im z_1}{4 t} + O \big( t^{-\f98} \la x \ra^{\f12} \la y \ra ^{\f12} \big).
\end{align*} 

\end{lemma}

\begin{proof}
Note that by \eqref{eq:R0low} and \eqref{eq:R0high} we obtain
\begin{align}
	 R^{+}(H_0; \lambda^4)-R^{-}(H_0; \lambda^4) = \chi(\lambda r) [ i \Im z_1 +  O((\lambda r)^{2}) ]+\widetilde{\chi}(\lambda r) e^{i\lambda r} \tilde{\omega}  ( \lambda r).
\end{align}
 Therefore
 \begin{align*}
	 &| R^{+}(H_0; \lambda^4)-R^{-}(H_0; \lambda^4)| \les 1, \\
	 &| \partial_\lambda \{ R^{+}(H_0; \lambda^4)-R^{-}(H_0; \lambda^4)| \les \frac{ \chi(\lambda r) r^{\f12}}{\lambda^{\f12}} + \frac{\widetilde{\chi}(\lambda r) r}{(\lambda r) (1+\lambda r)^{\f12}}.
\end{align*}

By integration by parts, we have
\begin{multline*}
	 \Big|\int_0^{\infty} e^{-it \lambda^4} \lambda^3 [R^{+}(H_0; \lambda^4)-R^{-}(H_0; \lambda^4)] d \lambda \Big| 
	 \les \frac{1}{t}+ \frac{1}{t} \int_0^{\infty}\Big[ \frac{ \chi(\lambda r) r^{\f12}}{\lambda^{\f12}} + \frac{\tilde{\chi}(\lambda r) r}{(\lambda r) (1+\lambda r)^{\f12}} \Big] d \lambda\\
	\les t^{-1}+t^{-1}\bigg( \int_0^{r^{-1}} \frac{r^{\f12}}{\lambda^{\f12}}\, d\lambda + \int_{r^{-1}}^\infty \frac{1}{\lambda^{\f32} r^{\f12}}\, d\lambda \bigg)
	  \les t^{-1}
\end{multline*}
This establishes the first claim. 

For the second claim, we note that using the expansion in Lemma~\ref{lem:R0 exp} and the fact that $\tilde g_1^+(\lambda)=\tilde g_1^-(\lambda) + i \Im{z_1}$
\begin{align} 
	R^{+}(H_0; \lambda^4)-R^{-}(H_0; \lambda^4) =  i \Im z_1+ E_0(\lambda)(x,y),
\end{align}
where, combining the error terms from Lemma~\ref{lem:R0 exp}, we have $E_0(\lambda)(x,y)= \widetilde O_2( (\lambda r)^{\f12} )$.  After the first integration by parts,  we have
\begin{multline*}
	\int_0^{\infty} e^{-it \lambda^4} \lambda^3 [R^{+}(H_0; \lambda^4)-R^{-}(H_0; \lambda^4)] d \lambda 
	=\frac{\Im(z_1)}{4t}+\frac{1}{4it}\int_0^\infty e^{it\lambda^4} \partial_\lambda E_0(\lambda)(x,y)\, d\lambda. 
\end{multline*}
To control the second integral we write
$$
	\bigg|\frac{1}{4it}\int_{0}^{t^{-\f14}} e^{-it\lambda^4} \partial_\lambda E_0(\lambda)(x,y)\, d\lambda+
	\frac{1}{4it}\int_{t^{-\f14}}^{\infty} e^{-it\lambda^4} \lambda^3 [\lambda^{-3}\partial_\lambda E_0(\lambda)(x,y)]\, d\lambda\bigg|. 
$$
Using Lemma~\ref{lem:R0 exp}, direct integration of the first term shows it may be bounded by $t^{-\f98}\la x\ra^{\f12}\la y\ra^{\f12}$.  The second term may be bounded, after an integration by parts, by $t^{-\f98}\la x\ra^{\f12}\la y\ra^{\f12}$. 

\end{proof}

We remark here that the above bounds can be modified if we insert cut-offs to low or high enery respectively.  In a neighborhood of zero, that is if we insert the cut-off $\chi(\lambda)$ into the integrand, the integrals are all bounded as $t\to 0$.  Outside of a neighborhood of zero, if we insert the cut-off $\tilde\chi(\lambda)$ into the integrand, the boundary term $\Im(z_1)/(4t)$ from integrating by parts is replaced with zero.

\section{Resolvent expansions about zero Energy} \label{sec:low energy} 

The effect of the presence of zero energy resonances   is only felt in the small energy regime, different resonances change the asymptotic behavior of the perturbed resolvents and hence that of the spectral measure as $\lambda\to 0$ which we study in this section.  We provide expansions of the resolvents and other operators we need to understand the effect of each type of resonance or lack of resonances on the spectral measure.

To understand \eqref{stone} for small energies, i.e. $ \lambda \ll 1$, we use the symmetric resolvent identity.  We define $U(x)=$sign$(V(x))$, $v(x)=|V(x)|^{\f12}$, and write 
\begin{align} \label{resid}
R^{\pm}_V(\lambda)= R^{\pm}(H_0, \lambda^4) - R^{\pm}(H_0, \lambda^4)v (M^{\pm} (\lambda))^{-1} vR^{\pm}(H_0, \lambda^4) 
\end{align}
where $M^{\pm}(\lambda) := U + v R^{\pm}(H_0, \lambda^4) v $. As a result, we need to obtain expansions for $(M^{\pm} (\lambda))^{-1}$ depending on the resonance type at zero, see Definition~\ref{def:restype}. In the following subsections we determine these expansions case by case and establish their contribution to Stone's formula via the symmetric resolvent identity, \eqref{resid}. 

Recall the definition of the Hilbert-Schmidt
norm of an operator $K$ with kernel $K(x,y)$,
$$
	\| K\|_{HS}:=\bigg(\iint_{\R^{8}}
	|K(x,y)|^2\, dx\, dy
	\bigg)^{\f12}.
$$
Let $T:= U + v G_1 v$, we have the following expansions. 

\begin{lemma}\label{lem:M_exp} Let $P$ be the projection onto the span of $v$ and $\widetilde{g}^\pm (\lambda)=\| V\|_1\tilde g_1^\pm(\lambda)$.  If $ |v(x)| \les \la x \ra^{-2-2\ell-}$, we have
\begin{align}
 \label{Mexp0} &M^{\pm}(\lambda)= \widetilde{g}^\pm (\lambda) P+T+  M_0^\pm (\lambda)&,  \\
	  & \sum_{j=0}^1\|\sup_{0<\lambda<\lambda_1} \lambda^{-l+j}  \partial_\lambda^j 
	 M_0^{\pm}(\lambda) \|_{HS} \les  1,  \,\,\ 0 <  \ell \leq 2,  \nn  \\ 
	 & \|\sup_{0<\lambda<b<\lambda_1}  |b-\lambda|^{-\ell} \lambda ^{1+\ell^2-\f{\ell}{2}} 
	|\partial_\lambda M_0^{\pm}(b)-\partial_\lambda M_0^{\pm}(\lambda)| \|_{HS} \les  1,   \,\,\, 0< \ell < 1 .  \nn  
	\end{align}
If $	|v(x)| \les \la x \ra^{-4-\ell-}$,
	\begin{align}
 \label{Mexp1}
           &M^{\pm}(\lambda)=  \widetilde{g}^\pm (\lambda) P+T+ c^{\pm} \lambda^2 v G_2 v +   M_1^\pm (\lambda), \hspace{4cm} \\  
          & \sum_{j=0}^1 \|\sup_{0<\lambda<\lambda_1} \lambda^{- \ell-2+j}  \partial_\lambda^j 
	 M_1^{\pm}(\lambda)\|_{HS}\les 1, \,\,\, 0 \leq \ell < 2   \nn 
	 \end{align}
If $|v(x)|  \les \la x \ra^{-6-\ell-}$,	 
	 \begin{align}
  \label{Mexp3}  &M^{\pm}(\lambda)= \widetilde{g}^\pm (\lambda) P+T+  c^{\pm} \lambda^2 v G_2 v + \tilde g_3^{\pm}(\lambda)  vG_4v   + \lambda^4 vG_5v + M_2^\pm (\lambda),  \\  
        & \sum_{j=0}^1 \|\sup_{0<\lambda<\lambda_1} \lambda^{ -4-l+j}  \partial_\lambda^j 
	 M_2 ^{\pm}(\lambda)\|_{HS}\les 1, \,\,\, 0 \leq \ell \leq 2. \nn 
  \end{align}
\end{lemma} 	

\begin{proof}
Recall that $M^{\pm}(\lambda) = U + v R^{\pm}(H_0, \lambda^4) v $.
Therefore, the proof of the first assertion in \eqref{Mexp0} follows easily by Lemma~\ref{lem:R0 exp}, whereas the second assertion can be obtained taking $\alpha = \ell$ in Corollary~\ref{cor:Error Lips}. 

Moreover, when $ \lambda |x-y| \ll 1$ the proof of \eqref{Mexp1} and \eqref{Mexp3} follow from the expansions \eqref{eq:R0low} in Lemma~\ref{lem:R0low} and Remark~\ref{rmk:R0low}. Finally, the following observation establishes the statement for \eqref{Mexp1} and \eqref{Mexp3} for $ \lambda |x-y| \gtrsim 1$
\begin{align*}
& \partial_\lambda \{ \tilde{\chi}(\lambda) R^{\pm} (H_0; \lambda^4)  \} = \partial_\lambda\{ e^{\pm  i \lambda r} \tilde{ \omega}_{\pm}  ( \lambda r) + e^{-\lambda r} \tilde{ \omega}_{+}  ( \lambda r)\} \les \lambda^{-1}= \partial_{\lambda}\{ (\lambda r)^{k} \}, \,\,\ k>0, \\
& \partial^2 _\lambda \{ \tilde{\chi}(\lambda) R^{\pm} (H_0; \lambda^4)  \} \les \lambda^{-3/2} r^{1/2} = \partial^2_{\lambda}\{ (\lambda r)^{1/2+k} \}, \,\,\ k>0. 
 \end{align*}
\end{proof}

The definition below classifies the type of resonances that may occur at the threshold energy. In Section~\ref{sec:classification}, we establish this classification in detail.
\begin{defin} \label{def:restype}
\begin{itemize}
\item  Let $Q=1-P$. We say that zero is regular point of the spectrum of $(-\Delta)^2 +V$ provided $QTQ$ is invertible on $QL^2$. In that case we define $D_0:=(QTQ)^{-1}$ as an operator on $QL^2$. 
\item Assume that zero is not regular point of the spectrum. Let $S_1$ be the Riesz projection onto the kernel of $QTQ$. Then $QTQ+S_1$ is invertible on $L^2$. Accordingly,  we
define $D_0 = (QT Q + S_1)^{-1}$,  as an operator on $QL^2$.  This doesn't conflict with the previous definition since $S_1=0$  when zero is regular. We say there is
a resonance of the first kind at zero if the operator $T_1 := S_1T P T S_1$ is invertible
on $S_1L^2$. 
 \item We say there is a resonance of the second kind if $T_1$ is not invertible on $S_1L^2$, but $T_2:= S_2vG_2vS_2$ is invertible where $S_2$ is the Riesz projection onto the kernel of $S_1T P T S_1$.  Moreover, we define $D_1:= (T_1 +S_2)^{-1} $  as an operator on $S_1L^2$.  
 
 \item We say there is a resonance of the third kind if $T_2$ is not invertible on $S_2L^2$  but $T_3:= S_3vG_4vS_3$ is invertible. Here $S_3$ is the Riesz projection onto the kernel of $S_2vG_2vS_2$. We define $D_2:= (T_2 + S_3)^{-1} $ as an operator on $S_2L^2$. 
 
 \item Finally if $T_3$ is not invertible we say there is a resonance of the fourth kind at zero. Note that in this case the operator $T_4:= S_4 v G_5v S_4$  is always invertible where $S_4$ the Riesz projection onto the kernel of $S_3vG_4vS_3$. We define $D_3:= (T_3 + S_4)^{-1}$ as an operator on $S_3L^2$. 

 \end{itemize}
\end{defin}
We note that the types of resonance present have a similarity to those that appear for a Schr\"odinger operator in dimension two.  Specifically, a resonance of the first kind is analogous to an `s-wave' resonance, a resonance of the third kind is analogous to a `p-wave' resonance and a resonance of the fourth kind is an eigenfunction.  The resonance of the second kind, and its dynamical consequences, are new and have no counterpart in the Schr\"odinger operator analogy.

As in the four dimensional Schr\"odinger operator, see the Remarks after Definition~2.5 in \cite{EGG}, $T$ is a compact perturbation of $U$.  Hence, the Fredholm alternative guarantees that $S_1$ is a finite-rank projection.
With these definitions first notice that, $ S_4 \leq S_3 \leq S_2 \leq S_1 \leq Q $, hence all $S_j$ are finite-rank projections orthogonal to the span of $v$. Second, since $T$ is a self-adjoint operator and $S_1$ is the Riesz projection onto its kernel, we have 
$S_1 D_0= D_0 S_1 = S_1$.  Similarly, $S_2 D_1= D_1 S_2 = S_2$, $S_3 D_2= D_2 S_3 = S_3$ and $S_4D_3=D_3S_4=S_4$. 
We introduce the following terminology from \cite{Sc2,eg2,eg3}:

\begin{defin}
	We say an operator $T:L^2(\R^2) \to   L^2(\R^2)$ with kernel
	$T(\cdot,\cdot)$ is absolutely bounded if the operator with kernel
	$|T(\cdot,\cdot)|$ is bounded from $  L^2(\R^2)$ to $ L^2(\R^2)$. 	
\end{defin}

We note that Hilbert-Schmidt and finite-rank operators are absolutely bounded operators.

\begin{lemma} Let $|V(x)| \les \la x \ra^{-\beta}$ for some $\beta>4$, then $QD_0Q$ is absolutely bounded. 
\end{lemma}

The proof follows the proof of Lemma~8 in \cite{Sc2}. The only difference is that one needs $v(x) \log|x-y| v(y)$ to be Hilbert-Schmidt in $\R^4$ instead of $\R^2$, which requires more decay on $V$. 

If zero is regular then one obtains the following expansion for $(M^\pm(\lambda))^{-1}$.
\begin{lemma}\label{regular} Assume $|v(x)| \les \la x \ra ^{-2-2\ell-}$ and assume  that zero is regular for $H$. Then, we have
\begin{align*}
(M^\pm(\lambda))^{-1}&=h_\pm(\lambda)^{-1} S + QD_0Q+E^\pm(\lambda), \\ 		
	&\sum_{j=0}^1\|\sup_{0<\lambda<\lambda_1}  
	\lambda^{j-\ell} \partial^j_\lambda E ^{\pm}(\lambda)\|_{HS}\les 1, \,\, 0 < \ell \leq 2,\\
	& \|\sup_{0<\lambda<b<\lambda_1}  |b-\lambda|^{-\ell} \lambda ^{1+\ell^2-\f{\ell}{2}} 
	|\partial_\lambda E^{\pm}(b)-\partial_\lambda E^{\pm}(\lambda)| \|_{HS} \les  1,   \,\,\, 0< \ell < 1 .
\end{align*}
Here $h^\pm(\lambda)=\tilde{g}_1^{\pm}(\lambda)+\,$trace$\,(PTP-PTQD_0QTP)$, and
$$
	S =\left[\begin{array}{cc}
	P& -PTQD_0Q\\ -QD_0QTP &QD_0QTPTQD_0Q
	\end{array} \right],
$$ 
is a self-adjoint, finite rank operator.

 \end{lemma}
 \begin{proof} 
	We consider only the `+' case, the case `-' proceeds identically.  Let
 \begin{align*} A(\lambda)= \widetilde{g}^{+}(\lambda) P+ T =  \left[\begin{array}{cc}
	\tilde{g}_1^{+}(\lambda)  P + PTP & PTQ  \\ QTP &QTQ 
	\end{array} \right].
 \end{align*}
Then by Feshbach formula (see  Lemma~2.8 in \cite{eg2}) we have  $A^{-1}(\lambda ) = h^{+}(\lambda)^{-1} S +QD_0Q$. Hence, using the equality
$$
	M^{+} (\lambda) = A(\lambda) + M_0^{+}(\lambda) = (I + M_0^{+}(\lambda) A^{-1}(\lambda)) A(\lambda),
$$
and Neumann series expansion we obtain
$$
(M^{+}(\lambda))^{-1}=A^{-1}(\lambda)(I +M_0^{+}(\lambda) A^{-1}(\lambda))^{-1}= h^{+}(\lambda)^{-1} S +QD_0Q+E^{+}(\lambda).
$$
The Lipschitz bound follows from the bounds in Lemma~\ref{lem:R0 exp} and Corollary~\ref{cor:Error Lips}, along with the fact that $A^{-1}(\lambda)=\widetilde O(1)$ implies that (for $0<\lambda<b<\lambda_1$)
$$
	|\partial_\lambda A^{-1}(\lambda)-\partial_\lambda A^{-1}(b)|\les |b-\lambda|^{\alpha} \lambda^{-1-\alpha}.
$$

\end{proof}
The following lemma from \cite{JN} is the main tool to obtain the expansions of $(M_{\pm}(\lambda))^{-1}$ in different assumptions  on zero energy.
\begin{lemma}\label{JNlemma}
	Let $M$ be a closed operator on a Hilbert space $\mathcal{H}$ and $S$ a projection. Suppose $M+S$ has a bounded
	inverse. Then $M$ has a bounded inverse if and only if
	$$
	B:=S-S(M+S)^{-1}S
	$$
	has a bounded inverse in $S\mathcal{H}$, and in this case
	\begin{align} \label{Minversegeneral}
	M^{-1}=(M+S)^{-1}+(M+S)^{-1}SB^{-1}S(M+S)^{-1}.
	\end{align}
\end{lemma}
We use this lemma repeatedly with $M=M^{\pm}(\lambda)$ and the projection $S_1$.  In fact, much of our technical work in the following sections is devoted to finding appropriate expansions for $B^{-1}$ under the various spectral assumptions. 
We first use Lemma~\ref{JNlemma} to compute appropriate expansions for $(M^\pm (\lambda)+S_1)^{-1}$. 

\begin{lemma} \label{lem:M+S inv} Let $0<\lambda \ll 1$, if  $|V(x)| \les \la x \ra^{-4-2\ell-} $, for some $0 < \ell \leq 2$, then
\begin{align} \label{M-1exp0}
  ( M^{\pm}(\lambda)+S_1 )^{-1} (\lambda)=h_{\pm}^{-1} (\lambda)S + QD_0Q+ & \widetilde {O}_1 (\lambda^{\ell}).
  \end{align}
  If $|V(x)| \les \la x \ra^{-8-2\ell-} $ for some $0 < \ell < 2$, then
  \begin{align}
   \label{M-1exp1}
	 ( M^{\pm}(\lambda)+S_1 )^{-1} (\lambda)= h_{\pm}^{-1} (\lambda)S + QD_0Q -& c^{\pm}  \lambda^2 A^{-1} vG_2vA^{-1} +\widetilde O_1 (\lambda^{2+ \ell}).   
	 \end{align}
	 If $|V(x)| \les \la x \ra^{-12-2\ell-} $ for some $0 < \ell \leq 2$, then
	 \begin{multline}
	 \label{M-1exp2}
           ( M^{\pm}(\lambda)+S_1 )^{-1} (\lambda)= h_{\pm}^{-1} (\lambda)S + QD_0Q 
          -  c^{\pm}\lambda^2 A^{-1} vG_2vA^{-1} \\
            - \tilde g_3^{\pm}(\lambda) A^{-1} vG_4vA^{-1} 
            +  \lambda^4 [ A^{-1} vG_5vA^{-1} - A^{-1} vG_2vA^{-1} vG_2vA^{-1} ]   + O_1(\lambda^{4+\ell}). 
 \end{multline}
Here 
\begin{align*}
&S =\left[\begin{array}{cc}
P& -P(T+S_1)QD_0Q\\ -QD_0Q(T+S_1) P &QD_0Q(T+S_1)P(T+S_1)QD_0Q
\end{array} \right],\\
& A^{-1} (\lambda) = h_{\pm}^{-1} (\lambda)S + QD_0Q.
\end{align*}
\end{lemma}
\begin{proof}
With some abuse of notation we redefine $A(\lambda)$ in Lemma~\ref{regular} as 
$$A(\lambda)= \tilde{g}_1^{\pm}(\lambda)  P+S_1= \left[\begin{array}{cc}
	\tilde{g}_1^{\pm}(\lambda)  P + P(T+S_1)P & P(T+S_1)Q  \\ Q(T+S_1)P &Q(T+S_1)Q 
	\end{array} \right] 
	$$
Since $(T+S_1)$ is invertible on $QL^2$, we can use the Feshbach formula as in the proof of Lemma~\ref{regular}.  Doing so, we obtain $A^{-1}(\lambda)= h_{\pm}^{-1} (\lambda)S + QD_0Q$, where $D_0$ is the inverse of $(T+S_1)^{-1}$ on $QL^2$ and $S$ is as above. Now, knowing that the leading term  is invertible for $0<\lambda \ll1$, we can use Neumann series expansion to invert $M(\lambda)+S_1$, using the expansions \eqref{Mexp0}, \eqref{Mexp1}, and  \eqref{Mexp3} to obtain  \eqref{M-1exp0}, \eqref{M-1exp1} and \eqref{M-1exp2} respectively. 

\end{proof}
\section{Low energy dispersive bounds when zero is regular} \label{sec:reg sect}

In this and the following section we analyze the perturbed evolution $e^{-itH}$ in $L^1 \rar L^{\infty}$ setting for small energy, when the spectral variable $\lambda$ is in a small neighborhood of the threshold energy $\lambda=0$. As in the free case, we represent the solution via Stone's formula, see \eqref{stone}. As usual, we analyze \eqref{stone} separately for large energy, when $\lambda \gtrsim 1$, and for small energy, when $ \lambda \ll 1$, see for example \cite{Sc2, eg2, ebru}.   The presence of zero energy resonances is not seen in the large energy expansions. We start with the small energies $ \lambda \ll 1$, and analyze the large energy in Section~\ref{sec:large} to complete the proofs of Theorems~\ref{thm:main} and \ref{thm:main1}.

In this section, we utilize the expansions in the previous section to understand the dispersive bounds in the low energy regime, when the spectral variable is in a sufficiently small neighborhood of zero.  We break the section into two subsections.  In the first subsection, we consider the first claim in Theorem~\ref{thm:main}, and prove a uniform bound with the natural $|t|^{-1}$ decay rate when zero is regular.  In the second subsection, we consider a weighted bound that attains faster time decay when zero is regular for the claim in Theorem~\ref{thm:main1}.

\subsection{The Unweighted bound}
In this subsection we consider the case when zero is regular and prove bounds on the solution operator form $L^1$ to $L^\infty$.  In particular, we prove the following low energy estimate. 
\begin{prop}\label{prop:zero reg unwtd} Let $|V(x)|\les \la x\ra^{-4-}$ and suppose that zero energy is regular.  Then, we have
	\begin{align*}
		\sup_{x,y} \left|\int_0^{\infty} e^{-it \lambda^4} \lambda^3 \chi(\lambda) [R^{+}_V-R^-_V](\lambda)(x,y) d \lambda \right| 
		\les \la t \ra ^{-1}.
	\end{align*}

\end{prop}
We prove this proposition in a series of lemmas.
Using \eqref{resid}, one has
\begin{align}\label{eqn:Rv expansion}
	R_V^\pm(\lambda)=R^\pm(H_0;\lambda^4)-R^\pm(H_0;\lambda^4)v
	\left[h^\pm(\lambda)^{-1}  S + 
	 QD_0Q+ 
	E^\pm(\lambda) \right] \\ 
	 vR^\pm(H_0;\lambda^4). \nn
\end{align}
The contribution of the first term to the Stone's formula is controlled by Lemma~\ref{lem:free est}.  We now consider the remaining terms.  For notational convenience, we write $R^\pm(\lambda^4):=R^\pm(H_0;\lambda^4)$.

\begin{lemma}\label{prop:QD0Q bound} 
   We have the bound
	$$
	\sup_{x,y} \bigg|	
	\int_0^{\infty} e^{-it \lambda^4} \lambda^3 \chi(\lambda)  
	\big[R^\pm vQD_0QvR^\pm](\lambda^4)(x,y)
	\big] 	 d \lambda \bigg| \les \la t \ra^{-1}.
	$$
	\end{lemma}

Before we start to prove Lemma~\ref{prop:QD0Q bound} we give some important bounds on the free resolvent which will be useful for our analysis. We first decompose the free resolvent into low and high arguments based on the size of $\lambda |x-y|$. In particular, we write
	\begin{align} \label{Rdecompose}
		R^\pm(\lambda^4)(x,y)= \chi(\lambda |x-y|)R^\pm(\lambda^4)(x,y)+\widetilde \chi(\lambda|x-y|)R^\pm(\lambda^4)(x,y).
	\end{align}
Recall that on the support of $\widetilde \chi(\lambda|x-y|)$ one has, 
\begin{align}\label{eqn:R0 hi exp}
	R^\pm(\lambda^4)(x_1,x)\widetilde\chi(\lambda|x-x_1|)=e^{\pm  i \lambda r} \tilde{ \omega}  ( \lambda r) + e^{-\lambda r} \tilde{ \omega}  ( \lambda r) :=A(\lambda, |x-x_1|) \\ 
	 \tilde{\omega} ( \lambda r)=  \tilde{\chi}(\lambda r) (\lambda r)^{-1} \omega(\lambda r) \nn \,\,\, , \,\,\
	|\omega_{\pm}^{(\ell)}(z)|\lesssim (1+|z|)^{-\frac{1}{2}-\ell}. \nn
	\end{align}
	This implies
	\be\label{eqn:A def}
		|A(\lambda,|x-x_1|)|\les \frac{\widetilde\chi(\lambda|x-x_1|)}{(\lambda|x-x_1|)^{\f32}}, \qquad
		|\partial_\lambda A(\lambda,|x-x_1|)|\les \frac{\widetilde\chi(\lambda|x-x_1|)}{\lambda^{\f32}|x-x_1|^{\f12}}.
	\ee

On the other hand, for  $\lambda|x-x_1|\ll 1$ we have 
\begin{multline} \label{Rsmall} 
	R^{\pm}(\lambda^4)(x,x_1) \chi(\lambda|x-x_1)\\
	=a_1\log (\lambda|x-x_1|)+\alpha^{\pm}+\widetilde O_2((\lambda|x-x_1|)^2)= \gamma^{\pm}(\lambda, |x-x_1|). 
\end{multline}
where $a_1\in \mathbb R\setminus \{0\}$ and $\alpha^{\pm} \in \mathbb{C}$. 	
The following lemma plays an important role on controlling the operators arising from  \eqref{Rsmall}.

\begin{lemma} \label{cancellemma} Let $p= |x-x_1|$, $q= \la x \ra $, and 
\begin{align*}
	F^{\pm}(\lambda,x,x_1):=\chi(\lambda p ) [a\log (\lambda p)+\alpha^{\pm}] - \chi(\lambda q ) [a\log (\lambda q )+  \alpha^{\pm}]  . 
	\end{align*}
	Defining $k(x,x_1):=1+\log^+|x_1|+\log^-|x-x_1|$, with $\log^-r:=\chi_{0<r<1}\log r$ and $\log ^+ r:= \chi_{r> 1} \log r$, one has 
   \begin{align*} 
	& |F^{\pm}(\lambda,x_1,x)|\leq \int_0^{2 \lambda_1} | \partial_\lambda F^{\pm}(\lambda,x,x_1)|\, d\lambda + | F^{\pm}(0+,x,x_1)|  \les  k(x,x_1), \\ 
	&  |\partial_\lambda F^{\pm}(\lambda,x_1,x)|\les \frac{1}{\lambda}  .
	\end{align*} 
\end{lemma}
Here, $F^{\pm}(0+,x,x_1)$ denotes $\lim_{\lambda\to 0^+} F^{\pm}(\lambda,x,x_1)$.
In particular, we note that $\partial_\lambda F(\lambda,x_1,x)$ is integrable in a neighborhod of zero.
\begin{proof} These bounds are  established in \cite{Sc2}. For the sake of completeness, we show that $\partial_\lambda F(\lambda,x_1,x)$ is integrable in a neighborhood of zero. 
Note that we have
\begin{align*}
|\partial_\lambda F^{\pm}(\lambda,x_1,x)| = p \chi^\prime (\lambda p ) [\log(\lambda p) + \alpha^{\pm}] + q [\chi^\prime(\lambda q ) \log(\lambda q) + \alpha^{\pm}] + \frac{\chi(\lambda p) - \chi(\lambda q)}{\lambda} 
\end{align*}
Notice that first term is supported only when $\lambda \approx p^{-1}$. Hence, its contribution to the integral is bounded. The second term is bounded similarly. For the third term, notice $\chi(\lambda p) - \chi(\lambda q)$ is supported in $[2 \lambda_1 p^{-1}, 2 \lambda_1 q^{-1}]$ and its contribution to the integral is bounded by $k(x,x_1)$. 
\end{proof}

\begin{proof}[Proof of  Lemma~\ref{prop:QD0Q bound}] We consider the following case
$$
	\sup_{x,y} \bigg|	
	\int_0^{\infty} e^{-it \lambda^4} \lambda^3 \chi(\lambda)  
	\big[R^+vQD_0QvR^+](\lambda^4)(x,y)
	\big] 	 d \lambda \bigg| \les \la t \ra^{-1}.
	$$
Following the argument below, the same bound holds if one exchanges $R^+$ with $R^-$. 

Notice that using the orthogonality property $Qv=vQ=0$ and \eqref{Rdecompose}, we can exchange $R^+$ on both sides of $vQD_0Qv$ with 
\begin{align}\label{eqn:H def}
	H^{+}(\lambda,x,x_1):= F^{+} (\lambda,x,x_1) + \chi(\lambda |x-x_1|) \widetilde O_1( (\lambda |x-x_1|)^{\f12}) + A( \lambda,|x-x_1|),
	\end{align}
and consider 
\begin{align} \label{HQH}	
	\int_0^{\infty} e^{-it \lambda^4} \lambda^3 \chi(\lambda)  
	\big[H^+vQD_0QvH^+](\lambda^4)(x,y)
	\big] 	 d \lambda.
	\end{align}
Note that this integral is bounded, since by Lemma~\ref{cancellemma} $|H^{\pm}(\lambda,x,x_1)| \les k(x,x_1)$ and one has
\begin{align} \label{spat bound}
  \|v(x_1)k(x,x_1)\|_{L^2_{x_1}} \| |QD_0Q| \|_{L^2\to L^2}
		\| v(y_1) k(x,x_1) \|_{L^2_{y_1}} \les 1
\end{align}
provided $v(x_1)[ \la x_1\ra^{0+}+\log^-|x_1-\cdot|]\in L^2$. Here we also used the fact that $QD_0Q$ is absolutely bounded. 

Next we prove the time dependent bound on \eqref{HQH}. Suppressing the operators and spatial integrals for the moment, by integration by parts it is enough to bound 
$$
		\bigg|\frac{1}{4it}\int_0^\infty e^{-it\lambda^4}\partial_\lambda \{ \chi(\lambda) H(\lambda,x,x_1)  H(\lambda,y,y_1) \} \, d\lambda \bigg|.  
$$
Note  by Fundamental Theorem of Calculus and the fact that $\chi(\lambda)$ has compact support, it suffices to bound this integral to bound \eqref{HQH}.

Notice that by Lemma~\ref{cancellemma} and \eqref{eqn:A def}, one has  
\begin{multline} \label{HH}
 \int_0^{\infty} | \partial_\lambda \{ \chi(\lambda) H^{\pm} \} (x,x_1) H^{\pm} (y_1,y)| \ d \lambda  \\
	\les k(y,y_1) \Bigg( \int_0^{2\lambda_1} |\partial_\lambda F| d \lambda + \int_{0} ^{r_1^{-1}}\frac{r_1^{\f12}}{\lambda^{\f12} } d \lambda + \int_{r_1^{-1}} ^{\infty} \frac{1}{\lambda^{\f32} r_2^{\f12}} d \lambda \Bigg) \les k(y,y_1) k(x,x_1)  . 
	\end{multline} 

Hence, by symmetry and \eqref{spat bound} we establish the statement. 
\end{proof}
	We next estimate the contribution of 
\begin{align*} 
 R^{\pm} (\lambda^4)vh_{\pm}^{-1} (\lambda)S v R^{\pm} (\lambda^4)
 \end{align*}
to the Stone's formula.  In contrast to the previous lemma, we do not have any orthogonality properties to use. Therefore, we have to utilize the cancellation between $+$ and $-$ terms arising due to Stone's formula. To do that we use the algebraic fact, 
\begin{align}\label{alg fact}
	\prod_{k=0}^MA_k^+-\prod_{k=0}^M A_k^-
	=\sum_{\ell=0}^M \bigg(\prod_{k=0}^{\ell-1}A_k^-\bigg)
	\big(A_\ell^+-A_\ell^-\big)\bigg(
	\prod_{k=\ell+1}^M A_k^+\bigg).
\end{align}

\begin{lemma}\label{prop:S bound}  We have the bound
	$$
		\sup_{x,y} \bigg|	
		\int_0^{\infty} e^{it \lambda^4} \lambda^3 \chi(\lambda)  
		\bigg[\frac{R^{+} (\lambda^4)vS v R^{+} (\lambda^4)}{h_+(\lambda)}-\frac{R^{-} (\lambda^4)v S v R^{-} (\lambda^4)}{h_-(\lambda)}\bigg](x,y)
		\, d \lambda \bigg| \les \la t\ra^{-1}.
	$$

\end{lemma}

\begin{proof} 
 Using \eqref{Rdecompose} and  \eqref{alg fact} we need to bound  the contribution of the following operators, (with $r_1:=|x-x_1|$ and $r_2:=|y-y_1|$)
\begin{align*}
&\Gamma_{ll}^1 := \big[\gamma^{+}(\lambda,r_1)-\gamma^{-}(\lambda, r_1)\big] \frac{v S v}{h_{\pm} (\lambda)} \gamma^{\pm}(\lambda, r_2), \\
 & \Gamma_{ll}^2 := \Big[ \frac{1}{h^{+} (\lambda)} - \frac{1}{h^{-} (\lambda)} \Big] \gamma^{+}(\lambda,r_1)vSv\gamma^{-}(\lambda,r_1),  \\
 & \Gamma_{lh}^\pm := \gamma^{\pm}(\lambda, r_1)\frac{v S v}{h_{\pm} (\lambda)}A(\lambda, r_2 ), \,\,\,\ \Gamma_{hh}^\pm= A(\lambda, r_1 ) \frac{v S v}{h_{\pm} (\lambda)}A(\lambda, r_2 ).
\end{align*} 
We first consider $\Gamma_{ll}^1$. Notice that on the support of $\chi(\lambda)$, one has 
$$
	\chi(\lambda) \chi(\lambda r) |\log(\lambda r)| \les 1+|\log (\lambda)|+\log^-(r).
$$
Hence, we see
\begin{align*} 
 \bigg|\frac{\chi(\lambda r) \log(\lambda r) }{ \log \lambda + z}\bigg|  \les 1+ \frac{1+\log^{-} r}{\log( \lambda) +z}.
\end{align*}
This gives, 
\begin{align} 
& \big|\frac{\gamma^{\pm}(\lambda, r)}{h^{\pm} (\lambda)} \big| \les 1 + \log^{-} r  \label{gammah} \\
 & \Big|\partial_\lambda \Big[ \frac{\gamma^{\pm}(\lambda, r)}{h^{\pm} (\lambda)} \Big]  \Big| 
\les \chi(\lambda r) \lambda^{-1/2}r^{1/2} + \frac{1+\log^{-} r}{\lambda (\log( \lambda) +z)^2} \label{gammahder}.
\end{align}
Moreover we have,  
\begin{align} \label{gammah dif}
 |\gamma^{+}(\lambda, r)-\gamma^{-}(\lambda, r)\big| \les \chi(\lambda r) , \quad
 \big| \partial_\lambda\big[\gamma^{+}(\lambda, r)-\gamma^{-}(\lambda, r)\big] \big|\les \chi(\lambda r) \lambda^{-\f12}r^{\f12}. 
\end{align}
Therefore, we have 
\begin{align} \label{Gamma bound}
\sup_{x,y}  \int_0^{\infty} \Big| \partial_\lambda& \{ \chi(\lambda) \Gamma_{ll}^1(\lambda)(x,y) \} \Big|d\lambda \les \\
 \int_{\R^{8}} \int_0^{\infty} & \frac {\chi(\lambda) \chi(\lambda r_1) \chi(\lambda r_2) \max(r_1^{\f12},r_2^{\f12},1 ) [vSv](x_1,y_1)}{\lambda^{\f12}} d \lambda dx_1 dy_1 \nn \\ + 
 & \int_{\R^{8}} \int_0^{\infty} \frac{\chi(\lambda) [vSv](x_1,y_1)k(y,y_1)}{\lambda (\log( \lambda) +z)^2} d \lambda dx_1 dy_1 \les 1,\nn
\end{align}
which suffices to establish the bound $t^{-1}$ for $\Gamma^1_{ll}$.
For the final bound, we note that  
$$
	\bigg|\int_0^\infty 	 \frac {\chi(\lambda) \chi(\lambda r_1) \chi(\lambda r_2) \max(r_1^{\f12},r_2^{\f12},1 )}{\lambda^{\f12}}\, d\lambda \bigg| \les \int_0^1 \lambda^{-\f12}\, d\lambda+ \sum_{j=1}^{2} \int_0^{r_j^{-1}} r_j^{\f12}\lambda^{-\f12}\, d\lambda \les 1.
$$
Moreover, noting that $S$ is absolutely bounded, similar to \eqref{spat bound}, one can control the spatial integral since
\begin{align} \label{vk bound}
 \|v(x_1) \|_{L^2_{x_1}} \|\,|S|\,\|_{L^2 \rightarrow L^2} \| v(y_1) k(y_1,\cdot)  \|_{L^2_{y_1}} \les 1,
\end{align}
uniformly in $x$ and $y$.
Finally, using \eqref{gammah},\eqref{gammah dif} together with \eqref{vk bound}, we obtain the boundedness of the contribution of $\Gamma_{ll}^1$ to the Stone's formula.

Next we prove the statement for $\Gamma_{ll}^2$. Note that 
$$
	\frac{1}{h_{+} (\lambda)} - \frac{1}{h_{-} (\lambda)} =
	\frac{1}{\widetilde g^+(\lambda)+c}-\frac{1}{\widetilde g^-(\lambda)+c} =\frac{2i\Im(z_1)}{(\widetilde g^-(\lambda)+c)(\widetilde g^+(\lambda)+c)}. $$
Therefore, 
\begin{align} \label{h dif bound}
\Bigg| \left[ \frac{1}{h^{+} (\lambda)} - \frac{1}{h^{-}(\lambda)} \right] \gamma^{\pm}(\lambda,r_1) \gamma^{\pm}(\lambda,r_2)\Bigg| \les k(x,x_1)k(y,y_1).
\end{align}
 Moreover, by \eqref{gammahder}, one has 
\begin{align*}
\Bigg| \partial_\lambda \left(\left[ \frac{1}{h^{+} (\lambda)} - \frac{1}{h^{-}(\lambda)} \right] \gamma^{\pm}(\lambda,r_1) \gamma^{\pm}(\lambda,r_2) \right)\Bigg| \les \frac{k(x,x_1)k(y,y_1)}{\lambda \log^2 \lambda} + \frac{ \max(r_1^{\f12},r_2^{\f12},1 ) }{\lambda^{1/2}}.
\end{align*}  
 Hence, we obtain
 \begin{multline*}
\sup_{x,y}  \int_0^{\infty} \Big| \partial_\lambda \{ \chi(\lambda) \Gamma_{ll}^2 (\lambda)(x,y) \} \Big|d\lambda \les \\
 \int_{\R^{8}} \int_0^{\infty}  \frac {\chi(\lambda) \chi(\lambda r_1) \chi(\lambda r_2) \max(r_1^{\f12},r_2^{\f12},1 ) [vSv](x_1,y_1)}{\lambda^{1/2}} d \lambda dx_1 dy_1  \\ + 
  \int_{\R^{8}} \int_0^{\infty} \frac{\chi(\lambda) k(x,x_1) [vSv](x_1,y_1) k(y,y_1)}{\lambda (\log (\lambda))^2}d \lambda dx_1 dy_1 \les 1.
\end{multline*}
The boundedness of the spatial integrals is controlled as in \eqref{spat bound} and the previous case. Moreover,  \eqref{h dif bound} together with \eqref{vk bound} shows that the contribution of $\Gamma_{ll}^2$ to the Stone's formula is bounded by one. 

For the remaining cases, we do not rely on any cancellation between the `+' and `-' terms. In fact, the contribution of
$\Gamma_{hh}$ is included in the analysis of integral \eqref{HQH}  in Lemma~\ref{prop:QD0Q bound}. Therefore, we consider only $\Gamma_{lh}$. The contribution of $\Gamma_{lh}$ is bounded by \eqref{gammah} and the fact that $ |A(\lambda, r )| \les 1$. For the time bound, note that if the derivative falls on $ h^{-1}_{\pm} (\lambda) \gamma^{\pm}(\lambda, r_1)$, we use \eqref{gammahder} and control the integral as in $\Gamma_{ll}^1$ or $\Gamma_{ll}^2$. If the derivative falls on $A(\lambda, r_2 )$,  we have 
\begin{align*}
	\sup_{x,y}  \int_0^{\infty} \Big| \chi(\lambda) &\gamma^{\pm}(\lambda, r_1)\frac{v S v}{h_{\pm} (\lambda)} \partial_{\lambda}A(\lambda, r_2 ) \} \Big|d\lambda \les \\
	\int_{\R^{8}} \int_0^{\infty} & \frac {\chi(\lambda) (1+\log^{-} r_1) |vSv|(x_1,y_1) \widetilde{\chi}(\lambda r_2)}{r_2^{\f12} \lambda^{\f32}} d \lambda dx_1 dy_1 \\
	\les  \int_{\R^{8}} &  (1+\log^{-} r_1) |vSv|(x_1,y_1) \int_0^{\infty} \frac{\tilde{\chi}(\lambda r_2)}{r^{\f12} \lambda^{\f32}} d \lambda dx_1 dy_1 \les 1.
\end{align*}
Again, the boundedness of the spatial integrals follows from the absolute boundedness of $S$ as in \eqref{spat bound}, and the previous cases.

\end{proof}

The smallness in $\lambda$ as $\lambda \to 0$ in the error term, see Lemma~\ref{regular}, allows us to integrate by parts directly.  While we cannot take advantage of any cancellation from the difference of `+' and `-' terms, this lack of cancellation is more than compensated for by the smallness of $E^\pm(\lambda)$. We prove the following lemma to control the contribution of $E^\pm(\lambda)$. 

\begin{lemma}\label{lem:Error term} Let $|V(x)| \les \la x \ra^{-4-}$. Assume $E(\lambda)= \widetilde{O}_1(\lambda^{\ell})$ for  some $0<\ell <1$ as an absolutely bounded operator. Then, we have the bound
	$$
	\sup_{x,y} \bigg|	
	\int_0^{\infty} e^{-it \lambda^4} \lambda^3 \chi(\lambda)  
	R^\pm (\lambda^4) vE^\pm (\lambda )vR^\pm (\lambda^4)	\, d \lambda \bigg| \les \la t\ra^{-1}.
	$$	
	
\end{lemma}

\begin{proof}
	
	For this proof, we use a less delicate expansion for the free resolvent, 
	$$
		R^\pm(\lambda^4)(x,y) =\widetilde O_1( (\lambda |x-y|)^{-\epsilon} ),
	$$
	where we may choose any $\epsilon>0$, see Remark~\ref{rmk:R0low}.  We choose $\epsilon=0+$ to minimize the required decay on the potential.
	Then we consider 
	$$
		\int_0^\infty e^{-it \lambda^4} \lambda^3 \chi(\lambda)  
		\widetilde O_1( \lambda^{-2\epsilon} |y-y_1|^{-\epsilon}|x-x_1|^{-\epsilon}) vE^\pm v (\lambda )(y_1,x_1) \, d \lambda
	$$
	This integral is easily seen to be bounded. To see the time bound, we integrate by parts once, noting that the bounds on $E^\pm(\lambda)$ ensure the lack of boundary terms. Let $\ell = 2\epsilon+$ in Lemma~\ref{regular}. We have (with $r_1=|x-x_1|$ and $r_2=|y-y_1|$)
	\begin{multline*}
		\bigg|\int_0^\infty e^{-it \lambda^4} \lambda^3 \chi(\lambda) 
		\widetilde O_1( \lambda^{-2\epsilon} r_1^{-\epsilon}r_2^{-\epsilon}) vE^\pm v (\lambda )(y_1,x_1) \, d \lambda \bigg|\\
		\les \frac{1}{t}\int_0^\infty \bigg| \partial_\lambda \big\{ \chi(\lambda)
		\widetilde O_1( \lambda^{-2\epsilon} r_1^{-\epsilon}r_2^{-\epsilon}) vE^\pm v (\lambda )(y_1,x_1) \big\} \bigg|\, d\lambda\\
		\les \frac{1}{t} \int_0^\infty \lambda^{ \ell-1-2\epsilon}  r_2^{-\epsilon}v(y_1) [ \lambda^{-\ell} |E^\pm(\lambda)|+\lambda^{-\ell+1}|\partial_\lambda E^\pm(\lambda)| ]v(x_1)r_1^{-\epsilon}\, d\lambda.
	\end{multline*}
	It is easy to see that the $\lambda$ integral converges. Moreover, the spatial integral converges since
	$$
		\||y-y_1|^{-\epsilon} v(y_1)\|_{L^2_{y_1}}\big\| \sup_{0< \lambda \ll 1} \sum_{j=0}^1 \lambda^{j-\ell} | \partial_\lambda ^j E^{\pm}(\lambda)(x,y)|\big\|_{HS} \||x-x_1|^{-\epsilon} v(x_1)\|_{L^2_{x_1}}
	$$
	is bounded uniformly in $x,y$ for our choice of $ \epsilon =0+ $.
	
\end{proof}


	


We are now ready to prove the main proposition.

\begin{proof}[Proof of Proposition~\ref{prop:zero reg unwtd} ]
	
	By the symmetric resolvent identity, \eqref{eqn:Rv expansion}, and the discussion following the statement of the proposition, we need to control the contribution of
	$$
		h^\pm(\lambda)^{-1} S +QD_0Q+E^\pm(\lambda)
	$$
	to $(M^{\pm}(\lambda))^{-1}$ in the Stone's formula.  The required bounds are established in Lemmas~\ref{prop:S bound}, \ref{prop:QD0Q bound} and \ref{lem:Error term} respectively.

\end{proof}

\subsection{Weighted Dispersive bound}\label{sec:wtd}

It is known that when zero energy is regular for the two dimensional Schr\"odinger equation, one can obtain a faster time decay at the cost of spatial weights, \cite{Mur,eg3}. In this section we show that this is also true for the fourth order Schr\"odinger equation in four dimensions. The proof here are inspired by the weighted dispersive bound for the two-dimensional Schr\"odinger operator obtained in \cite{eg3}.
The following Proposition is the main result of this section. 

\begin{prop} \label{prop:main weight}Let $|V(x)|\les \la x\ra^{-4-}$. We have the bound, for $t>2$
	$$
		\bigg|\int_0^\infty e^{-it\lambda^4} \lambda^3 \chi(\lambda) [R_V^+-R_V^-](\lambda)(x,y)\, d\lambda \bigg| 
		\les \frac{w(x)w(y)}{t \log^2 t}+\frac{\la x\ra^{\f12} \la y \ra^{\f12}}{t^{1+}}
	$$
	where $w(x)=\log^2 (2+|x|)$.
	
\end{prop}

As usual, we begin by using the symmetric resolvent identity \eqref{eqn:Rv expansion}, where we use the expansion in Lemma~\ref{regular} for $(M^{\pm}(\lambda))^{-1}$. Recall that, by Lemma~\ref{lem:free est} the contribution of the first summand in \eqref{eqn:Rv expansion} is
$$
	\frac{\Im z_1}{4 t} + O \big( t^{-\f98} \la x \ra^{\f12} \la y \ra ^{\f12} \big).
$$

\begin{prop}\label{prop:S cancel} Let $|V(x)|\les \la x\ra^{-4-4\ell-}$ for some $\ell >0$. For $t>2$, we have 
	\begin{multline*}
		\int_0^\infty e^{-it\lambda^4} \lambda^4 \chi(\lambda) [R^+(\lambda^4)\frac{vSv}{h^+(\lambda)}R^+(\lambda^4)
		-R^-(\lambda^4)\frac{vSv}{h^-(\lambda)}R^-(\lambda^4)](\lambda)(x,y)\, d\lambda \\
		=\frac{\Im z_1}{4 t}+O\bigg( \frac{w(x)w(y)}{t \log^2 t} \bigg)+O\bigg( \frac{\la x \ra^{2\ell} \la y \ra ^{2\ell} }{t^{1+}} \bigg).
	\end{multline*}
	
\end{prop}

We observe in the proposition that the leading term   from the contribution of
$$
R^+(\lambda^4)\frac{vSv}{h^+(\lambda)}R^+(\lambda^4)
-R^-(\lambda^4)\frac{vSv}{h^-(\lambda)}R^-(\lambda^4)
$$
exactly cancels the term $\frac{\Im z_1}{4 t}$ arises from the contribution of the free resolvent leading term in \eqref{eqn:Rv expansion}.   This allows for the faster time decay  in Proposition~\ref{prop:main weight}.

To establish these bounds, we require the following oscillatory integral estimates.

\begin{lemma}\label{lem:ibp osc}
	
	For $\mathcal E(\lambda)$ compactly supported and $t>2$, we have
	$$
	\bigg| \int_0^\infty e^{-it\lambda^4}\lambda^3 \mathcal E(\lambda)\, d\lambda + \frac{i\mathcal E(0)}{4t} \bigg|\les \frac{1}{t}\int_0^{t^{-\f14}} |\mathcal E'(\lambda)|\, d\lambda+\frac{|\mathcal E'(t^{-\f14})|}{t^{\f54}}
	+\frac{1}{t^2}\int_{t^{-\f14}}^\infty \bigg|\bigg(\frac{\mathcal E'(\lambda)}{\lambda^3} \bigg)^\prime \, d\lambda \bigg|.
	$$
	
\end{lemma}

\begin{proof}
	
	We integrate by parts using $e^{-it\lambda^4}\lambda^3=-\partial_\lambda e^{-it\lambda^4}/4it$ to see
	$$
	\int_0^\infty e^{-it\lambda^4} \lambda^3 \mathcal E(\lambda)\, d\lambda=\frac{-e^{-it\lambda^4} \mathcal E(\lambda)}{4it}\bigg|^\infty_0 +\frac{1}{4it}\int_0^{\infty}e^{-it\lambda^4} \mathcal E'(\lambda)\, d\lambda.
	$$
	The boundary term for large $\lambda$ is zero because of the support of $\mathcal E(\lambda)$.  We break up the remaining integral into two pieces, first on $[0,t^{-\f14}]$ we use the triangle inequality.
	On the second piece, we integrate by parts again and we have
	$$
	\frac{1}{t^2}\frac{|\mathcal E'(\lambda)|}{\lambda^3} \bigg|^\infty_{t^{-\f14}} +\frac{1}{t^2}\int_{t^{-\f14}}^\infty |\partial_\lambda \big(\lambda^{-3}\mathcal E'(\lambda) \big)|\, d\lambda \les \frac{|\mathcal E'(t^{-\f14} )|}{t^{\f54}} +\frac{1}{t^2}\int_{t^{-\f14}}^\infty |\partial_\lambda \big(\lambda^{-3}\mathcal E'(\lambda) \big)|\, d\lambda.
	$$
	
\end{proof}

\begin{lemma}\label{lem:log osc}
	
	If $\mathcal E(\lambda)=\widetilde O_2(\frac{1}{\log^2 \lambda})$ is supported on $0<\lambda \leq \lambda_1\ll 1$, then for $t>2$ we have
	$$
	\bigg| \int_0^\infty e^{-it\lambda^4} \lambda^3 \mathcal E(\lambda)\, d\lambda \bigg| \les \frac{1}{t \log^2 t}.
	$$
	
\end{lemma}

\begin{proof}
	
	We apply Lemma~\ref{lem:ibp osc}.
	The boundary terms are zero because $\mathcal E(0)=0$.  On $[0,t^{-\f14}]$ we have the bound
	$$
	\frac{1}{t} \int_0^{t^{-\f14}} \frac{d\lambda}{\lambda |\log \lambda|^3} \sim \frac{1}{t \log^2 t}.
	$$
	On the second piece, we integrate by parts again and seek to bound
	$$
	\frac{|\mathcal E'(t^{-\f14} )|}{t^{\f54}} + \frac{1}{t^2} \int_{t^{-\f14}}^{\infty} |\partial_\lambda{ \mathcal E^\prime (\lambda)}| d \lambda .
	$$
	The boundary term contributes $\frac{1}{t \log^3 t}$.  The remaining integral is bounded by
	\begin{multline*}
	\frac{1}{t^2}\bigg[\int_{t^{-\f14}}^{t^{-\f18}} +\int_{t^{-\f18}}^{\f12}
	+\int_{\f12}^{\infty}\bigg] |(\lambda^{-3}\mathcal E'(\lambda))^{\prime}|\, d\lambda
	\les \frac{1}{t^2 |\log t|^3 } \int_{t^{-\f14}}^{t^{-\f18}} \frac{d\lambda}{\lambda^5}+\frac{1}{t^2}\int_{t^{-\f18}}^{\f12} \frac{d\lambda}{\lambda^5}+\frac{1}{t^2} \\
	\les \frac{1}{t|\log t|^3}+\frac{1}{t^{\f32}|\log t|^3}+\frac{1}{t^{\f32}}+\frac{1}{t^2}.
	\end{multline*}
	We used the fact that the integral on $\lambda\geq \f12$ converges.  Combining these bounds proves the assertion.

\end{proof}
We need the following lemma to utilize the Lipschitz continuity of the error terms in the expansions of the spectral measure.  These estimates allow us to match the assumptions on the decay on $V(x)$  in the unweighted bound of Theorem~\ref{thm:main}. 
\begin{lemma}\label{lem:ibpweight}
	
	If $\mathcal E(0)=0$ and $t>2$, then
	$$
	\bigg| \int_0^\infty e^{-it\lambda^4}\lambda^3 \mathcal E(\lambda)\, d\lambda \bigg| \les \frac{1}{t}\int_0^\infty \frac{|\mathcal E'(\lambda)|}{1+\lambda^4 t}\, d\lambda
	+\frac{1}{t} \int_{t^{-\f14}}^\infty \bigg| \mathcal E'(\lambda \sqrt[4]{1+\pi t^{-1} \lambda^{-4}})-\mathcal E'(\lambda) \bigg|\, d\lambda
	$$
	
\end{lemma}

\begin{proof}
	
	We first  integrate by parts once and use the change of variables $s=\lambda^4$ see
	$$
	\int_0^\infty e^{-it\lambda^4}\lambda^3 \mathcal E(\lambda)\, d\lambda
	=\frac{-1}{4it} \int_0^\infty e^{-it\lambda^4} \mathcal E'(\lambda)\, d\lambda = \frac{-1}{16it} \int_0^\infty e^{-its} \frac{\mathcal E'(s^{\f14})}{s^{\f34}}\, ds.
	$$
	We then break the integral up into two pieces, on $[0,\frac{2\pi}{t}]$ and $[\frac{2\pi}{t},\infty)$.  For the first piece, we note that
	$$
	\bigg|\int_0^{2\pi/t} e^{-its} \frac{\mathcal E'(s^{\f14})}{s^{\f34}}\, ds \bigg|
	=\bigg| \int_0^{ \sqrt[4]{2\pi/t} } e^{-it\lambda^4} \mathcal E'(\lambda)\, d\lambda \bigg|\les \int_0^\infty \frac{|\mathcal E'(\lambda)|}{1+\lambda^4 t}\, d\lambda.
	$$
	On the second piece, we write
	$$
	\int_{2\pi/t}^\infty e^{-its} \frac{\mathcal E'(s^{\f14})}{s^{\f34}}\, ds=-\int_{2\pi/t}^\infty e^{-it(s-\frac{\pi}{t})} \frac{\mathcal E'(s^{\f14})}{s^{\f34}}\, ds
	=-\int_{\pi/t}^\infty e^{-its} \frac{\mathcal E'(\sqrt[4]{s+\frac{\pi}{t}} )}{(s+\frac{\pi}{t})^{\f34}}\, ds
	$$
	Thus, we need only control the contribution of
	$$
	\int_{\pi/t}^\infty e^{-its} \bigg(\frac{\mathcal E'(s )}{s^{\f34}} - \frac{\mathcal E'(\sqrt[4]{s+\frac{\pi}{t}} )}{(s+\frac{\pi}{t})^{\f34}} \bigg)\, ds
	$$
	We now consider
	\begin{multline*}
	\bigg|\frac{\mathcal E'(s^{\f14} )}{s^{\f34}} - \frac{\mathcal E'(\sqrt[4]{s+\frac{\pi}{t}} )}{(s+\frac{\pi}{t})^{\f34}} \bigg|
	=\bigg|\frac{\mathcal E'(s^{\f14} )-\mathcal E'(\sqrt[4]{s+\frac{\pi}{t}} )}{(s+\frac{\pi}{t})^{\f34}} +\mathcal E'(s^{\f14}) \bigg(\frac{1}{s^{\f34}}-\frac{1}{(s+\frac{\pi}{t})^{\f34}} \bigg) \bigg|\\
	\les \frac{|\mathcal E'(s^{\f14} )-\mathcal E'(\sqrt[4]{s+\frac{\pi}{t}}) |}{s^{\f34}}+\frac{|\mathcal E'(s^{\f14})|}{ts^{\f74}}.
	\end{multline*}
	The first summand is controlled by the second integral in the claim, while the second summand is controlled by the first integral.
	
\end{proof}

The oscillatory integral bound in Lemma~\ref{lem:ibpweight} is used to control  the error term in the expansion of $(M^{\pm}(\lambda))^{-1}$. We note that the $\lambda$ smallness in Lemma~\ref{regular} in $E^\pm(\lambda)$ is not optimal.  At the cost of further decay in $V$, one obtains further smallness in $\lambda$.

\begin{lemma}\label{lem:error wtd lips} 
	Let $ |V(x)| \les \la x \ra^{-4-4\ell-}$ for some $0<\ell <1$. For $t>2$, we have the bound
	$$
	\bigg|	
	\int_0^{\infty} e^{-it \lambda^4} \lambda^3 \chi(\lambda)  
	R^\pm (\lambda^4) vE^\pm (\lambda )vR^\pm (\lambda^4)	\, d \lambda \bigg| \les \frac{\la x\ra^{2\ell} \la y\ra^{2\ell} }{t^{1+}}.
	$$	
	
\end{lemma} 
\begin{proof}
We use Lemma~\ref{lem:ibpweight} for $\mathcal E (\lambda) = \chi(\lambda)  
	R^\pm (\lambda^4) vE^\pm (\lambda )vR^\pm (\lambda^4)$. To do that first we obtain the required bounds on $\mathcal E (\lambda)$.

Recall by Lemma~\ref{regular}, for $|v(x)| \les \la x \ra^{-2-2\ell-}$, we have $E(\lambda)=\widetilde{O}_1(\lambda^{\ell})$ and 
\begin{align}
 |\partial_\lambda E^{\pm}(b)-\partial_\lambda E^{\pm}(\lambda)| \les |b-\lambda|^{\ell} \lambda ^{-1-\ell^2+\f{\ell}{2}}  \label{Edifferbound}
\end{align}
as an absolutely bounded operator.
Moreover, taking $\alpha = \ell$ in  Corollary~\ref{cor:Error Lips} we obtain, for $\lambda \ll 1$
\begin{align}
& |R(\lambda^4)(x,x_1)| \les \lambda^{0-} \log|x-x_1|, \,\,\ |\partial_{\lambda}R(\lambda^4)| \les \frac{1}{\lambda} + \lambda^{-1+\ell} |x-x_1|^{\ell} , \label{Rweightbound}\\
& 
	|\partial_{\lambda}R(b^4)-\partial_{\lambda}R(\lambda^4)| \les \frac{1}{\lambda}+  |b-\lambda|^{\ell} \lambda ^{-1 - \ell^2+\f{\ell}{2}} |x-x_1|^{- \ell^2+\f{3\ell}{2}} \label{Rdifferbound}.
\end{align}
The fact that $E(\lambda)=\widetilde{O}_1(\lambda^{\ell})$ and  \eqref{Rweightbound} gives 
\begin{align}
 | \partial_{\lambda} \mathcal E (\lambda)| \les \chi(\lambda) \lambda^{-1+\ell-} \la x \ra^{\ell} \la y \ra^{\ell}.\label{errorder}
\end{align}
Here, for the spatial bound we write  $\log|x-y| = \log^{-} |x-y| + \log^{+} |x-y| $, and note that if $\Gamma$ is an absolutely bounded operator, 
\begin{multline} \label{spatial bound}
\||x-x_1|^{p} v \Gamma v |y-y_1|^{p}\|_{L^1 \rightarrow L^{\infty} } \les \la x\ra^{p} \la y \ra^{p} \| \la x_1\ra^{p} v\Gamma v \la y_1\ra^{p}\|_{L^2 \rightarrow L^2 }\\
 \les \la x\ra^{\min(p,0)} \la y \ra^{\min(p,0)}
\end{multline}
for any $ -2 < p $, provided $|v(x) | \les \la x \ra^{-2-p-}$.

Furthermore, taking $b= \lambda \sqrt[4]{1+\pi t^{-1} \lambda^{-4}}$ in \eqref{Edifferbound} and \eqref{Rdifferbound}, and noting that
$
\lambda \sqrt[4]{1+\pi t^{-1} \lambda^{-4}}-\lambda \approx (t\lambda^3)^{-1}
$
we obtain 
\begin{align}
 |\partial_{\lambda}\mathcal E(b)-\partial_{\lambda}\mathcal E(\lambda)| \les (t \lambda^3)^{-\ell} \lambda ^{-1 - \ell^2+\f{\ell}{2}} \la x \ra^{2\ell} \la y \ra^{2\ell} \label{errordif}. 
\end{align}

Using the bounds \eqref{errorder} and \eqref{errordif} in  Lemma~\ref{lem:ibpweight}, we have 
\begin{multline*}
\int_0^\infty \frac{|\mathcal E'(\lambda)|}{1+\lambda^4 t}\, d\lambda
	+ \int_{t^{-\f14}}^\infty \bigg| \mathcal E'(\lambda \sqrt[4]{1+\pi t^{-1} \lambda^{-4}})-\mathcal E'(\lambda) \bigg|\, d\lambda \\ 
	\les \int_{\lambda \ll t^{-\f14}}| \mathcal E'(\lambda)| + \frac{1}{t}\int_{\lambda \gtrsim t^{-\f14} } \frac{\mathcal E'(\lambda)}{\lambda^4} +  \frac{1}{t^{\ell}} \int_{t^{-\f14}}^\infty \lambda ^{-1 - \ell^2-\f{5\ell}{2}} \la x \ra^{2\ell} \la y \ra^{2\ell} d \lambda \les  \frac{\la x \ra^{2\ell} \la y \ra^{2\ell}} { t^{\min \{{\f \ell 4},  \frac{3\ell -2 \ell^2}{8}\}}}.
\end{multline*}
 That fact  that $ 3\ell -2 \ell^2 >0$ for $ 0<\ell <1$ finishes the proof.
\end{proof}

We are ready to prove Proposition~\ref{prop:S cancel}.

\begin{proof}[Proof of Proposition~\ref{prop:S cancel}] 
	Recall the expansion for the resolvent in Lemma~\ref{lem:R0 exp} . We have 
\begin{align*}
	\frac{R^\pm(\lambda^4)vSvR^\pm(\lambda^4)}{h^\pm(\lambda)} = \frac{(\widetilde g^\pm(\lambda))^2} {h^\pm(\lambda)} vSv+ \frac{\widetilde g^\pm(\lambda)} {h^\pm(\lambda)} [ G_1 vSv + vSv G_1]  + \frac{G_1 vSvG_1}{h^\pm(\lambda)} + \widetilde{E}_0 (\lambda) (x,y)  
\end{align*}
where
\begin{align}\label{SErrorbound}
&|\partial_{\lambda} \widetilde{E}_0 (\lambda) (x,y)| \les \lambda^{-1+\ell-} \la x \ra^{\ell} \la y \ra^{\ell} , \\
&|\partial_{\lambda} \widetilde{E}_0 (b) -\partial_{\lambda} \widetilde{E}_0 (\lambda)|\les  (b- \lambda)^{\alpha} \lambda^{-1 - \ell^2+\f{\ell}{2}} \la x \ra^{\f{3\ell}{2} - \ell^2}\la y \ra^{\f{3\ell}{2} - \ell^2}.\label{E-dif}
\end{align}

Notice that \eqref{SErrorbound} is an immediate consequence of Lemma~\ref{lem:R0 exp} and \eqref{spatial bound}, whereas we need to validate \eqref{E-dif}. By symmetry it will be enough to analyze $h_{\pm}^{-1}(\lambda)R^\pm(\lambda^4)(x,x_1) E^{\pm}(\lambda)(y,y_1)$.
Recall that by Lemma~\ref{lem:R0 exp}, we have for $0 < \lambda \ll1$
\begin{align}
&\label{REfirst} |\partial_{\lambda} [h_{\pm}^{-1}R^\pm E^{\pm}](\lambda) | \les \lambda^{-1+\ell-} k(x,x_1)|x-x_1|^{\ell}|y-y_1|^{\ell}, \\ 
&|\partial^2_{\lambda} [h_{\pm}^{-1}R^\pm E^{\pm}](\lambda) | \les \lambda^{-\f 32-} k(x,x_1)|x-x_1|^{\f 12}|y-y_1|^{\f 12}. \label{REsecond}
\end{align}
Note that \eqref{REsecond}, and the Mean Value Theorem with $0<b<\lambda$ gives 
$$
|\partial_{\lambda} [h_{\pm}^{-1}R^\pm E^{\pm}](b)- \partial_{\lambda} [h_{\pm}^{-1}R^\pm E^{\pm}](\lambda) |  
\les (b- \lambda) \lambda^{-\f 32-} k(x,x_1)|x-x_1|^{\f 12}|y-y_1|^{\f 12}. 
$$
Moreover, by \eqref{REfirst} we have 
$$
|\partial_{\lambda} [h_{\pm}^{-1}R^\pm E^{\pm}](b)- \partial_{\lambda} [h_{\pm}^{-1}R^\pm E^{\pm}](\lambda) |  
\les \lambda^{-1+\ell-} k(x,x_1)|x-x_1|^{\ell}|y-y_1|^{\ell}.
$$
Interpolating these two inequality we obtain
\begin{multline}
|\partial_{\lambda} [h_{\pm}^{-1}R^\pm(\lambda^4) E^{\pm}](b)- \partial_{\lambda} [h_{\pm}^{-1}R^\pm(\lambda^4) E^{\pm}](\lambda) |  \\ 
\les (b- \lambda)^{\alpha} \lambda^{(-1+\ell-)(1-\alpha)- {\f {3\alpha} 2}-}k(x,x_1)|x-x_1|^{\ell(1-\alpha)+\f{\alpha}{2}} |y-y_1|^{\ell(1-\alpha)+\f{\alpha}{2}} \nn
\end{multline}
Letting, $\ell= \alpha$, we obtain \eqref{E-dif}  provided $ v(x) \les \la x \ra^{-2-2\ell-}$. Note that, the bounds  \eqref{SErrorbound} and \eqref{E-dif} are exact same bounds that we have in  \eqref{errorder} and \eqref{errordif} respectively (letting $b= \lambda \sqrt[4]{1+\pi t^{-1} \lambda^{-4}}$). Therefore, Lemma~\ref{lem:ibpweight} for $\mathcal{E}= \widetilde{E}_0$ establishes the contribution of $\widetilde{E}_0$ as $t^{-1-} \la x \ra^{2\ell} \la y \ra^{2\ell}$.

For the other terms we note that  $h^\pm(\lambda)= \|V\|_{1} \widetilde g_1^\pm(\lambda)+c$. Therefore, 
\begin{align*}
&\frac{(\widetilde g_1^{+}(\lambda))^2} {h^{+}(\lambda)} -\frac{(\widetilde g_1^{-}(\lambda))^2} {h^{-}(\lambda)} = \frac{ i \Im{z_1}} { \|V\|_1} + \widetilde O_2\big((\log \lambda)^{-2} \big) \\
&\frac{\widetilde g_1^+(\lambda)} {h^+(\lambda)} - \frac{\widetilde g_1^-(\lambda)} {h^-(\lambda)} =  \widetilde   O_2\big((\log \lambda)^{-2} \big) ,\,\, \,\,\, \frac{1} {h^+(\lambda)} - \frac{1} {h^-(\lambda)} =   \widetilde O_2\big((\log \lambda)^{-2} \big) . 
\end{align*}
Moreover,  by the absolutely boundedness of $S$, and the decay assumption on $v$ we have $| G_1 vSv + vSv G_1 + G_1 vSv G_1| \les (1+ \log^{+} |x| ) (1+ \log^{+} |y| )$. Hence, 
\begin{multline*}
\frac{R^{+}(\lambda^4)SR^{+}(\lambda^4)}{h^{+}(\lambda)} -\frac{R^{-}(\lambda^4)SR^-(\lambda^4)}{h^-(\lambda)} -\widetilde{E}_0 (\lambda)
= \frac{ i \Im{z_1}} { \|V\|_1} vSv \\ + \widetilde O_2\big((\log \lambda)^{-2} \big)(1+ \log^{+} |x| ) (1+ \log^{+} |y| ) . 
\end{multline*}
The second summand on the right hand side is bounded by $(t \log^2 t )^{-1} (1+ \log^{+} |x| ) (1+ \log^{+} |y| )$ using Lemma~\ref{lem:log osc}. For the first summand, first notice that by $Qv=0$, we have $vSv= vPv=\|V\|_1 $. Hence, by integration by parts the  first summand contributions
$$
	\int_{\R^8}   \bigg(  \frac{\Im(z_1)}{\|V\|_1 t}+O(t^{-1-})\bigg) [vSv](x_1 ,y_1)\,  dx_1 dy_1    =\frac{\Im(z_1)}{4 t}+O(t^{-1-}).
	$$
\end{proof}

Lastly, we consider the contribution of the $QD_0Q$.  

\begin{lemma} \label{lem:QDQweight}Let $ |V(x)| \les \la x \ra^{-4-4\ell-}$ for some $1> \ell > 0$. For $t>2$, we have the bound
	\begin{multline*}
	\bigg|	
	\int_0^{\infty} e^{-it \lambda^4} \lambda^3 \chi(\lambda)  
	[R^+(\lambda^4) vQD_0Q vR^+ (\lambda^4)- R^-(\lambda^4) vQD_0Q vR^- (\lambda^4)]	\, d \lambda \bigg|  
	\les \frac{\la x\ra^{2\ell} \la y\ra^{2\ell} }{t^{1+}} .
	\end{multline*}

\end{lemma}

\begin{proof}

By Lemma~\ref{lem:R0 exp} and the orthogonality property $Qv=0$, it suffices to consider
	\begin{multline*}
	R^+(\lambda^4)vQD_0Qv[R^+(\lambda^4)-R^-(\lambda^4)]\\
	=[ G_1(x,x_1)+E_0(\lambda)(x,x_1)]vQD_0Qv(x_1,y_1)E_0^\pm(\lambda)(y_1,y). 
	\end{multline*}
Let $\mathcal{E}(\lambda)=[ G_1(x,x_1)+E_0(\lambda)(x,x_1)]vQD_0Qv(x_1,y_1)E_0^\pm(\lambda)(y_1,y)$, then we have 
\begin{align*}
 | \partial_{\lambda} \mathcal E (\lambda)| \les \chi(\lambda) \lambda^{-1+\ell} \la x \ra^{\ell} \la y \ra^{\ell}, \,\,\,\
& |\partial_{\lambda}\mathcal E(b)-\partial_{\lambda}\mathcal E(\lambda)| \les (b- \lambda)^{\alpha} \lambda^{-1 - \ell^2+\f{\ell}{2}} \la x \ra^{\f{3\ell}{2} - \ell^2}\la y \ra^{\f{3\ell}{2} - \ell^2} . 
\end{align*}
Note that, again the first inequality is consequence of Lemma~\ref{lem:R0 exp} and \eqref{spatial bound}, whereas  the following bounds gives the second inequality by means of interpolation as in the proof of  Proposition~\ref{prop:S cancel}. 
\begin{align*}
& |\partial_{\lambda}\big([ G_1(x,x_1)+E_0(\lambda)(x,x_1)]E_0^\pm(\lambda)(y_1,y)\big)| \les \lambda^{-1+\ell} k(x,x_1)|x-x_1|^{\ell}|y-y_1|^{\ell} \\
&|\partial^2_{\lambda} \big([ G_1(x,x_1)+E_0(\lambda)(x,x_1)]E_0^\pm(\lambda)(y_1,y)\big)| \les \lambda^{-\f 32+\ell} k(x,x_1)|x-x_1|^{\f 12}|y-y_1|^{\f 12}
\end{align*}
Therefore, letting $b= \lambda \sqrt[4]{1+\pi t^{-1} \lambda^{-4}}$, Lemma~\ref{lem:ibpweight} bounds the contribution of this term by $t^{-1-} \la x \ra^{2\ell} \la y \ra^{2\ell}$. 
	
\end{proof}
We are now ready to prove the main proposition.

\begin{proof}[Proof of Proposition~\ref{prop:main weight} ]
	
	By the symmetric resolvent identity, \eqref{eqn:Rv expansion}, we need to control the contribution of
	$$
		h^\pm(\lambda)^{-1} S +QD_0Q+E^\pm(\lambda)
	$$
	to the Stone's formula.  The required bounds are established in Proposition~\ref{prop:S cancel}, Lemma~\ref{lem:QDQweight} and Lemma~\ref{lem:error wtd lips}  respectively, where the exact cancellation of the order $t^{-1}$ terms follows from Proposition~\ref{prop:S cancel} and Lemma~\ref{lem:free est}.

\end{proof}

\section{Low energy dispersive bounds when zero is not regular}\label{sec:nonreg}

In this section we consider the remaining claims in Theorem~\ref{thm:main} concerning the dispersive bounds when zero is not regular.  We provide a detailed analysis of the effect of each type of resonance on the asymptotic behavior of the spectral measure and detail the dynamical consequences on the dispersive bounds.  We divide the section into three subsections, consider resonances of each type separately except for resonances of the third and fourth kind for which we provide a unified treatment.

\subsection{Resonance of the first kind}

In this section we analyze the resonance the dispersive estimates when there is a resonance of the first kind. Recall that in this case,  $T=U + v G_1 v$ is not invertible on $QL^2$. Therefore, to utilize Lemma~\ref{JNlemma} define $B= S_1-S_1(M+S_1)^{-1}S_1$ and seek to invert $B$ . Using the expansion \eqref{M-1exp0} for $(M+S_1)^{-1}$, we have 
$$B(\lambda) = - h^{\pm}(\lambda)^{-1}  S_1 S S_1 + \widetilde O_1(\lambda^k) =  -h^{\pm}(\lambda)^{-1}(\lambda) S_1 T PT S_1 + \widetilde O_1(\lambda^k),\,\,\  0 < k \leq 2.$$

Notice that we are able to invert $B (\lambda)$ when there is a resonance of the first kind due to invertibility of $S_1 T PT S_1$ on $S_1L^2$, and a Neumann series computation. More care is required if there is a resonance of second, third or fourth kind.  These will be considered in subsequent sections.
By a Neumann series computation, provided $|v(x)| \les \la x \ra ^{-k-2-} $, we obtain 
$$ 
	B^{-1} (\lambda) = - h^{\pm} (\lambda) D_1 [ I +  \widetilde O_1(\lambda^k)]^{-1}=- h^{\pm} (\lambda) D_1  + \widetilde O_1(\lambda^k) , \,\,\  0 < k \leq 2. 
$$

Using $B^{-1} (\lambda)$ in \eqref{Minversegeneral} from Lemma~\ref{JNlemma} we obtain the following proposition. Similar expansion is obtained in \cite{eg2}, Lemma~2.5 to Proposition~2.7  in the analysis of two dimensional Schr\"odinger evolution.  For the sake of brevity, we refer the reader to \cite{eg2}.
\begin{prop}\label{prop:Minv firstkind} 
	Let $|v(x)| \les \la x \ra ^{-k-2-} $. If there is a resonance of the first kind, we have the expansion 	\begin{multline}\label{eqn:Mive first}
		M^{\pm}(\lambda)^{-1}=-h_\pm(\lambda) D_1  - SS_1D_1S_1-S_1D_1S_1S \\
		-h_\pm(\lambda)^{-1} (SS_1D_1S_1S+S)+QD_0Q+\widetilde O_1(\lambda^k)
	\end{multline}
	for $0 < k\leq 2$.

\end{prop}

The following is the main result in this section. 

\begin{prop} \label{prop:firstkind} Let $|V(x)|\les \la x\ra^{-4-}$ and suppose that there is a resonance of the first kind at zero.  Then, we have
	\begin{align*}
		\sup_{x,y} \left|\int_0^{\infty} e^{-it \lambda^4} \lambda^3 \chi(\lambda) [R^{+}_V-R^-_V](\lambda)(x,y) d \lambda \right| 
		\les \la t \ra ^{-1}.
	\end{align*}
	\end{prop}

As in the previous sections, to prove Proposition~\ref{prop:firstkind} we need to control the contribution of every term in 
$M^{\pm}(\lambda)^{-1}$ to the Stone's formula. Note that the last three summands in the expansion of $M^{\pm}(\lambda)^{-1}$ in \eqref{eqn:Mive first} are analogous to summands in the expansion when zero is regular.  In particular, the contribution of these operators can be obtained similarly to  Lemma~\ref{prop:QD0Q bound}, Lemma~\ref{prop:S bound} and Lemma~\ref{lem:Error term} respectively.  Accordingly, Proposition~\ref{prop:S_1T_1^{-1}S_1 bound} and Lemma~\ref{lem:SS1 bound} below, where we consider the first summands in \eqref{eqn:Mive first}, suffices to complete the proof of Proposition~\ref{prop:firstkind}.

\begin{prop}\label{prop:S_1T_1^{-1}S_1 bound} 
	We have the bound
	\begin{multline*}
	\sup_{x,y} \bigg|	
	\int_0^{\infty} e^{-it \lambda^4} \lambda^3 \chi(\lambda)  
	\big[h^{+}(\lambda) R^+(\lambda^4)vD_1vR^+(\lambda^4)\\ -h^-(\lambda)R^-(\lambda^4)vD_1vR^-(\lambda^4)](x,y)
	\big] 	 d \lambda \bigg| \les \frac{1}{\la t\ra}.
	\end{multline*}
\end{prop}

In the proof we again utilize from the algebraic fact \eqref{alg fact} and consider the following two terms by symmetry
\begin{align} \label{termsariseS}
	[h^{+} - h^{-} ](\lambda) R^{\pm}(\lambda^4)vS_1D_1S_1vR^{\pm}(\lambda^4), \,\ h^{\pm}(\lambda)[  R^+- R^-](\lambda^4) vS_1D_1S_1vR^{\pm}(\lambda^4). 
	\end{align}
Recall by Lemma~\ref{lem:free est}, one has 
$$ 
	[R^{+}-R^{-}](\lambda^4)(x,y)= \frac{i}{2\lambda^2} \frac{\lambda }{4\pi |x-y|} J_1(\lambda |x-y|) .
$$
We have the following lemma and its corollary which we use together with the orthogonality property $Qv=0$ to estimate the terms arising in interaction with $R^{+}-R^{-}$.  

Recalling \eqref{eqn:R0 hi exp} and \eqref{eqn:A def}, we define
\begin{align}
	\widetilde G^\pm (\lambda, p, q):=e^{\pm  i \lambda p} \tilde{ \omega}_{\pm}  ( \lambda p) + e^{-\lambda p} \tilde{ \omega}_{+}  ( \lambda p)-e^{\pm  i \lambda q} \tilde{ \omega}_{\pm}  ( \lambda q) - e^{-\lambda q} \tilde{ \omega}_{+}  ( \lambda q)
\end{align}

\begin{lemma}\label{lem:tildeG gain}
	
	For any $0\leq \tau\leq 1$, we have the bounds
	\begin{align*}
		|\widetilde G^\pm (\lambda, p, q)|\les (\lambda |p-q|)^{\tau},
		\qquad
		|\partial_\lambda \widetilde G^\pm (\lambda, p, q)|\les \lambda^{\tau-1} |p-q|^{\tau}
	\end{align*}

\end{lemma}

\begin{proof}
	
	We note that, we have $\tilde{ \omega}_{\pm}  ( s)=\widetilde O((1+s)^{-\f32})$ due to the support condition $s\gtrsim 1$.  We consider the oscillatory terms, those with the exponential decay follow similarly.
	We first note the trivial bound
	\begin{align}
		|e^{\pm  i \lambda p} \tilde{ \omega}_{\pm}  ( \lambda p) 
		-e^{\pm  i \lambda q} \tilde{ \omega}_{\pm}  ( \lambda q)|\les 1.
	\end{align}
	Additionally, we note
	\begin{multline}
		|e^{\pm  i \lambda p} \tilde{ \omega}_{\pm}  ( \lambda p) 
		-e^{\pm  i \lambda q} \tilde{ \omega}_{\pm}  ( \lambda q)|
		=\bigg| \int_{\lambda p}^{\lambda q} \partial_s \big( e^{is}\tilde \omega_\pm(s) \big)\, ds \bigg| \\
		=\bigg| \int_{\lambda p}^{\lambda q} \big( ie^{is}\tilde \omega_\pm(s)+e^{is}\tilde \omega^{'}_\pm(s) \big)\, ds \bigg|
		\les \bigg| \int_{\lambda p}^{\lambda q} (1+s)^{-\f32} \, ds \bigg| \les (\lambda|p-q| ),
	\end{multline}
	since $(1+s)^{-\f32}\les 1$.  Interpolating between these bounds suffices to prove the claim for $\widetilde G$, since the terms with exponential decay satisfy the same bounds used here.
	
	For the derivative, we again consider only the oscillatory terms.  First, 
	$$
		\partial_\lambda e^{\pm  i \lambda p} \tilde{ \omega}_{\pm}  ( \lambda p) =\frac{1}{\lambda }\big[ \pm i \lambda p   e^{\pm  i \lambda p} \tilde{ \omega}_{\pm}  ( \lambda p) +\lambda p e^{\pm  i \lambda p} \tilde{ \omega}_{\pm}^{'}  ( \lambda p) \big].
	$$
	So that, 
	\begin{align}
		\bigg|\big( \partial_\lambda e^{\pm  i \lambda p} \tilde{ \omega}_{\pm}  ( \lambda p) \big)
		-\partial_\lambda \big(e^{\pm  i \lambda q} \tilde{ \omega}_{\pm}  ( \lambda q)\big) \bigg|\les \lambda^{-1}.
	\end{align}
	Furthermore,
	\begin{multline}
		\bigg|\big( \partial_\lambda e^{\pm  i \lambda p} \tilde{ \omega}_{\pm}  ( \lambda p) \big)
		-\partial_\lambda \big(e^{\pm  i \lambda q} \tilde{ \omega}_{\pm}  ( \lambda q)\big) \bigg|
		=\frac{1}{\lambda } \bigg| \int_{\lambda p}^{\lambda q} \partial_s \big(i s e^{is}\tilde \omega_\pm(s)+s e^{is}\tilde \omega_\pm ^{'}(s) \big)\, ds \bigg| \\
		\les  \frac{1}{\lambda} \bigg| \int_{\lambda p}^{\lambda q} (1+s)^{-\f12} \, ds \bigg| \les  |p-q|.
	\end{multline}	
	Interpolating between the two bounds proves the assertion.  As before, the same proof works for the exponentially decaying term.
	
\end{proof}

\begin{corollary} \label{corJ1}
 Let 
$$ G(\lambda, p ,q ):= \frac{1}{8 \pi} \left[ \frac{J_1(\lambda p )}{\lambda p}  - \frac{J_1(\lambda q )}{ \lambda q} \right]. $$
For any $0\leq \tau\leq 1$, we have the bounds
	\begin{align*}
		| G (\lambda, p, q)|\les (\lambda |p-q|)^{\tau},
		\qquad
		|\partial_\lambda G (\lambda, p, q)|\les \lambda^{\tau-1} |p-q|^{\tau}.
	\end{align*}
\end{corollary}
\begin{proof} Using the expansion for Bessels function, we have 
 \begin{align}
	\frac{J_1(\lambda p)}{8 \pi \lambda p} = \chi(\lambda p) [\Im z_1+  O((\lambda p)^{2}) ]+\tilde{\chi}(\lambda p) e^{it\lambda} \tilde{\omega}  ( \lambda p)
\end{align}
By Lemma~\ref{lem:tildeG gain}, we only consider when $\lambda p$ and $\lambda q$ are small. Note that by the Mean Value Theorem one has
\begin{align*}
	\bigg|\chi(\lambda p)\frac{J_1(\lambda p)}{8 \pi \lambda p}  -\chi(\lambda p)\frac{J_1(\lambda p)}{8 \pi \lambda p}\bigg| \les \lambda |p-q| \sup_s \bigg| \partial_s \bigg( \chi(s)\frac{J_1(s)}{8 \pi s}  \bigg)\bigg|,
\end{align*} 
where $s$ lies between $\lambda p$ and $\lambda q$.  By the power series expansion for $J_1$, see \cite{AS}, one has
$$
	\chi(s)\frac{J_1(s)}{s}=c_0+\sum_{k=1}^{\infty} c_k s^k
$$
for some constants $c_k$.  Hence, one can differentiate and obtain a bounded function for $|s|\ll 1$.
Interpolating between this and the trivial bound  $|G(\lambda, p ,q )| \les 1$ suffices to prove the claim for $G$. Moreover, 
for the derivative we have for $s=\lambda p$
$$
	\partial_\lambda \bigg(\chi(s)\frac{J_1(s)}{s}\bigg)=\sum_{k=1}^{\infty} pkc_k s^{k-1}=\frac{1}{\lambda}\sum_{k=1}^{\infty} kc_k s^k
$$
From here, one can see the bound of $|\partial_\lambda G (\lambda, p, q)|\les \lambda^{-1}$ holds.  We interpolate with the following bound
\begin{multline*}
	\bigg|\partial_\lambda \bigg(\chi(\lambda p)\frac{J_1(\lambda p)}{8 \pi \lambda p}\bigg)  -\partial_\lambda \bigg(\chi(\lambda p)\frac{J_1(\lambda p)}{8 \pi \lambda p}\bigg) \bigg|\\ 
	\les  |p-q| \sup_s \bigg| \partial_s\bigg[ \lambda \partial_\lambda \bigg(\chi(s)\frac{J_1(s)}{s}\bigg) \bigg]\bigg| \les  |p-q|,
\end{multline*} 
to prove the desired statement.

\end{proof}

We are now ready to prove the main proposition.

\begin{proof} [Proof of Propostion~\ref{prop:S_1T_1^{-1}S_1 bound} ] Note that the first term in \eqref{termsariseS} is similar to the operator that we considered in Lemma~\ref{prop:QD0Q bound}. Since $S_1 \leq Q$, the orthogonality  $S_1P=0$ allows us to exchange  $R^{\pm}$ on both sides of $vS_1D_1S_1v$ with $H^{\pm}$, see \eqref{eqn:H def}. Therefore, we prove the statement only for the second term in \eqref{termsariseS}. Using the  orthogonality we exchange $R^{+} - R^{-}$ with $G$ from Corollary~\ref{corJ1} and  consider the following integral. 
\begin{align} \label{HSG}
\int_0^{\infty} e^{-it\lambda^4}\lambda^3 \chi(\lambda) h^{\pm}(\lambda)  H^{\pm}(\lambda,y,y_1) vS_1D_1S_1v G (\lambda, |x-x_1|,\la x\ra )  \ d \lambda.
\end{align}
Recall that we have $H^{\pm}(\lambda,y,y_1) \les k(y,y_1)$. Moreover, using the bounds for G from Corollary~\ref{corJ1}, we have 
	$$ | h^{\pm} G(\lambda, |x-x_1|,\la x\ra ) | \les \lambda^{\tau-} \la x_1 \ra ^{\tau}, \quad |\partial_\lambda \{ h^{\pm} G(\lambda, |x-x_1|,\la x\ra ) \} | \les  \lambda^{\tau-1-} \la x_1 \ra ^{\tau} .$$

Therefore, the integral \eqref{HSG} is bounded. For the time decay, by integration by parts, we need to show the following integral is bounded.
	\begin{multline} \label{HG}
	\int_0^{\infty} | \partial_\lambda \{ h^{\pm}(\lambda) \chi(\lambda) H^{\pm}(\lambda,y,y_1) vS_1D_1S_1v G (\lambda, |x-x_1|,\la x\ra ) \} | \ d \lambda  \\ 
	 \les  \int_0^{\infty} |h^{\pm}(\lambda) \partial_\lambda \{ \chi(\lambda) H^{\pm}(\lambda,y,y_1) \}  vS_1D_1S_1v G(\lambda, |x-x_1|,\la x\ra ) | \ d \lambda\\ +  \int_0^{\infty} | \ \chi(\lambda) H^{\pm}(\lambda,y,y_1)  vS_1D_1S_1v \partial_\lambda \{ h^{\pm}(\lambda) G(\lambda, |x-x_1|,\la x\ra ) \}| \ d \lambda .
	\end{multline}
	
	Hence picking $\tau = \f12$, the first integral  in \eqref{HG} can be estimated as in \eqref{HH}, and the second integral in \eqref{HG} as follows,
	\begin{align*}
		\int_{\R^8} \int_0^{\infty} |h^{\pm}(\lambda) \chi(\lambda) H^{\pm}(\lambda,y,y_1) [vS_1D_1S_1v] \partial_\lambda \{ h^{\pm}(\lambda) G(\lambda, |x-x_1|,\la x\ra ) \}| \ d \lambda dx_1 dy_1 \\ 
		\les \int_{\R^8}  \int_0^{\lambda_1} \lambda^{-\f12-} k(y,y_1) [vS_1D_1S_1v] \la x_1 \ra^{\f12} \ d \lambda dx_1 dy_1 \les 1.  
	\end{align*}
As usual, the boundedness of the spatial integrals follows from the absolute boundedness of $D_1$ and an argument as in \eqref{spat bound}.
	
\end{proof}

We now turn to the contribution of the $\lambda$ independent, finite rank operator(s) $\Gamma:=SS_1D_1S_1-S_1D_1S_1S$.

\begin{lemma}\label{lem:SS1 bound} 
	We have the bound
	\begin{align*}
	\sup_{x,y}\bigg|	
	\int_0^{\infty} e^{-it \lambda^4} \lambda^3 \chi(\lambda)  
	\big[ R^+v\Gamma vR^+ -R^-v\Gamma vR^-](\lambda^4)(x,y)
	\big] 	 d \lambda \bigg| \les \la t\ra^{-1} .
	\end{align*}
\end{lemma}

\begin{proof}
	
	Note that the +/- difference leads us to bound the contribution of
	$$
		R^-(\lambda^4)\Gamma (R^+-R^-)(\lambda^4).
	$$
	By symmetry, we consider only this case.
	We begin by considering the `low-low' case.  The worst case is when the $S_1$ projection is on the left.  Then, we need to bound an integral of the form
	\begin{multline*}
		\bigg| \int_0^\infty e^{-it\lambda^4}\lambda^3 \chi(\lambda)F(\lambda, y_1,y) [1+\widetilde O_2((\lambda |x-x_1|)^2)] \, d\lambda\bigg| \\
		\les \frac{1}{t} \bigg[ \int_0^{2\lambda_1} |\partial_\lambda F(\lambda, y,y_1)|\, d\lambda+k(y,y_1)\int_0^{|x-x_1|^{-1}}\lambda |x-x_1|^2\, d\lambda	\bigg] \les \frac{k(y,y_1)}{t}.
	\end{multline*}
	This suffices for the desired bound.
	
	For the `high-low' interaction, the worst case is when $S_1$ is on the right whose resolvent is supported on large argument, and we may replace $R^\pm$ with $\widetilde G(\lambda, |x-x_1|,\la x\ra)$.  In this case, we note that
	$$
	R^-(\lambda^4)(y,y_1)\chi(\lambda|y-y_1|)= O_1(\log(\lambda |y-y_1|)).
	$$
	Accordingly, we seek to bound
	\begin{multline*}
		\bigg| \int_0^\infty e^{-it\lambda^4}\lambda^3 \chi(\lambda) \chi(\lambda r_1) \log (\lambda r_1) \widetilde G(\lambda, |x-x_1|,\la x\ra)\, d\lambda\bigg| \\
		\les \frac{1}{t}\int_0^\infty \bigg| \chi(\lambda) \chi(\lambda r_1) \partial_\lambda \big(\log (\lambda r_1) \widetilde G(\lambda, |x-x_1|,\la x\ra)\big) 	\bigg| \, d\lambda \\
		\les \frac{1}{t} \int_{0}^\infty \lambda^{\tau-1-} (1+|y-y_1|^{0-}) \la x_1 \ra^{\tau}  d\lambda
	\end{multline*}
	This suffices to prove our desired bound provided we select $\tau=0+$.  The spatial bounds follow by the absolute boundedness of $\Gamma$.  
	
	The high-high bound follows similarly by using $\tau=0+$ in Lemma~\ref{lem:tildeG gain}. The boundedness of the  contribution of $R^-(\lambda^4)\Gamma (R^+-R^-)(\lambda^4)$ in each case is clear. 
	
\end{proof}

\subsection{Resonance of the Second Kind}
In this section, we analyze the evolution when there is a resonance of the second kind.  We develop expansions for the operator $B(\lambda)^{-1}$ and consequently for the spectral measure in this case, which we then use to prove the dispersive bounds. We emphasize that this type of resonance  has no  analogue   in the analysis of Schr\"odinger operator.  Both the characterization in terms of solutions of $H\psi=0$ and the effect on the time decay are new to the fourth order equation. Our main result is the following.
\begin{prop}\label{prop:2nd kind} 
	
	Let $|V(x)|\les \la x \ra^{-12-}$.
	In the case of a resonance of the second kind, we have
	$$
	\| e^{-itH}P_{ac}(H)\chi(H)\|_{L^1\to L^\infty} \les \la t\ra^{-\f12}.
	$$
	Furthermore, there is a finite rank operator $F_t$ satisfying $\|F_t\|_{1\to \infty} \les \la t\ra^{-\f12}$, so that
	$$
	\| e^{-itH}P_{ac}(H)\chi(H)-F_t\|_{L^{1,2+}\to L^{\infty,-2-}} \les \la t\ra^{-1}.
	$$

\end{prop}
We have an explicit representation of the operator $F_t$ given in \eqref{eqn:Ft explicit}.
Our time decay bounds follow by finding a detailed expansion for $(M^{\pm}(\lambda))^{-1}$.

\begin{prop}  \label{prop: B-1 exp second kind}
	Let $| V(x) | \les \la x \ra ^{-12-2\ell-}$.
	In the case of a resonance of the second kind,
	\begin{multline*}
		B_\pm^{-1}(\lambda)=\frac{D_2}{c^+\lambda^2}-\frac{\widetilde g_3^\pm(\lambda)}{(c^+)^2\lambda^4}D_2vG_4vD_2-\frac{1}{(c^+)^2}D_2vG_5vD_2\\ -h^\pm(\lambda) [D_1 + D_2\Gamma_0 D_2]
		+h^\pm(\lambda) \bigg(1+\frac{1}{h^\pm(\lambda)} \bigg)^2[D_2\Gamma_1+\Gamma_1D_2]+\widetilde O_1(\lambda^{\ell-})
	\end{multline*}
	where $\Gamma_0$, $\Gamma_1$ are absolutely bounded operators.

\end{prop}

\begin{proof}

We will do the `+' case.  The expansions for the `-' are analogous.  We note carefully where any consequential differences occur.

Recall \eqref{Mexp0}, we have $M(\lambda)=A(\lambda)+M_0(\lambda)$ where $A(\lambda)$ is as in Lemma~\ref{regular} and 
$$
	M^{+}_0(\lambda)=c^{+}\lambda^2 vG_2v+\tilde{g}^{+}_3(\lambda)vG_4v+\lambda^4 vG_5v+\widetilde O_1(\lambda^{4+\ell})
$$
where we may take any $0<\ell\leq 2$, at the cost of decay on $v$.  We now write
\begin{align*}
	(M(\lambda)+S_1)^{-1}&=A^{-1}(\lambda) [1+M_0(\lambda)A^{-1}(\lambda)]^{-1}\\
	&=A^{-1}(\lambda)-A^{-1}(\lambda)M_0(\lambda)A^{-1}(\lambda)
	[1+M_0(\lambda)A^{-1}(\lambda)]^{-1}
\end{align*}
Using Lemma~\ref{JNlemma}, we have
\begin{align}\label{eqn:B2nd}
	B(\lambda)=S_1-S_1[h(\lambda)^{-1}S+QD_0Q]S_1+E(\lambda)
\end{align}
with
\begin{align}\label{eqn:B error}
	E(\lambda)=S_1 A^{-1}M_0A^{-1}(\lambda)S_1-S_1A^{-1}(\lambda)[M_0A^{-1}(\lambda)]^2[1+M_0A^{-1}(\lambda)]^{-1}S_1
\end{align}
Since $S_1D_0=D_0S_1=S_1$ and $S_1SS_1=S_1TPTS_1$, we have
$$
	B(\lambda)=-h(\lambda)^{-1}S_1TPTS_1+E(\lambda)
	:=-h(\lambda)^{-1}B_1(\lambda).
$$
We write $B_1=T_1-h(\lambda)E(\lambda)$ and let $S_2$ be the Riesz projection onto the kernel of $T_1=S_1TPTS_1$ and $D_1=(T_1+S_2)^{-1}$.  We then work to compute
\begin{align}\label{eqn:B1+S2inv}
	(B_1(\lambda)+S_2)^{-1}&=D_1[1-h(\lambda)E(\lambda)D_1]^{-1}
	=D_1+E_1(\lambda),
\end{align}
where
\begin{align}\label{eqn:B1error1}
	E_1(\lambda)=D_1h(\lambda)E(\lambda)D_1+D_1[h(\lambda)E(\lambda)D_1]^2[1-h(\lambda)E(\lambda)D_1]^{-1}.
\end{align}
We use Lemma~\ref{JNlemma} to invert $B_1(\lambda)$ and as such we need to compute and invert $B_2=S_2-S_2(B_1+S_2)^{-1}S_2$.  Since $S_2D_1=D_1S_2=S_2$, we see
\begin{align*}
	B_2=-S_2E_1(\lambda)S_2
	=-S_2h(\lambda)E(\lambda)S_2 -S_2[h(\lambda)E(\lambda)D_1]^2[1-h(\lambda)E(\lambda)D_1]^{-1}S_2
\end{align*}
Since we have $S_2TP=PTS_2=0$, it follows that $S_2A^{-1}(\lambda)=A^{-1}(\lambda)S_2=S_2$.  Hence, we see
\begin{align*}
	S_2h(\lambda)E(\lambda)S_2&=S_2h(\lambda)M^{+}_0(\lambda)S_2-S_2h(\lambda) (M^{+}_0A^{-1}(\lambda))^2 (1+M^{+}_0A^{-1}(\lambda))^{-1}S_2.
\end{align*}
We note that (with $c^+:=-\frac{ia_2\pi}{2}$, for the `-' case we have $c^-=c^+-i\Im(z_2)$)
\begin{multline*}
	S_2h(\lambda)M_0(\lambda)S_2
	=c^+\lambda^2h(\lambda) S_2vG_2vS_2+h(\lambda)\tilde g^{+}_3(\lambda)S_2vG_4vS_2\\
	+\lambda^4h(\lambda)S_2vG_5vS_2+\widetilde O_1(\lambda^{4+\ell}),
\end{multline*}
and
$$
	S_2(h(\lambda)E(\lambda)D_1)^2=(c^+)^2(h(\lambda))^2\lambda^4 S_2vG_2vD_1vG_2vS_2+\widetilde E_6(\lambda)
$$
where $\widetilde E_6(\lambda)=\widetilde O_1(\lambda^{4+\ell})$.
Now, with $\Gamma_0:= S_2vG_2vD_1vG_2vS_2$, we see
\begin{multline}\label{eqn:B2 simple}
	B_2=-c^+\lambda^2 h(\lambda)\bigg[ S_2vG_2vS_2+\frac{\tilde g^{+}_3(\lambda)}{c^+\lambda^2} S_2vG_4vS_2+ c^+ h(\lambda) \lambda^2\Gamma_0 \\ 
	+\frac{\lambda^2}{c^+} S_2vG_5vS_2+\widetilde E_2(\lambda)	\bigg]
\end{multline}
with $\widetilde E_2(\lambda)=\widetilde O_1(\lambda^{2+\ell})$, $0<\ell<2$.
 In the case of a resonance of the second kind, $S_2vG_2vS_2$ is invertible.  We write $D_2:=(S_2vG_2vS_2)^{-1}$ and compute
\begin{align}\label{eqn:B2inv}
	B_2(\lambda)^{-1}=-\frac{1}{c^+\lambda^2 h(\lambda)}D_2\bigg[
	1+\frac{\tilde g_3(\lambda)}{c^+\lambda^2}S_2vG_2vD_2 & +c^+ h(\lambda) \lambda^2\Gamma_0 D_2 \nn \\
	&+\frac{\lambda^2}{c^+}{S_2vG_5vD_2}+\widetilde E_2(\lambda)D_2	\bigg]^{-1}\nn\\
	=\frac{-D_2}{c^+\lambda^2 h(\lambda)}+\frac{\tilde g_3(\lambda)}{(c^+)^2\lambda^4 h(\lambda)}D_2vG_4vD_2+&D_2\Gamma_0D_2  \\ &+ \frac{1}{(c^+)^2h(\lambda)}D_2vG_5vD_2+ E_3(\lambda) \nn
\end{align}
with $E_3(\lambda)=\widetilde O_1(\lambda^\ell)$ for $0<\ell<2$.  Since
$B_1^{-1}=(B_1+S_2)^{-1}+(B_1+S_2)^{-1}S_2B_2^{-1}S_2(B_1+S_2)^{-1}$ and
$B^{-1}=-h(\lambda)B_1^{-1}$.  Using \eqref{eqn:B1+S2inv} and \eqref{eqn:B1error1}, we arrive at
\begin{align}\label{eqn:B2inv long}
	B^{-1}(\lambda)&=\frac{D_2}{c^+\lambda^2} -\frac{\tilde g_3(\lambda)}{(c^+)^2\lambda^4}D_2vG_4vD_2 -h(\lambda)D_2\Gamma_0D_2- \frac{1}{(c^+)^2}D_2vG_5vD_2\nn \\
	&+E_1(\lambda)S_2\bigg[\frac{D_2}{c^+\lambda^2 }-\frac{\tilde g_3(\lambda)}{(c^+)^2\lambda^4}D_2vG_4vD_2-h(\lambda)D_2\Gamma_0D_2-\frac{1}{(c^+)^2}D_2vG_5vD_2 \bigg]\\
	&+\bigg[\frac{D_2}{c^+\lambda^2 }-\frac{\tilde g_3(\lambda)}{c^+\lambda^4}D_2vG_4vD_2-h(\lambda)D_2\Gamma_0D_2- \frac{1}{(c^+)^2}D_2vG_5vD_2 \bigg]S_2 E_1(\lambda) \nn \\ 
	&-h(\lambda)D_1-h(\lambda)E_1(\lambda) \nn
\end{align}
Since $E_1(\lambda)=D_1h(\lambda)E(\lambda)D_1+\widetilde O_1(\lambda^{2+\ell-})$ and
\begin{align} \label{Elambdaexp}
	E(\lambda)&=S_1 A^{-1}M_0A^{-1}(\lambda)S_1-S_1A^{-1}(\lambda)[M_0A^{-1}(\lambda)]^2[1+M_0A^{-1}  \nn \\
	& = \lambda^2c^{+} \bigg( \frac{S_1SvG_2vSS_1}{h^2(\lambda)} + \frac{S_1[SvG_2v + vG_2vS]S_1 } {h(\lambda)} + S_1vG_2vS_1 \bigg) + \widetilde O_1(\lambda^{4-})   \\
	&=\lambda^2\bigg(1+\frac{1}{h(\lambda)} \bigg)^2 c^+\Gamma_1 +\widetilde O_1(\lambda^{4-}) \nn
\end{align}
Combining this with \eqref{eqn:B2inv long} 
suffices to establish the statement.
\end{proof}

\begin{corollary}  Let $| V(x) | \les \la x \ra ^{-12-}$. If there is a resonance of the second kind then one has
\begin{align*}
	(M^{\pm}(\lambda))^{-1} = \frac{D_2 }{c^{\pm}\lambda^2} +\frac{\widetilde{g}^{\pm}_3(\lambda)}{(c^{\pm})^2 \lambda^4}  \Gamma_2 +  h_{\pm}(\lambda) \Gamma_3  + \Gamma^{\pm}_4+ h_{\pm}^{-1} (\lambda) \Gamma_5 + O_1(\lambda^{0+}) 
\end{align*}
where the absolutely bounded operators $\Gamma_2 $, $\Gamma_3 $ and $\Gamma_4 $  have either $S_1$ or $Q$ on both sides and $c^+=-\frac{ia_2\pi}{2}$,  $c^-=c^+-i\Im(z_2)$
\end{corollary}
\begin{proof} Recall by Lemma~\ref{JNlemma}, we have 
$$ M^{-1}=(M+S_1)^{-1}+(M+S_1)^{-1}S_1B^{-1}S_1(M+S_1)^{-1}.
	$$	
We use the expansion
$$
(M+S_1)^{-1} = A^{-1} + \lambda^2 A^{-1} vG_2 vA^{-1} +O_1(\lambda^{2+}).
$$
Note that, one has $A^{-1} S_2 = S_2A^{-1} = S_2$. Therefore, 
\begin{multline} \label{third eq}
(M+S_1)^{-1}S_1B^{-1}S_1(M+S_1)^{-1} \\
= S_1B^{-1} S_1 +  D_2 vG_2 vA^{-1} + A^{-1} vG_2 v D_2 + O_1(\lambda^{0+}).  
\end{multline}
 Since $A^{-1} = h^{-1}S + QD_0Q$ we further have 
\begin{align} \label{firsteq}
 D_2 vG_2 vA^{-1} = h_{\pm}^{-1}  D_2  vG_2v S +  D_2 vG_2v QD_0Q . 
\end{align}
Using  \eqref{third eq}, the expansion for $A^{-1}$, and \eqref{firsteq} in $ M^{-1}$, we obtain the statement.
\end{proof}

\begin{proof}[Proof of Proposition~\ref{prop:2nd kind}] First, we note that all but the first two terms in the expansion of $M_{\pm}^{-1} (\lambda)$ is controlled in previous sections by $\la t \ra ^{-1}$. Noting the orthogonality property $S_2v=vS_2=0$, the contribution of $\Gamma_3$, and $\Gamma_4$ can be estimated similarly to Proposition~\ref{prop:S_1T_1^{-1}S_1 bound} and Lemma~\ref{prop:QD0Q bound} respectively. Moreover, the contribution of $\Gamma_5$ is controlled by Lemma~\ref{prop:S bound}, while Lemma~\ref{lem:Error term} suffices to control the error term.

Furthermore, for the second term in $(M^{\pm}(\lambda))^{-1}$, by \eqref{alg fact} we need to consider the contribution of
\begin{multline*} 
	[R^{+}-R^{-}] v \frac{\widetilde{g}^{-}_3(\lambda)}{(c^{+})^2 \lambda^4}v  \Gamma_2 R^{+} + R^{-} v \Big[ \frac{\widetilde{g}^{+}_3(\lambda)}{(c^{+})^2 \lambda^4}-\frac{\widetilde{g}^{-}_3(\lambda)}{(c^{-})^2 \lambda^4} \Big] v \Gamma_2 R^{+} \\
	+ R^{-} v \frac{\widetilde{g}^{-}_3(\lambda)}{(c^{-})^2 \lambda^4}v [R^{+}-R^{-}]
\end{multline*}
Note that $\tilde g_3(\lambda)/\lambda^4=a_3\log \lambda+z_3$ is similar to the function $h(\lambda)$ encountered when there is a resonance of the first kind at zero. Since $\Gamma_2=S_2\Gamma_2 S_2$, for the first and third term, when the difference falls on the free resolvents, we can exchange $[R^{+}-R^{-}]$ with the auxiliary function $G$ as in the Corollary~\ref{corJ1}. Hence, the contribution of these two terms can be controlled by $\la t \ra^{-1}$  following the proof of Proposition~\ref{prop:S_1T_1^{-1}S_1 bound}. 

Now, we let 
$$
K:= R^{-} v \Big[ \frac{\widetilde{g}^{+}_3(\lambda)}{(c^{+})^2 \lambda^4}-\frac{\widetilde{g}^{-}_3(\lambda)}{(c^{-})^2\lambda^4} \Big] v \Gamma_2 R^{+} + \left[ \frac{R^{+} vS_2D_2S_2 R^{+}}{\lambda^2 c^{+}}- \frac{R^{-} vS_2D_2S_2 R^{-}}{\lambda^2 c^{-}} \right]
$$
and prove the following lemma which is sufficient to conclude  Proposition~\ref{prop:2nd kind}.
\end{proof}

\begin{lemma} \label{lem: S2D2S2 bound}Let $|V(x)| \les \la x \ra ^{-12-}$. Then one has 
\begin{align}
&  \int_{0}^{\infty} e^{-it\lambda^4} \lambda^3 \chi(\lambda) K(x,y) d \lambda = {O}(\la t\ra^{-1/2}), \nn\\
& \int_{0}^{\infty} e^{-it\lambda^4} \lambda^3 \chi(\lambda) K(x,y)d \lambda = F_t + O \Bigg( \frac{\la x \ra^{2+} \la y \ra^{2+} }{\la t\ra } \Bigg).\nn 
\end{align}
where $\|F_t\|_{1\to \infty} \les \la t\ra^{-\f12}$.
\end{lemma}

Before we start to prove Lemma~ \ref{lem: S2D2S2 bound}, we give the following oscillatory estimate, 

\begin{lemma}\label{lem:osc lame}
	
	If $\mathcal E(\lambda)=\widetilde O_1(\lambda^{-2})$ is supported on $0<\lambda \leq \lambda_1\ll 1$, then we have
	$$
		\bigg| \int_0^\infty e^{-it\lambda^4} \lambda^3 \mathcal E(\lambda)\, d\lambda \bigg| \les  \la t \ra^{-\f12}.
	$$
	\end{lemma}

\begin{proof}
	
	The boundedness is clear.  For the large time decay we break the domain of integration up into two pieces as in Lemma~\ref{lem:ibp osc}.  However, due to the singular behavior, we integrate by parts only once away from zero as follows
	\begin{align*}
		\bigg| \int_0^\infty e^{-it\lambda^4} \lambda^3 \mathcal E(\lambda)\, d\lambda \bigg|&\les \int_0^{t^{-\f14}} \lambda^3 |\mathcal E(\lambda)|\, d\lambda	+ \bigg| \int_{t^{-\f14}}^\infty e^{it\lambda^4} \lambda^3 \mathcal E(\lambda)\, d\lambda \bigg|
	\end{align*}
	The first integral is clearly bounded by $t^{-\f12}$.  For the second integral, we integrate by parts once to see
	\begin{align*}
		\bigg| \int_{t^{-\f14}}^\infty e^{-it\lambda^4} \lambda^3 \mathcal E(\lambda)\, d\lambda \bigg|&\les \frac{|\mathcal E(t^{-\f14})|}{t}+\frac{1}{t} \int_{t^{-\f14}}^\infty |\mathcal E'(\lambda)|\, d\lambda \les |t|^{-\f12}+\frac{1}{t} \int_{t^{-\f14}}^\infty \lambda^{-3}\, d\lambda \les |t|^{-\f12}.
	\end{align*}

\end{proof}

\begin{proof}[Proof of Lemma~\ref{lem: S2D2S2 bound}] We first prove the first assertion. To do that recall \eqref{eqn:H def}; the definition of $H(\lambda, x, x_1)$. As in the proof of Lemma~\ref{prop:QD0Q bound}, we use the orthogonality $S_2v=0$ and exchange $ R^{\pm}$ on both sides of $S_2D_2S_2$ and $\Gamma_2$ with $H^{\pm}$. By Lemma~\ref{cancellemma} and \eqref{eqn:A def}, we have 
$$
K(x,y)= \widetilde{O}_1(\lambda^{-2}) \| k(y,y_1)v(y_1)\|_{L^2_{y_1}} \| |D_2^\pm | + |\Gamma_2| \|_{L^2\to L^2} \| k(x,x_1)v(x_1)\|_{L^2_{x_1}}
$$
Therefore, Lemma~\ref{lem:osc lame} establishes the the time decay $\la t \ra^{-1/2}$. 

For the second assertion, note that using \eqref{alg fact}, we need to estimate 
\begin{multline}
	\frac{R^+(\lambda^4)vD^{+}_2vR^+(\lambda^4)}{\lambda^2}
	-\frac{R^-(\lambda^4)vD_2^-vR^-(\lambda^4)}{\lambda^2}	\\
	=R^-(\lambda^4)\frac{v[D_2^+-D_2^-]v}{\lambda^2} R^+(\lambda^4 ) + [R^+-R^-](\lambda^4) \frac{vD_2^+v}{\lambda^2}R^+(\lambda^4)
	\\ +R^-(\lambda^4)\frac{vD_2^-v}{\lambda^2} [R^+-R^-](\lambda^4):=I+II+III
\end{multline}
and 
$$
IV:= R^{+} v \Big[ \frac{\widetilde{g}^{+}_3(\lambda)}{(c^{+})^2 \lambda^4}-\frac{\widetilde{g}^{-}_3(\lambda)}{(c^{-})^2\lambda^4} \Big] v \Gamma_2 R^{-}
$$

From \eqref{eq:R0low}, and  \eqref{eqn:A def},
we have the expansions
\begin{align*}
& R^{\pm}(\lambda^4)(x,x_1)= \tilde g^{\pm}(\lambda) + G_1(x,x_1) + c^{\pm} \lambda^2 G_2(x,y) + \widetilde{O}_1((\lambda|x-x_1|)^{2+} ), \\ 
& [R^+-R^-](\lambda^4)(y_1,y) = \Im z_1 - \Im z_2 \lambda^2 G_2(y,y_1) + \widetilde{O}_1((\lambda |y-y_1|)^{2+} ).
\end{align*}
Therefore, since $S_2v=0$, we see 
\begin{align*}
&IV= c_1 \log \lambda G_1(x,x_1)v \Gamma_2 v G_1(y,y_1) +  c_2 G_1(x,x_1)v \Gamma_2 v G_1(y,y_1)  + \widetilde{O}_1(\lambda^{0+} \la x \ra^{2+} \la y \ra^{2+} ),\\
& I+II+ III  = \frac{G_1(x,x_1)v[D^{+}_2-D^{-}_2] v G_1(y,y_1)}{\lambda^2} \\ 
 & \hspace{1cm} + c_3 [ G_1(x,x_1)vD_2v G_2(y,y_1)+ G_2(x,x_1) vD_2v G_1(y,y_1)] + \widetilde{O}_1(\lambda^{0+} \la x \ra^{2+} \la y \ra^{2+} ).
\end{align*}
Note that  $ G_1(x,x_1)=-\frac{1}{8\pi^2} \log |x-y| $ and  $G_2(x,x_1) =c_2 |x-y|^2$. Hence, to obtain the error terms we use   $\log|x-y| \les \chi (|x-y|) |x-y|^{0-}+ \widetilde\chi (|x-y|) |x-y|^{0+} $ and \eqref{spatial bound}.  It is easy to see that  the contribution of the error terms in $IV$ and in $I+ II+ III$ can be bounded by $t^{-1} \la x \ra^{2+} \la y \ra^{2+} $ by a single integration by parts. For the contribution of the remaining terms, we define 
\be\label{eqn:Ft explicit}
	F_t:=\int_0^\infty e^{-it\lambda^4} \lambda^3 \chi(\lambda) G_1 v \Big[ \frac{D^{+}_2-D^{-}_2}{\lambda^2} +c_1\log \lambda \ \Gamma_2 \Big] v G_1  \, d\lambda
\ee
(with $D_2^{\pm} = D_2/c^{\pm}$,  $\Gamma_2:= D_2vG_4vD_2$), using Lemma~\ref{lem:osc lame} we see that
the $\lambda$ integral is bounded by $\la t\ra ^{-\f12}$.
 Therefore, it is enough to see that this operator is bounded $L^1\to L^\infty$.  For that,  we can use $S_2v=0$,  and replace $\log|x-x_1| $ with $\log|x-x_1|-\log |x| $.  Note that
$
	\big|\log|x-x_1|-\log |x| \big|\les 1+\log^-|x-x_1|+\log^+|x_1|=k(x,x_1),
$
and we have $ \| k(x,x_1)v(x_1)\|_{L^2_{x_1}} \les 1$. Therefore, the boundedness of the spatial integral follows by the absolutely boundedness of $D_2$ and $\Gamma_2$. 
The fact that $F_t$ is finite-rank follows from the fact that $D_2$ and $\Gamma_2$ are $ \lambda $ independent operators defined on the finite dimensional space $S_2L^2$. 

\end{proof}

\subsection{Resonances of the Third and Fourth Kind}
Finally, we analyze the decay properties of $e^{itH}P_{ac}(H)$  when there is a resonance of the third or fourth kind at zero. The reader will notice that the singularity of $(M(\lambda)+S_1)^{-1}$ at zero increases as we  iteratively use Lemma~\ref{JNlemma}. Accordingly, the time decay in the dispersive estimate is seen to be slower in the resonance of the third and fourth kinds.   Our main result in this section is the following.
  
\begin{prop}\label{prop:3rd kind} Let $|V(x) | \les \la x \ra^{-12-}$. In the case of a resonance of the third or fourth kind, we have
	$$
	\| e^{-itH}P_{ac}(H)\chi(H) \|_{L^1\to L^\infty} \les \min\left( 1, \frac{1}{\log t} \right) .
	$$

\end{prop}

As in the previous sections, we first need to find expansions for $B^{-1}(\lambda)$ in the third and fourth cases. We start with the resonance of the third kind, 

\begin{prop} Let $|V(x)| \les \la x \ra^{-12-}$. If there is a resonance of the third kind then we have 
\begin{align}\label{eqn:Binv 3rdkind}
	B_{\pm}(\lambda)^{-1}=  \frac{D_3}{\tilde g^{\pm}_3(\lambda)}-\frac{\lambda^4}{(\tilde g^{\pm}_3(\lambda))^2}D_3vG_5vD_3  + E_{0} (\lambda) 
\end{align}
where $E_{0} (\lambda)=  \widetilde O_1( \lambda^{-4} (\log\lambda)^{-3} )  $. 
\end{prop}

\begin{proof} 
	As before, we consider only the `+' case, the `-' case follows with only minor modifications.
	Recall the expansion \eqref{eqn:B2 simple} for $B_2(\lambda)$ in Proposition~\ref{prop: B-1 exp second kind}. We have 
\begin{multline*}
	B_2(\lambda)=-c^+\lambda^2 h(\lambda)\bigg[ S_2vG_2vS_2+\frac{\tilde g_3(\lambda)}{c^+\lambda^2} S_2vG_4vS_2+ c^+ h(\lambda) \lambda^2\Gamma_0 
	+\frac{\lambda^2}{c^+} S_2vG_5vS_2+ \widetilde E_2(\lambda)	\bigg],
\end{multline*}
with $\widetilde E_2(\lambda)=\widetilde O_1(\lambda^{2+\ell})$, $0<\ell<2$.  In this case, the operator $S_2vG_2vS_2$ is not invertible. Therefore, we let $ -c^+\lambda^2 h(\lambda) \widetilde{B}_2 (\lambda)= B_2(\lambda)$ and invert $\widetilde{B}_2 (\lambda) + S_3$, where $S_3$ is the Riesz projection onto the kernel of $S_2vG_2vS_2$. We obtain, 
\begin{align*}
	(\widetilde B_2(\lambda)+S_3)^{-1}:= D_2- \frac{\tilde g_3(\lambda)}{c^{+}\lambda^2}D_2vG_4vD_2 - c^+ h(\lambda) \lambda^2\Gamma_0 
	-\frac{\lambda^2}{c^+} S_2vG_5vS_2+ E_4(\lambda)
\end{align*}
with $E_4=\widetilde O_1(\lambda^{2+\ell})$.  Then, using Lemma~\ref{JNlemma}, we seek to invert $B_3=S_3-S_3(\widetilde B_2(\lambda)+S_3)^{-1}S_3$. Since $S_3D_2=D_2S_3=S_3$, and $S_3\Gamma_0=\Gamma_0 S_3=0$ we have
\begin{align}\label{eqn:B3 defn}
	B_3(\lambda)= \frac{\tilde g_3(\lambda)}{c^{+} \lambda^2}S_3vG_4vS_3+\frac{\lambda^2}{c^{+}} S_3vG_5vS_3+S_3E_4(\lambda)S_3.
\end{align}
Then, letting $D_3:=(S_3vG_4vS_3)^{-1}$,  we may write
\begin{align*}
	B_3(\lambda)^{-1}=\frac{ c^{+} \lambda^2}{\tilde g_3(\lambda)}D_3 \bigg[
	1+\frac{\lambda^4}{\tilde g_3(\lambda)}S_3vG_5vD_3+S_3\widetilde E_4(\lambda)D_3	\bigg]^{-1}&\\
	=\frac{c^{+}\lambda^2}{\tilde g_3(\lambda)}D_3-\frac{c^{+} \lambda^6}{(\tilde g_3(\lambda))^2}D_3vG_5vD_3+\frac{c^{+} \lambda^{10}}{(\tilde g_3(\lambda))^3} &D_3vG_5vD_3vG_5vD_3 \\ 
	&+\widetilde O_1(\lambda^{ -2}(\log \lambda)^{-4}).
\end{align*}
So, we obtain
\begin{multline*}
	B_2(\lambda)^{-1}=\frac{1}{c^+\lambda^2 h(\lambda)} \widetilde B_2(\lambda)^{-1}
	=-\frac{1}{ h(\lambda)\tilde g_3(\lambda)}D_3+\frac{\lambda^4}{ h(\lambda) (\tilde g_3(\lambda))^2}D_3vG_5vD_3\\
	+\frac{\lambda^8}{ h(\lambda)(\tilde g_3(\lambda))^3}D_3vG_5vD_3vG_5vD_3+\widetilde O_1(\lambda^{-4}(\log\lambda)^{-4}).
\end{multline*}
Using  $(B_1(\lambda)+S_2)^{-1}= D_1+\widetilde{O}_1(\lambda^{2-})$, see  \eqref{eqn:B1+S2inv}, and since $B(\lambda)=-h(\lambda)^{-1}B_1(\lambda)$, we obtain the statement.
\end{proof}

\begin{prop} Let $|V(x)| \les \la x \ra^{-12-}$. If there is a resonance of the fourth kind then we have 
\begin{align}\label{eqn:Binv 4rdkind}
	B_{\pm}(\lambda)^{-1}= \frac{D_4}{\lambda^4}
	+h^{\pm}_2(\lambda)S_3\Gamma S_3 +E_0^\pm (\lambda),
\end{align}
where $E_{0} (\lambda)=  \widetilde O_1( \lambda^{-4} (\log\lambda)^{-3} )  $,  $\Gamma$ is an absolutely bounded operator and
$$
	h_2^+(\lambda)-h_2^-(\lambda)=O_1\bigg(\frac{1}{\lambda^4 (\log \lambda)^2}\bigg)
$$
\end{prop}

\begin{proof}

In the case of a resonance of the fourth kind, we need to use Lemma~\ref{JNlemma} one more time since $S_3vG_4vS_3$ is not invertible, though $S_4vG_5vS_4$ is.  As usual we consider the `+' case. We let 
$$
 \frac{c^{+}\lambda^2}{\tilde g_3(\lambda)}B_3(\lambda)     =:\widetilde B_3(\lambda)=S_3vG_4vS_3+\frac{\lambda^4}{\tilde g_3(\lambda)} S_3vG_5vS_3+\widetilde E_4(\lambda)
$$
where $\widetilde E_4(\lambda)= O_1(\lambda^{\ell} (\log \lambda)^{-1})$ and invert $(\widetilde B_3(\lambda)+S_4)$, where $S_4$ is the Riesz projection onto the kernel of $S_3vG_4vS_3$. We obtain, 
\begin{align*}
	(\widetilde B_3(\lambda)+S_4)^{-1}=D_3-\frac{\lambda^4}{\tilde g_3(\lambda)}D_3vG_5vD_3+\frac{\lambda^8}{(\tilde g_3(\lambda))^2}(D_3vG_5v)^2D_3 \\ 
	- \frac{\lambda^{12}}{(\tilde g_3(\lambda))^3}(D_3vG_5v)^3D_3+  \widetilde O_1((\log \lambda)^{-4}).
\end{align*}
Next, we define  $B_4:=S_4-S_4(\widetilde B_3(\lambda)+S_4)^{-1}S_4$.  Using the above expansion we have
\begin{multline*}
	B_4(\lambda)=\frac{\lambda^4}{\tilde g_3(\lambda)}S_4vG_5vS_4-\frac{\lambda^8}{(\tilde g_3(\lambda))^2}S_4vG_5vD_3vG_5vS_4 \\ - \frac{\lambda^{12}}{(\tilde g_3(\lambda))^3}S_4(vG_5vD_3)^2vG_5vS_4+  \widetilde O_1((\log \lambda)^{-4}).
\end{multline*}

We note that by Lemma~\ref{invertible},  $S_4vG_5vS_4$ is invertible. Hence, we invert $B_4(\lambda)$, and obtain 
\begin{align*}
	B_4(\lambda)^{-1}&=\frac{\tilde g_3(\lambda)}{\lambda^4}D_4\bigg[1+\frac{\lambda^4}{\tilde g_3(\lambda)}\Gamma_3 +\frac{\lambda^{8}}{(\tilde g_3(\lambda))^2}\Gamma_4 +\widetilde O_1((\log \lambda)^{-3})	\bigg]^{-1}\\
	&=\frac{\tilde g_3(\lambda)}{\lambda^4}D_4+\Gamma_3 +\frac{\lambda^4}{\tilde g_3(\lambda)}\Gamma_4+\widetilde O_1((\log \lambda)^{-2})
\end{align*}

Hence, in case of there is a resonance of the fourth kind we find 
$$B_3^{-1}(\lambda) =-c^{+}\frac{D_4} {\lambda^2} + \frac{ c^{+} \lambda^2}{\tilde{g}_3(\lambda)} \Gamma_3 + \frac{c^{+} \lambda^6}{(\tilde{g}_3(\lambda))^2} \Gamma_4 + \widetilde O_1 \Big( \frac{1}{\lambda^2 (\log \lambda)^3 } \Big) $$
Noting that $(\widetilde B_2 + S_3)^{-1} = D_2 + O(\lambda^{2-} )$, we have 
$$
B_2^{-1}(\lambda) = -\frac{D_4}{\lambda^4 h(\lambda)} + \frac{ S_3 \Gamma_3 S_3}{ \tilde{g}_3(\lambda)h(\lambda)} + \frac{ S_3 \Gamma_4 S_3} { (\tilde{g}_3(\lambda))^2 h(\lambda)} +  \widetilde O_1 \Big( \frac{1}{\lambda^4 (\log\lambda)^4 } \Big)
$$
and recalling $ B_1(\lambda) = D_1 + \widetilde O_1 (\lambda^{2-}) $ we finally obtain 
  $$
 B^{-1}(\lambda) =-  h(\lambda) B^{-1}_1(\lambda) =  \frac{D_4}{ \lambda^4} + \frac{ S_3 \Gamma_3 S_3}{ \tilde{g}_3(\lambda)} + \frac{ S_3 \Gamma_4 S_3} { (\tilde{g}_3(\lambda))^2} +  \widetilde O_1 \Big( \frac{1}{\lambda^4 (\log \lambda)^3 } \Big)
 $$ 
this establishes the statement. Here, $ \Gamma_3 $ and  $ \Gamma_4 $ can be determined precisely, however for our purpose the fact that they are absolutely bounded is enough.
\end{proof}

Before we prove the statement of Proposition~\ref{prop:3rd kind}, we note the following oscillatory integral estimate, which is a Corollary of Lemma~\ref{lem:log osc}. 
\begin{lemma}\label{cor:log osc}
	
	If $\mathcal E(\lambda)=\widetilde O_1(\frac{1}{\lambda^4 \log^2 \lambda})$ is supported on $0<\lambda \leq \lambda_1\ll 1$, then we have
	$$
	\bigg| \int_0^\infty e^{-it\lambda^4} \lambda^3 \mathcal E(\lambda)\, d\lambda \bigg| \les \min \bigg( 1,\frac{1}{ \log t}\bigg).
	$$
	
\end{lemma}

We are now ready to prove Proposition~\ref{prop:3rd kind}. 
\begin{proof} [Proof of Proposition~\ref{prop:3rd kind}] We divide the proof two cases and first consider a resonance of the third kind.
Using the expansion $(M +S_1)^{-1} (\lambda) = QD_0Q + h^{-1} S +\widetilde O_1(\lambda^{0+})$,  and recalling $S_2 S= SS_2=0$, one can obtain the following expansion in the case of resonance of the third kind, 
\begin{align}
M_{\pm}^{-1}(\lambda) &= - \frac{D_3}{\tilde g^{\pm}_3(\lambda)}+\frac{\lambda^4}{(\tilde g^{\pm}_3(\lambda))^2}D_3vG_5vD_3 + QD_0Q E_{0} (\lambda)QD_0Q + \widetilde{O}_1(\lambda^{-4}(\log \lambda)^{-4}) \nn  \\&
=:  \frac{D_3}{\tilde g^{\pm}_3(\lambda)} + Q\tilde{E}^{\pm}_{0}(\lambda)Q + \tilde{E}^{\pm}_{1}(\lambda). 
\end{align}
where $\tilde{E}^{\pm}_{0}(\lambda)= \tilde O( \lambda^{-4}(\log \lambda)^{-2})$ and $\tilde{E}^{\pm}_{1}(\lambda)= \widetilde{O}_1(\lambda^{-4}(\log \lambda)^{-4})$

As usual, we need to estimate the contribution of $ [R^{+} v M_{+}^{-1} v R^{+} -R^{-} v M_{-}^{-1} v R^{-} ] $ to the Stone's formula. Notice that using the orthogonality property $Qv=0$,  we may exchange $R^{\pm} (\lambda^4) (x,x_1) $ with $H (\lambda, x,x_1)=\widetilde O_1(k(x,x_1))$, see \eqref{eqn:H def} in the proof of  Lemma~\ref{prop:QD0Q bound}. Therefore, we see that 
\begin{align*}
\left[ R^{+} v \frac{D_3}{\tilde g^{+}_3(\lambda)} v R^{+} -R^{-} v \frac{D_3}{\tilde g^{-}_3(\lambda)} v R^{-}  \right]  &= \left[ \frac{1}{g^{+}_3(\lambda)}
 -\frac{1}{g^{-}_3(\lambda)} \right] H^{-}vD_3vH^{+}  \\ & + \left[ \frac{R^{+} - R^{-}} {g^{+}_3(\lambda)} \right] vD_3vH^{+} + H^{-}vD_3v \left[ \frac{R^{+} - R^{-}} {g^{-}_3(\lambda)}  \right] 
\end{align*} 
 Note that, we further exchange $[R^{+} - R^{-}](\lambda^4)(x,x_1)$ with $G(\lambda, |x-x_1|, x) = \widetilde O_1( (\lambda \la x_1 \ra  )^{0+})$ and use Corollary~\ref{corJ1} to see 
$$\left[ R^{+} v \frac{D_3}{\tilde g^{+}_3(\lambda)} v R^{+} -R^{-} v \frac{D_3}{\tilde g^{-}_3(\lambda)} v R^{-}  \right] =\widetilde O_1 \big(  \lambda^{-4} (\log \lambda)^{-2} \big). $$
Hence, Lemma~\ref{cor:log osc} establishes the contribution of this term.
For the contribution of  $Q\tilde{E}^{\pm}_{0}(\lambda)Q + \tilde{E}^{\pm}_{1}(\lambda)$, we have 
$$
R^{\pm} v[Q\tilde{E}^{\pm}_{0}(\lambda)Q + \tilde{E}^{\pm}_{1}(\lambda)]vR^{\pm} = H^{\pm}v\tilde{E}^{\pm}_{0} vH^{\pm}
  + R^{\pm} v\tilde{E}^{\pm}_{1} v R^{\pm}.
$$
It is easy to see that the first summand is $\widetilde O_1 \big(  \lambda^{-4} (\log \lambda)^{-2} \big)$. Furthermore, writing $ R^{\pm} = \log^{-}(\lambda r) + A(\lambda,r)$, and noting $ \log^{-}(\lambda r) = \widetilde O (r^{0-} (\log \lambda) )$ and $A(\lambda,r)=\widetilde O_1(1)$ we see that $R^{\pm} v\tilde{E}^{\pm}_{1} v R^{\pm}= \widetilde O_1 \big(  \lambda^{-4} (\log \lambda)^{-2} \big)$. Therefore, Lemma~\ref{cor:log osc} establishes the claim in the case of a resonance of the third kind. 

Finally, we consider a resonance of the fourth kind.
As in the case of a resonance of the third kind, one can see that 
\begin{align} \label{M-1 fourth case}
M_{\pm}^{-1}(\lambda) = - \frac{D_4}{\lambda^4}+ h_2^{\pm}(\lambda) \Gamma + QD_0Q E_{0} (\lambda)QD_0Q + \widetilde{O}_1(\lambda^{-4}(\log \lambda)^{-4}).
\end{align}
Since $\Gamma$ is defined as an operator with the projection $S_3$ on both sides.  Then since $
	h_2^+(\lambda)-h_2^-(\lambda)=\widetilde{O}_1(\lambda^{-4}(\log \lambda)^{-2}) $
the contribution of all terms except the first one to the evolution can be estimated as in the previous case. Moreover, for the first term we have,  
\begin{align*}
 [R^{+} v D_4 v R^{+} -R^{-} v D_4 v R^{-} ] &= [R^{+}-R^{-}] v D_4 v R^{-} + R^{+} v D_4 v [R^{+}-R^{-}] \\
  &= G v D_4 v H^{-} + H^{+} v D_4 v G = \widetilde{O}_1(\lambda^{0+}) 
\end{align*}
Here, $G$ is as in the Corollary~\ref{corJ1}. Therefore, the contribution of the first term in \eqref{M-1 fourth case} is $t^{-0-}$.

\end{proof}

\begin{rmk}\label{rmk:eigenavalue}
	
	In the case of an eigenvalue only, when $S_1=S_2=S_3=S_4$, we expect to be able to obtain a better time decay.  Such results are known for the Schr\"odinger evolution in dimension $d\geq 3$, \cite{goldE,GG1,GG2,GT}.  As in the case of the Schr\"odinger evolution in two dimensions, \cite{eg2}, we expect the improved time decay to come at the cost of spatial weights.  In contrast to the Schr\"odinger evolution in which the natural time decay rate is achieved, we expect a decay rate of $\la t\ra^{-\f12}$ in this case.
	
\end{rmk}

\section{The Perturbed Evolution For Large Energy } \label{sec:large}
In this section, to complete the proofs of Theorems~\ref{thm:main} and \ref{thm:main1} we analyze the evolution of fourth order Schr\"odinger equation for  energies separated from zero, that is in the support of $\widetilde \chi (\lambda)$. We prove the following high energy results
\begin{prop} \label{prop:largemain}Let $|V(x)| \les \la x \ra^{-4-}$. We have  
\begin{align*}
 &\|e^{-itH}P_{ac}(H) \widetilde\chi(H) f \|_{L^{\infty}} \les |t| ^{-1} \| f \|_{L^{1}}  , \\
 & \| e^{-itH} P_{ac}(H) \widetilde\chi(H) f \|_{L^{\infty,-\f 12}} \les |t|^{-1-} \|  f \|_{L^{1,\f 12 }}, \,\,\ \text{for }t \geq 2.
\end{align*}  
\end{prop}
The main contrast for large energy is that we need the oscillation in the Stone's formula to ensure the integrals converge, hence the singular powers of $t$ as $t\to 0$ arise.

For large energy, we utilize the  resolvent identity and write 
\begin{align}\label{large symmetric}
R_V(\lambda^4)= R^\pm(\lambda^4) - R^\pm(\lambda^4)VR^\pm(\lambda^4) + R^\pm(\lambda^4)VR_V(\lambda^4) V R^\pm(\lambda^4).
\end{align}

We estimate the contribution of the all terms in \eqref{large symmetric} to the Stone' formula. We first  note that the contribution of the first term on the right hand side is already controlled in Lemma~\ref{lem:free est} by $|t|^{-1}$. Moreover, as noted after Lemma~\ref{lem:free est}, one can obtain the decay rate $t^{-\f98} \la x \ra^{\f12} \la y \ra ^{\f12}$ the contribution of this term in the large energy regime. 

The next lemma will control the contribution of the second term in \eqref{large symmetric}. Similar to the low energy, to equalize the decay on the potential in Proposition~\ref{prop:largemain}, we utilize Lemma~\ref{lem:ibpweight} in proving weighted bound. 

\begin{lemma} Let $|V(x)| \les \la x \ra^{-4-\alpha}$ for some $0<\alpha<1$. Then, we have
\begin{align*}
&\Big| \int_{0}^{\infty} e^{-it\lambda^4} \widetilde\chi(\lambda) \lambda^{3} [R^\pm VR^\pm](\lambda^4)(x,y) d \lambda \Big| \les |t|^{-1}, \\
&\Big| \int_{0}^{\infty} e^{-it\lambda^4} \widetilde\chi(\lambda) \lambda^{3} [R^\pm VR^\pm](\lambda^4)(x,y) d \lambda \Big| \les \frac{\la x \ra ^{\alpha} \la y \ra^{\alpha}} {t^{1+}}, \,\,\ t \geq 2.
\end{align*}
\end{lemma}

\begin{proof} Note that we have $R^{\pm}(\lambda^4)= \widetilde{O}_1( (\lambda r) ^{0-})$ and hence
$$ [R^\pm VR^\pm](\lambda^4)(x,y) = \widetilde{O}_1( \lambda ^{0-}) \int_{\R^4} \frac{V(x_1)}{ |x-x_1|^{0+} |y-x_1|^{0+} } dx_1= \widetilde{O}_1( \lambda  ^{0-}). $$
Now applying integration by parts, we obtain
\begin{multline*}
\Big| \int_{0}^{\infty} e^{-it\lambda^4} \lambda^{3} \widetilde\chi(\lambda)    [R^\pm  VR^\pm](\lambda^4)(x,y) d \lambda \Big|\\ 
 \les \frac{1}{t} \Big| \int_{\R^4} \int_{\lambda_1}^{\infty}  \frac{ V(x_1) }{ \lambda^{1+} |x-x_1|^{0+} |y-x_1|^{0+}}  d\lambda dx_1 \Big| \les t^{-1}.
\end{multline*}
Next, we prove the weighted bound. Here we need to write $ R^{\pm} (\lambda^4)= \rho(\lambda r) + A(\lambda, r) $ where $\rho(\lambda r)= \widetilde{O}( (\lambda r)^{0-}) $  and $A(\lambda, r)$ is as in \eqref{eqn:A def}. We obtain ($r_1=|x-x_1|$, $r_2=|y-x_1|$)
\begin{align}
	&|\partial_{\lambda} [ R ^{\pm}(x,x_1)R ^{\pm}(y,x_1)](\lambda^4)| \les  \lambda  ^{-1-} r_1^{0-} r_2^{0-}, \label{firsrderlarge}\\
	&|\partial^2_{\lambda} [ R ^{\pm}(x,x_1)R ^{\pm}(y,x_1)](\lambda^4)| \les  \lambda ^{-2-} r_1^{0-} r_2^{0-} \la r_1\ra^{\f12} \la r_2\ra ^{\f12}. \label{secondderlarge}
\end{align}
The growth in the spatial variables arises when the derivatives hit the phase in $A(\lambda, r)$.
For $b>\lambda>0$, and using \eqref{secondderlarge} in Mean Value Theorem we have 
\begin{align}\label{largemvt}
|\partial_{\lambda} [R ^{\pm}R ^{\pm}](b^4)-\partial_{\lambda} [R ^{\pm}R^{\pm}](\lambda^4)| \les (b-\lambda) \lambda^{-2-} r_1^{0-} r_2^{0-} \la r_1\ra^{\f12} \la r_2\ra ^{\f12}
\end{align}
Moreover, by  \eqref{firsrderlarge} we have 
\begin{align} \label{largetointerpolate}
|\partial_{\lambda} [R ^{\pm}R ^{\pm}](b^4)-\partial_{\lambda} [R ^{\pm}R^{\pm}](\lambda^4)| \les \lambda^{-1-} r_1^{0-} r_2^{0-} 
\end{align}

Interpolating between \eqref{largetointerpolate} and \eqref{largemvt}, taking  $b=\lambda \sqrt[4]{1+\pi\lambda^{-4} t^{-1}}$,  we have the following bound, 
\begin{align}
	& |\partial_{\lambda} [R ^{\pm} R ^{\pm}](b^4)-\partial_{\lambda} [R ^{\pm} R ^{\pm}](\lambda^4)|  
	\les t^{-\alpha} \lambda^{-1-4\alpha-}  r_1^{0-} r_2^{0-} \la r_1\ra^{\f\alpha2}  \la r_2\ra ^{\f\alpha2}  \label{largedif}
\end{align}
Here, recall that for large $t$, we have $\lambda \sqrt[4]{1+\pi \lambda^{-4}t^{-1}}- \lambda \approx (\lambda ^3 t)^{-1}$. 

We now apply Lemma~\ref{lem:ibpweight}  with $ \mathcal E(\lambda)= \widetilde{\chi}(\lambda) [R^\pm VR^\pm ](\lambda^4) $. Then, since $\lambda \gtrsim1$ in the support of  $\widetilde{\chi}(\lambda)$, by  \eqref{firsrderlarge} we have
\begin{align*}
	\Big| \int_0^\infty \frac{|\mathcal E'(\lambda)|}{1+\lambda^4 t} d\lambda \Big|\les \frac{1}{t^\alpha} \int_{0}^{\infty} \int_{\R^4}  \frac{  |V(x_1)| }{\la \lambda \ra^{1+\alpha+}  r_1^{0+} r_2^{0+}} dx_1 d \lambda \les t^{-\alpha}.
\end{align*}
Now, using \eqref{largedif},  we have 
\begin{multline*}
	\int_{t^{-\f14}}^\infty \bigg| \mathcal E'(\lambda \sqrt[4]{1+\pi t^{-1} \lambda^{-4}})-\mathcal E'(\lambda) \bigg|\, d\lambda \\  \les \frac{1}{t^{\alpha}} \int_{0}^\infty  \int_{\R^4}  \frac{  \la r_1\ra^{\f\alpha2}  |V(x_1)|\la r_2\ra ^{\f\alpha2}} {\la \lambda \ra^{1+4 \alpha+}  r_1^{0+} r_2^{0+}} dx_1 d \lambda 
	\les \frac{ \la x \ra^{\f \alpha 2} \la y \ra^{\f \alpha 2}}{t^{\alpha} }.
\end{multline*}

\end{proof}
Lastly, we prove 
\begin{prop}\label{prop:large} Let $|V(x)| \les \la x \ra^{-4-}$. We have 
\begin{align*} 
	&\Big| \int_{0}^{\infty} e^{-it\lambda^4} \widetilde\chi(\lambda) \lambda^{3} [R^\pm V R_V^\pm VR^\pm](\lambda^4)(x,y) d \lambda \Big| \les |t|^{-1}, \\
	& \Big| \int_{0}^{\infty} e^{-it\lambda^4} \widetilde\chi(\lambda) \lambda^{3} [R^\pm V R_V^\pm VR^\pm](\lambda^4)(x,y) d \lambda \Big| \les \frac{\la x \ra ^{\f 12 } \la y \ra^{\f 12 }} {t^{1+}}, \,\,\ t \geq 2.
\end{align*}
\end{prop}

In the proof of of Proposition~\ref{prop:large} we utilize the limiting absorption principle, the boundedness of the resolvent operators between weighted $L^2$ spaces.  Notice that using the expansion \eqref{RH_0 rep} and the limiting absorption principle for Schr\"odinger resolvent, see \cite{agmon}, one can see that for $\sigma>1/2$,
\begin{align*}
\| R(H_0, \lambda^4) \|_{L^{2,\sigma} \rightarrow L^{2,-\sigma}} \les \lambda^{-2} \| R_0 (\lambda^2) \|_{L^{2,\sigma} \rightarrow L^{2,-\sigma}} \les \lambda^{-3}.
\end{align*}
Moreover,  $\| \partial^{k}_{\lambda}R(H_0, \lambda^4) \|_{L^{2,\sigma} \rightarrow L^{2,-\sigma}} \les \lambda^{-3-k}$ for $\sigma > k+ 1/2$.  We notice that in general, the extension of this property to $R_V(\lambda)$, is not possible as in Schr\"odinger resolvent. This is because unlike the $-\Delta + V$, $H$ might possess embedded eigenvalues even for decaying potentials. For that purpose, we assume absence of embedded eigenvalues and use the following theorem.

\begin{theorem} \label{th:LAP}\cite[Theorem~2.23]{fsy} Let $|V(x)|\les \la x \ra ^{-k-1}$. Then for any $\sigma>k+1/2$, $\partial_z^k R_V(z) \in \mathcal{B}(L^{2,\sigma}(\R^d), L^{2,-\sigma}(\R^d))$ is continuous for $z \notin {0}\cup \Sigma$. Further, 
\begin{align*} 
\|\partial_z^k R(H_0;z) \|_{L^{2,\sigma}(\R^d) \rar L^{2,-\sigma}(\R^d)} \les z^{-(3+3k)/4}. 
\end{align*}
\end{theorem}

\begin{proof}[Proof of Proposition~\ref{prop:large}]
Recall one has  $R^{\pm}(\lambda^4)= \widetilde{O}_1( (\lambda r) ^{0-})$. Therefore, using an analysis as in \eqref{spat bound} together with Theorem~\ref{th:LAP}, we see
$$
\big| \partial^k_{\lambda} \{ \widetilde \chi(\lambda) [R^\pm VR_V^\pm VR^\pm ](\lambda^4)(x,y)  \} \big| \les  \widetilde \chi(\lambda) \la \lambda \ra^{-k-3-} , \,\,\  \text{for} \,\,\ k=0,1. 
$$
Hence, the first assertion follows by a single integration by parts.

To prove the weighted bound we recall $ R ^{\pm}(\lambda^4) =\rho(\lambda r) + A(\lambda,r)$. Theorem~\ref{th:LAP}   with the bounds \eqref{firsrderlarge} and \eqref{secondderlarge} gives 
\begin{align*}
\big| \partial^{k}_{\lambda} \{ \widetilde{\chi}(\lambda) [R^\pm VR_V^\pm VR^\pm ](\lambda^4) \} \big| \les \la \lambda \ra ^{-3-k-} \la x \ra^{1/2} \la y \ra^{1/2}\,\,\, k=0,1,2
\end{align*} 

We note that $|V(x)| \les \la x \ra^{-4-}$ is enough to establish the spatial bound above. This is because by Theorem~\ref{th:LAP}, the inner product $  \widetilde{\chi}(\lambda) \la  \partial^{k_1}_{\lambda} R^\pm (\cdot, x), \partial^{k_2}_{\lambda}[VR_V^\pm VR^\pm] (\cdot,y) \ra$ is meaningful, if one can assure $V$, as a multiplication operator, to map $L^{2,-k_1-{\f 12}}$ to $L^{2, k_2+{\f 12}}$ for any $k_1+k_2 \leq 2$. This holds  holds if  $|V(x)| \les \la x \ra^{-(k_1+k_2+2)-}$.

We now integrate by parts twice and have the bound
\begin{align*}
\frac{1}{t^2} \Big| \int_{0}^{\infty} \partial_{\lambda} \big\{ \widetilde\chi(\lambda) \lambda^{-3} \partial_{\lambda} \{[R^\pm V R_V^\pm VR^\pm](\lambda^4)(x,y)\}  \big\} d \lambda \Big| \les \frac{ \la x \ra^{1/2} \la y \ra^{1/2} }{t^2}.
\end{align*}
Due to support of $\widetilde \chi(\lambda)$, we do not encounter any boundary term in the integration by parts.

\end{proof}

\section{Classification of threshold spectral subspaces}\label{sec:classification}

\begin{lemma}\label{lem:esa1}
Assume $|v(x)| \les \la x \ra ^{-2-}$, if $\phi \in S_1 L^2(\R^4) \setminus \{0\} $, then $\phi= Uv \psi $  where $\psi \in L^{\infty}$, $H\psi=0$ in distributional sense, and 
\be \label{eq:psi def}
	\psi= c_0 - G_1v \phi,\,\, \text{where} \,\,\ c_0= \frac{1}{ \| V\|_{L^1}} \la v,T \phi \ra .
\ee
Moreover, if $|v(x)| \les \la x \ra ^{-3-}$, then  $G_1v \phi \in L^p$ for all $4<p \leq \infty$.
\end{lemma} 
\begin{proof} Assume $\phi \in S_1 L^2(\R^4)$, one has $Q(U+vG_1v)\phi=0$. Note that 
\begin{align*}
	Q(U+vG_1v)\phi & = (1-P)(U+vG_1v)\phi 
	 = U\phi + vG_1v\phi - PT\phi=0  \\
   &  \iff \phi = Uv ( -G_1v\phi + c_0)\,\, \text{where} \,\, c_0= \frac{1}{ \| V\|_{L^1}} \la v,T \phi \ra.
\end{align*}
To show $ [(-\Delta)^2+ V] ( -G_1v\phi + c_0)=0 $ first notice that $(-\Delta)^2 c_0=0$. Moreover, one has $ (-\Delta)^2 G_1v \phi = (-\Delta) G_0 v \phi = v \phi$ distributionally. Therefore,
$$ 
	[(-\Delta)^2+ V] ( -G_1v\phi + c_0)=v\psi+V(c_0-G_1v\phi) = v\phi - vUv\psi =0.
$$
Next we show that $G_1v \phi \in L^p$ for all $4<p \leq \infty$. First, we show that $G_1v \phi \in L^{\infty}$. First notice that we have $\la v , \phi \ra =0$ and hence, 
\begin{align*}
\int_{\R^4} \log|x-y| &v(y) \phi(y) dy \\
& = \int_{\R^4} \log^{-}|x-y| v(y) \phi(y) dy + \int_{\R^4} [\log^{+}|x-y|-\log^{+}|x|] v(y) \phi(y) dy \\
& \les \int_{|x-y| <1} \frac{|v(y) \phi(y)|}{|x-y|^{0+} } \phi(y) dy + \int_{\R^4} \la y \ra^{0+} |v(y) \phi(y)| dy < \infty.
\end{align*}

For the last equality notice that since $\log^{+}$ is an increasing function and $|x-y| \leq |x| (1+|y|) $ for $|x| \geq1 $, one has 
$$ \log^{+}  |x-y| \leq \log^{+}|x| + \log ^{+} \la y \ra .$$

Knowing the above estimate, to finish the proof it is enough to show $[G_1v \phi](x)  = O(|x|^{-1})$ when $|x|>10$. Using $\la v , \phi \ra =0$, we have 
\begin{align*}
\int_{\R^4} \log|x-y| v(y) \phi(y) dy = \int_{\R^4} \log \Big( \frac{|x-y|}{|x|} \Big) v(y) \phi(y) dy. 
\end{align*}
First assume that $|x| >2 |y| $. Notice that in this case $|x-y| < 2 |x|$, and hence
$$
\log \Big( \frac{|x-y|}{|x|} \Big) = \log \Big(1+ \frac{|x-y|-|x|}{|x|} \Big) = O \Big(\Big| \frac {y}{x} \Big|\Big).
$$
This gives, 
\begin{align*}
\Big| \int_{ |x| > 2 |y|} \log \Big( \frac{|x-y|}{|x|} \Big)v(y) \phi(y) dy\Big| \les \frac{1}{|x|} \int_{\R^4} \la y \ra ^{-2-} |\phi(y)| \les \frac{1}{|x|}. 
\end{align*} 
We note that this is the limiting factor for the decay, which is why we need $p>4$.

Second assume that  $|x| \leq 2 |y| $. In this case, one has $|y|/ |x| \gtrsim 1$ and 
\begin{multline} \label{Ox-1}
\Big| \int_{ |x| \leq 2|y|}  \log \Big( \frac{|x-y|}{|x|} \Big) v(y) \phi(y) dy\Big| \\
  \les   \int_{ |x| \leq |y|/2} [ \la y \ra^{0+} +  \log^{-} |x-y| ]   \la y \ra ^{-2-} |\phi(y)| dy \les \frac{1}{|x|} .
\end{multline} 
 
\end{proof}
\begin{rmk} \label{rem:loggood}Notice that the integral \eqref{Ox-1} can have faster decay for large $|x|$, provided that $|v(x)|$ has faster decay at infinity. In particular, for any $p\geq 1$ if $ \beta > p+2 $ , then one has the improved estimate
\begin{multline}
\Big| \int_{ |x| \leq 2|y|}  \log \Big( \frac{|x-y|}{|x|} \Big) v(y) \phi(y) dy \Big|  \les \int_{\R^4}  [\la y \ra^{0+} + \log^{-}|x-y|] |v(y) \phi(y)| dy \\  \les \frac{1}{|x|^p} \int_{\R^4} [\la y \ra^{0+} + \log^{-}|x-y|] \la y \ra ^{p-\beta-} |\phi(y)| dy \les \frac{1}{|x|^p} .
\end{multline} 

Recalling also $\psi \in L^{\infty}$, in the domain $B_{a}=\{(x,y): |x| \leq a |y| \}$ for $a  \gtrsim 1$ one has $\psi-c_0 \in L^{2}$ provided $\beta> 4$.

\end{rmk}

Define $S_2$ be the projection on the kernel of $ S_1TPTS_1$ then we have 

\begin{lemma} Let $|v(x)| \les \la x \ra ^{-3-}$. Then $\phi = Uv \psi \in S_2L^2 $ if and only if $\psi \in L^p $, for all $4<p \leq \infty$.
\end{lemma}

\begin{proof} It is enough to show that $c_0 =0$ in \eqref{eq:psi def} if and only if $\phi \in S_2 L^2 $. 
Taking $\phi \in S_2L^2$, we have
$$
S_1TPTS_1 \phi= 0 \Rightarrow 0= \la TPT \phi, \phi \ra = \| PT \phi \|^2=0. 
$$
This gives $c_0 =0 $ and $\psi \in L^p$ for $4<p \leq \infty$ by Lemma~\ref{lem:esa1}.

On the other hand, if $\psi \in L^p$ for all $p>4$, using \eqref{eq:psi def} then one must have $c_0 =0 $ and hence $PT\phi=0$. This gives $\phi \in S_2 L^2$. 
\end{proof} 

Let $S_3$ be the projection on the kernel of $ S_2 v G_2 v S_2 $  where $G_2(x,y) = |x-y|^2  $.

\begin{lemma} Let $|v(x)| \les \la x \ra ^{-3-}$. Assume that the function $ \psi= c +\Lambda= c+ \Lambda_1 + \Lambda_2+\Lambda_3$, with $\Lambda_1 \in L^{p_1}$ for some $ 4<p_1<\infty$, $\Lambda_2 \in L^{p_2}$ for some $ 2 < p_2 \leq 4$, $\Lambda_3 \in L^2$, solves $H\psi=0$ in the sense of distributions. Then $\phi= Uv\psi \in S_1L^2$, and we have $\psi = c - G_1 v \phi$, 
$c=\frac{1}{ \| V\|_{L^1}} \la v,T \phi \ra$. In particular, by the previous claim, $\psi -c \in L^p$ for any $p>4$.
\end{lemma}
\begin{proof}
Let $\psi$ be in the form that is described with $H\psi=0$, or equivalently $-(-\Delta)^2 \psi = V \psi$. We first show that for $\phi= Uv\psi$ , one has 
$$
\int_{\R^4} v(x) \phi(x) dx =0
$$
 Let $\eta(x)$ be a smooth cutoff function with $\eta(x)=1$ for all $|x| \leq 1$, and take any $\delta>0$,
\begin{align*}
\la v \phi(\cdot) , \eta(\delta\cdot) \ra= \la  V\psi (\cdot),  \eta(\delta \cdot) \ra =- \la  (-\Delta)^2 \psi(\cdot),  \eta\eta(\delta \cdot) \ra =  - \delta^4 \la \Lambda (\cdot), [(-\Delta)^2 \eta] ( \delta \cdot)  \ra
\end{align*}
Where we used that $(-\Delta)^2\psi=(-\Delta)^2 \Lambda$.
Therefore, with $\eta_2=(-\Delta)^2\eta$,
\begin{multline*}
\big|\la v \phi , \eta \ra \big| \les \delta^{4-4/p_1} \|\Lambda_1\|_{L_{p_1}} \| \eta_2\|_{L_{p'_1}} \\ 
  +\delta^{4-4/p_2} \|\Lambda_1\|_{L_{p_2}} \| \eta_2 \|_{L_{p'_2}} +\delta^2 \|\Lambda_1\|_{L_2} \|\eta_2\|_{L_2}  \to 0 , \,\, as \,\,  \delta \to 0
  \end{multline*}
Hence by dominated convergence theorem we conclude $\la v ,\phi \ra=0$.

Moreover, let $ \tilde{\psi} = \psi+ G_1v\phi$, then by assumption $ \tilde{\psi}$ is bounded and $(-\Delta)^2 \tilde{\psi}=0$. By Liouville's theorem on $\R^n$, $ \tilde{\psi} =c$. Hence, $\psi= c- G_1v\phi$. Since we have,
$$
H \psi=[ (-\Delta)^2 +V] \psi = Vc+Uv(U+v G_1v)\phi \Rightarrow vc= (U+ v G_1v)\phi 
$$
one has $c= \frac{1}{ \| V\|_{L^1}} \la v,T \phi \ra$.
 Lastly notice that, 
$$ Q(U+vG_1v) Q \phi = Q(U+vG_1v) \phi = Q (U \phi + vG_1v \phi )= Q (U \phi -v \psi +cv) = 0, $$
hence $\phi \in S_1L^2$ as claimed.
\end{proof}
\begin{lemma}\label{lem:S_3} Let $|v(x)| \les \la x \ra ^{- 4-}$.  Then $\phi = Uv \psi \in S_3L^2 $ if and only if $\psi \in L^p $, for all $2<p \leq \infty$.

\end{lemma}

\begin{proof} First we show that if $\phi \in S_3L^2 $ then $\psi \in L^p $, for all $2<p \leq \infty$.  Let, $\phi \in S_3L^2 $ then 
$$
S_2 v G_2  v \phi=0 \Rightarrow \int_{\R^8} \phi(x) v(x) [|x|^2- 2 x\cdot y + |y|^2] v(y) \phi(y) dx dy=0 .
 $$
Using $S_3\leq Q$, we have $\la v, \phi \ra =0$, and
\begin{align*}
\int_{\R^8} \phi(x) v(x) [|x|^2- 2 x \cdot y + |y|^2] v(y) \phi(y) dx dy= -2 \Big[\int_{\R^4} y v(y) \phi(y) dy \Big]^2 
\end{align*} 
which gives 
$$\int_{\R^4} y v(y) \phi(y) dy=0 .$$
Using this, \eqref{eq:psi def} and noting $c_0 =0$, one has 
\begin{align*} 
\psi(x) = c \int_{\R^4} \Bigg( \log \Big( \frac{|x-y|^2}{|x|^2} \Big) + 2 \frac{ x \cdot y}{|x|^2} \Bigg) v(y) \phi(y) dy
\end{align*}
We show that $\psi(x)= O(|x|^{-2})$ for $|x| >10$.
As in the proof of Lemma~\ref{lem:esa1}, first assume $|x| > 4 |y| $. Then we have $|y|^2 + 2|x \cdot y| < |x|^2$ and 
\begin{align}\label{log ext}
\log \Big( \frac{|x-y|^2}{|x|^2} \Big) + 2 \frac{ x \cdot y}{|x|^2} 
  & = \log \Big(1+ \frac{|x-y|^2-|x|^2}{|x|^2} \Big)+ 2 \frac{ x \cdot y}{|x|^2}  \\ \nn
   &= \frac{|y|^2}{|x|^2}+ O \Bigg( \frac {|y|^{2+}}{|x|^{2+}} \Bigg). 
\end{align}
Therefore,
\begin{align} \label{logbad}
\int_{ |x| > 4 |y|} \log \Big( \frac{|x-y|}{|x|} \Big)v(y) \phi(y) dy \les \frac{1}{|x|^2} \int_{\R^4} \la y \ra ^{-2-} |\phi(y)| \les \frac{1}{|x|^2}. 
\end{align} 
Second, if $|x| \leq 4 |y| $ one has $|x \cdot y| \les |y|^2$ and 
\begin{align} \label{boundxy}
\int_{ |x| \leq 4 |y|} 2 \frac{ x \cdot y}{|x|^2} v(y) \phi(y) \les \frac{1}{|x|^2}  \int_{\R^4} |y|^2 |v(y) \phi(y)| dy \les \frac{1}{|x|^2}. 
\end{align}
Remark~\ref{rem:loggood} will take care of the logarithmic term in \eqref{log ext}.

For the converse, we assume $\psi \in L^p $, for all $p>2$ and $\phi = Uv \psi$, and show 
\begin{align} \label{xperp}
 \int_{\R^4} y v(y) \phi(y) dy=0 .
\end{align}
Note that if  $\psi \in L^{2+}$ then, from \eqref{eq:psi def}, $c_0=0$ and
\begin{multline*}
 \int_{ \R^4 }  \log \Big( \frac{|x-y|^2}{|x|^2} \Big) v(y) \phi(y) dy \\ 
 = \int_{ |x| > 4 |y| }  \log \Big( \frac{|x-y|^2}{|x|^2} \Big) v(y) \phi(y) dy  + \int_{ |x| \leq 4 |y| }  \log \Big( \frac{|x-y|^2}{|x|^2} \Big) v(y) \phi(y) dy\in L^{2+}
\end{multline*} 
Note that, by Remark~\ref{rem:loggood} the second term is in $L^{2}\cap L^\infty$. Furthermore, using \eqref{log ext} and \eqref{logbad}, we have
$$
\int_{ |x| > 4 |y| }  \log \Big( \frac{|x-y|^2}{|x|^2} \Big) v(y) \phi(y) dy =- \int_{ |x| > 4 |y| } 2 \frac{ x \cdot y}{|x|^2}  v(y) \phi(y) dy + O_{L^{2+}} (1)
$$
Therefore, using also  \eqref{boundxy}, and noting that 
$$
- \int_{ |x| < 4 |y| } 2 \frac{ x \cdot y}{|x|^2}  v(y) \phi(y) dy	\les \frac{1}{\la x\ra^{2+}} \in L^{2+},
$$
one has 
$$
\psi= -2  \frac{ x}{\la x \ra ^2} \int_{\R^4} y v(y) \phi(y) dy + O_{L^{2+}} (1)
$$
and hence \eqref{xperp} holds and $\phi \in S_3L^2$.

\end{proof}

Define $S_4$ the projection on to the kernel of $S_3vG_4 vS_3$. 

\begin{lemma}  \label{invertible}Let $|v(x)| \les \la x \ra ^{- 4-}$. Then the kernel of the operator $S_4 v G_5 v S_4$ on $S_4L^2$ is trivial. 
\end{lemma} 
\begin{proof} Take $f$ to be in the kernel of $S_4 v G_5 v S_4$ on $S_4L^2$ and recall the expansion \eqref{eq:R0low};
$$
R^{+} ( \lambda^4)= \tilde g_1^{+}(\lambda)+G_1(x,y)
+\alpha^{+}_1\lambda^2 G_2(x,y)
+\tilde g^{+}_3(\lambda) G_4(x,y) + \lambda^{4} G_5 (x,y)+ \widetilde O_2((\lambda |x-y|)^{6-}).
$$ 
Notice that since $f \in S_4L^2$ one has $0 = \la v,f \ra =\la G_4 v f, vf \ra $, and therefore 
\begin{align} \label{G5toG0}
0&= \la S_4 v G_5 v f, f \ra = \la G_5 v f, vf \ra \\
&=  \lim_{\lambda \rar 0}  \Big\la \frac{R^{+} ( \lambda^4)- \tilde{g}_1^{+}(\lambda) - G_1 - c^+ \lambda^2 G_2- \tilde{g}_3^{+}(\lambda)G_4}{\lambda^4} vf, vf \Big\ra 
 = \lim_{\lambda \rar 0} \Big\la \frac{R^{+} ( \lambda^4)- G_1}{\lambda^4} vf, vf \Big\ra \nn.
\end{align}
Further, one has 
\begin{align} \label{G5toG0lim}
	\lim_{\lambda \rar 0^+} \Big\la \frac{R^{+} ( \lambda^4)- G_1}{\lambda^4} vf, vf \Big\ra= \lim_{\lambda \to 0^+} \frac{1}{\lambda^4} \Big\la \Big( \frac{1}{ 8 \pi^2 \xi^2 + \lambda^4} - \frac{1}{ 8 \pi^2 \xi^2} \Big) \widehat{vf}(\xi), \widehat{vf}(\xi) \Big\ra \\ 
	= \lim_{\lambda \to 0^+} \int_{\R^4} \frac{-1}{(8 \pi^2 \xi^2 + \lambda^4) 8 \pi^2 \xi^2} |\widehat{vf}(\xi)| d \xi = \frac{-1}{ (4 \pi)^4} \int_{\R^4} \frac{|\widehat{vf}(\xi)| }{ \xi^4} d \xi =0 \nn. 
\end{align}
Note that this gives $vf=0$ since $vf \in L^1$ and hence $f=0$. This establishes the invertibility of  $S_4 v G_5 v S_4$ on $S_4L^2$.
\end{proof}
\begin{rmk} \label{G5toG0inuse}
Notice that, \eqref{G5toG0} and \eqref{G5toG0lim} imply that for any $\phi \in S_4$ one has 
\begin{align*}
\la S_4 v G_5 v \phi, \phi \ra = \frac{1}{ (4 \pi)^4} \int_{\R^4} \la \frac{|\widehat{v\phi}(\xi)| }{ \xi^2},  \frac{|\widehat{v \phi}(\xi)| }{ \xi^2} \ra= \la G_1v \phi, G_1v \phi \ra
\end{align*}
provided $|v(x)| \les \la x \ra ^{- 4-}$.
\end{rmk}

\begin{lemma}Let $|v(x)| \les \la x \ra ^{- 4-}$, $\phi = Uv \psi \in S_4L^2 $ if and only if $\psi \in L^2 $. 
\end{lemma}

\begin{proof}
Assume for now that $\phi \in S_3L^2$ and  
\begin{align} \label{x^2perp}
\int_{\R^4} |y|^2 v(y) \phi(y) =0
\end{align}
then we have 
\begin{align*} 
\psi(x) = c \int_{\R^4} \Bigg( \log \Big( \frac{|x-y|^2}{|x|^2} \Big) + 2 \frac{ x \cdot y}{|x|^2}- \frac{ |y|^2}{|x|^2}  \Bigg) v(y) \phi(y) dy
\end{align*}
Note that $\psi \in L^2$, by  \eqref{log ext} in the domain $|x|>4|y|$  and by Remark~\ref{rem:loggood} in the domain $|x|<4|y|$ . Therefore, to prove the only if part of the statement it is enough to show that \eqref{x^2perp} holds if $\phi \in S_4L^2 $. 

First recall the definition of $G_5$  and notice that
\begin{align}
\label{exp4-1} |x-y|^4= |x|^4 + |y|^4 - 4x \cdot y |y|^2 - 4 y \cdot x |x|^2 + 2 |x|^2 |y|^2 + 4 ( x \cdot y )^2.  
\end{align}
Next recall  that $\phi \in S_4L^2 \leq S_3L^2 \leq S_2 L^2$. Hence, using the expansion \eqref{exp4-1} in $S_4vG_5vS_4$ we see all but the final two terms contribute zero.
For the contribution of $G_{5,1}(x,y):=2 |x|^2 |y|^2$, we note
\begin{multline*}
	\la S_4 v G_{5,1}v \phi, \phi \ra
	=2\int_{\R^4} \int_{\R^4} v(x)v(y)|x|^2|y|^2 \phi(x) \overline{\phi(y)} \, dx\, dy\\
	=2\int_{\R^4} |y|^2 v(y)\overline{\phi(y)} \int_{\R^4} v(x) |x|^2 \phi(x)  \, dx\, dy=2 \bigg|\int |y|^2 v(y) \phi(y)\, dy \bigg|^2.
\end{multline*}
Here, we note that $v$ and $|y|$ are all real-valued, while $\phi$ is complex valued.  Since $y_j$ is real, a similiar argument applies and
we obtain
$$
	0= \la S_4 v G_5v \phi, \phi \ra =2 \bigg|\int |y|^2 v(y) \phi(y)\, dy \bigg|^2
+ 4 \sum_{i,j=1}^4 \bigg| \int y_i y_j v(y) \phi(y)\, dy \bigg| ^2.
$$
Since, both quantities are non-negative,   they both must be zero.    Hence,
$$
\int_{\R^4} y_j y_i v(y) \phi(y) =0, \,\,\ 1\leq i,j \leq 4,
$$
and hence \eqref{x^2perp} holds. For the reverse implication, notice that we need to show 
\begin{multline*}
 \int_{ \R^4 }  \log \Big( \frac{|x-y|^4}{|x|^4} \Big) v(y) \phi(y) dy \\ 
 = \int_{ |x| > 32 |y| }  \log \Big( \frac{|x-y|^4}{|x|^4} \Big) v(y) \phi(y) dy  + \int_{ |x| \leq 32 |y| }  \log \Big( \frac{|x-y|^4}{|x|^4} \Big) v(y) \phi(y) dy\in L^{2}
\end{multline*} 
By Remark~\ref{rem:loggood} the second integral is in $L^2$, Therefore, the first integral should also be in $L^2$. Using \eqref{exp4-1}  in the domain of $|x| > 32 |y|$  we have
\begin{align*}
  \log \Big( \frac{|x-y|^4}{|x|^4} \Big) =  - 4 \frac{y\cdot x }{|x|^2} + 2 \frac{|y|^2}{|x|^2} + 4 \frac{ (x \cdot y)^2}{|x|^4} +  O \Bigg( \frac {|y|^{2+}}{|x|^{2+}} \Bigg)
\end{align*}
Therefore, one has 
\begin{multline} \label{psi4}
\psi(x)= -4  \frac{ x}{|x|^2} \int_{\R^4} y v(y) \phi(y) dy + 6  \frac{ 1}{|x|^2} \int_{\R^4} |y|^2 v(y) \phi(y) dy \\ 
+ 2\sum_{\substack{i,j=1 \\ i>j}}^4 \frac{ x_i x_j}{|x|^4} \int_{\R^4} y_i y_j   v(y) \phi(y) dy + O_{L^2}(1) 
\end{multline}
Notice that to conclude \eqref{psi4}, one has to make sure  the integrals in \eqref{psi4} are in $L^2$ in the domain $|x| \leq 32 |y|$. But this is true since if $|x| \leq 32 |y|$ then $|x \cdot y | \les |y|^2 \les |y|^{2+}/ |x|^{0+} $. 

By Lemma~\ref{lem:S_3}, the first integral is zero and hence $\psi \in L^2$ if  
$$ \int_{\R^4} y_j y_i v(y) \phi(y) =0, \,\,\ 1\leq i,j \leq 4 $$
which corresponds $\phi \in S_4 L^2$. 
\end{proof}

\begin{lemma} The operator $G_1 v S_4 [ S_4 v G_5 v S_4]^{-1} S_4 v G_1$  is the orthogonal projection on $L^2$ onto the zero energy eigenspace of $H = (-\Delta)^2 + V$. 
\end{lemma}

\begin{proof}
Let $ \{ \phi_k \}_{k=1}^N $ be the orthonormal basis of $S_4L^2$, then $S_4 f = \sum_{j=1}^N \phi_j \la f,  \phi_j \ra $. Moreover, for all $\phi_k$, one has $ \psi_k = - G_1 v \phi_k $ are linearly independent for each $k$ and $\psi_k \in L^2$. We will show that $P_e \psi_k := G_1 v S_4 [ S_4 v G_5 v S_4]^{-1} S_4 v G_1 \psi_k = \psi_k$ for all $1\leq k \leq N$. 

First notice that, by the representation of $S_4$, one has 
$$
S_4 v G_1 \psi_k = \sum^N_{j=1} \phi_j \la v G_1 \psi_k,  \phi_j \ra = \sum^N_{j=1} \phi_j \la \psi_k,  \psi_j \ra =: \sum^N_{j=1} \phi_j a_{k,j}
$$
Let $\{A_{ij}\}_{i,j=1}^N$ be the matrix that represents the kernel of $S_4 v G_5 v S_4$, then by Remark~\ref{G5toG0inuse}
$$
A_{ij}(x,y) =  \la S_4 v G_5 v \phi_i, \phi_j \ra \phi_i(x) \phi_j(y)= \la  G_1 v \phi_i,  G_1 v \phi_j \ra \phi_i(x) \phi_j(y)= a_{i,j} \phi_i(x) \phi_j(y).
$$
Hence, one has 
\begin{align*}
P_e \psi_k &= - \sum^N_{j=1} G_1 v S_4 [ S_4 v G_5 v S_4]^{-1} \phi_j  a_{k,j} \\ &
 = \sum^N_{i,j=1} G_1 v S_4 (a^{-1})_{i,j}  \phi_i a_{k,j} 
=  \sum^N_{i,j=1} \psi_i (a^{-1})_{i,j}  a_{j,k}= \psi_k
\end{align*}

This finishes the proof.

\end{proof}

\begin{rmk}
	
	One consequence of the preceding results is that any zero-energy resonance function is of the form:
	$$
		\psi(x)= c_0+ c_1 \frac{x_1}{\la x \ra ^{2}}+ c_2 \frac{x_2}{\la x \ra ^{2}} +c_3 \frac{x_3}{\la x \ra ^{2}}  + c_3 \frac{x_4}{\la x \ra ^{2}} + \sum_{i,j=1}^4 c_{ij} \frac{x_i x_j}{\la x \ra^{4}} + O_{L^2}(1)
	$$
	For some constants $c_0,c_1, c_2,c_3,c_4$ and $c_{ij}$, $1\leq i,j,\leq 4$.
	Hence, the resonance space is 15 dimensional along with the finite-dimensional eigenspace. Moreover, $S_1-S_2 $ is one dimensional, $S_2-S_3$ is four dimensional, $S_3-S_4$ is 10 dimensional. 
	
\end{rmk}

\end{document}